\documentclass[a4paper]{amsart}
\usepackage{a4}
\usepackage{amsmath}
\usepackage{amssymb}
\usepackage[applemac]{inputenc}
\usepackage{graphicx}
\usepackage{bbm}
\usepackage{mathtools}
\usepackage{color}
\usepackage{ulem}
 \newtheorem{theo}{Theorem}[section]
 \newtheorem{lem}{Lemma}[section]
 \newtheorem{prop}{Proposition}[section]
  
  \newtheorem{cor}{Corollary}[section]
 
 \newtheorem{rem}{Remark}[section]
\newcommand{\be}{\begin{eqnarray}}
\newcommand{\ee}{  \end{eqnarray}}
\newcommand{\beno}{\begin{eqnarray*}}
\newcommand{\eeno}{  \end{eqnarray*}}

\newcommand\R{{\mathbb R}}

\newcommand{\ymax}{\bar y}

\newcommand{\Uad}{\mathcal U^{t,y}_{\text{ad}}}

\begin{document} 

\title[Optimal management of pumped-storage hydroelectric production] {Optimal management of pumped hydroelectric production with state constrained optimal control} 

\author{Athena Picarelli}\thanks{Universit\`a di Verona, Dipartimento di Scienze Economiche, Via Cantarane 37129, Verona (VR), Italy }
\author{Tiziano Vargiolu}\thanks{Universit\`a degli Studi di Padova, Dipartimento di Matematica, Via Trieste 63, 35121, Padova (PD), Italy} 

\date{\today} 

\begin{abstract}
We present a novel technique to solve the problem of managing optimally a pumped hydroelectric storage system. This technique relies on representing the system as a stochastic optimal control problem with state constraints, these latter corresponding to the finite volume of the reservoirs. Following the recent level-set approach  presented in \textit{O. Bokanowski, A. Picarelli, H. Zidani, State-constrained stochastic optimal control problems via reachability approach, SIAM J. Control and Optim. 54 (5) (2016)}, we transform the original constrained problem in an auxiliary unconstrained one in augmented state and control spaces, obtained by introducing an exact penalization of  the original state constraints. The latter problem is fully treatable by classical dynamic programming arguments. 
%We finally present a numerical implementation of a two-basins pumped storage system. 
\end{abstract}

\maketitle 

\section{introduction}

In the current transition to a low carbon economy, one of the most prominent issues is that the most used renewable energy sources (RES), i.e. photovoltaic and wind, are non-dispatchable: for this reason, there is a strong need to store the energy produced when they are available in order to use it when they are not. 

 As for today, the most cost-efficient method to store electricity is with pumped hydroelectric storage (PHS) systems, which in 2017 accounted for about 95\% of all active tracked storage installations worldwide, with a total installed capacity of over 184 GW \cite{DOE}. In brief, these systems consist of two or more dam-based hydroelectric plants, linked sequentially so that, while the water used to produce electricity in the lower reservoirs is lost as in traditional hydroelectric plants, upper reservoirs discharge their water in lower reservoirs, where it can be stored and possibly pumped back to upper reservoirs when the electricity price is low (typically in off-peak periods). In this way, a PHS can produce electricity in peak periods, i.e. when its price is high, and recharge the upper reservoir in off-peak periods. These systems can be efficiently coupled with non-dispatchable renewable energy sources (like photovoltaic or wind), thus providing a very effective way to store possible surplus renewable energy when the price is low, and to use it when needed. An alternative is of course to connect the system directly to the power grid: also in this case, PHS is likely to use renewable-generated electricity to pump back water in the upper reservoirs, given the very low marginal generating costs of renewable energy. 

In the current literature, the mathematical treatment of PHS is usually done by computing the fair value as the sum of discounted payoffs when operated at optimum\footnote{ In this setting, we assume that the investor is risk-neutral. Although here we are not evaluating financial assets, but rather incomes coming from industrial activity, this is in line with all the related literature (see e.g. \cite{McDonaldSiegel}), and is justified by the following financial argument. Even if the underlying assets are in principle not traded, a risk-neutral evaluation can be applied as long as one can find hedging instruments that can be storable and liquidly traded: see \cite{TDR} for more economic insight, and \cite[Remark 3.6]{CCGV} for the mathematical derivation of such a result for structured products like power plants.}. To perform this computation, one can find in literature two alternative approaches. The first one is via operations research techniques by formulating a linear/quadratic programming model, see e.g. \cite{BLM,LWM,VMBI}. While these techniques allow to model potentially complex networks and constraints, the optimal exercise policy of the system within this approach turns out to be given only by the numerical solution of a linear/quadratic program, from which it is very difficult to extract the policy as a function of the relevant state variables (typically, the spot price of electricity and the reservoir levels). 

{The second approach is based on stochastic optimal control in continuous time, where with the dynamic programming approach one derives a partial differential equation called the Hamilton-Jacobi-Bellman (HJB) equation, i.e. a second-order partial differential equation, see e.g. \cite{FelWeb,ShaWun,TDR,ZhaDav,ForChe08} and references therein. The HJB equation is usually nonlinear, and without an explicit analytical solution: thus, one should recur to numerical methods, and this limits the dimensionality that one can reach. Proof of this is that in all these papers (apart from the notable exception of \cite{FelWeb}), the lower reservoir has infinite capacity (e.g. is a sea basin): this entails that their model has one less state variable and much easier state constraints. However, the main advantage of this technique is that one can  obtain the optimal pumping/producing strategy as a feedback control, i.e. as a function of the relevant state variables (here being time, electricity price and water levels in the basins). }

{As said above, our approach is based on optimal stochastic control in continuous time and dynamic programming, leading to a HJB equation.} Indeed, dynamic programming techniques are usually applied to prove that the value function associated to optimal control problems is the unique solution of the HJB equation in the viscosity sense and to characterize the optimal policy  as feedback of the state variables. However, while this is a well-established research field for generic unconstrained optimal control problems, PHS has the peculiarity that, being the reservoirs finite, the state variables corresponding to their levels have to satisfy given constraints: thus, we have to formulate a state-constrained optimal control problem. 
% fine Tiziano, inizio Athena!
There exists a huge literature on state-constrained optimal control problems and their HJB characterization. We refer the reader e.g. to \cite{BB95, BR98, IL02, K94}  for stochastic control problems and to \cite{CL90, IK96, S86, S862} for deterministic ones. In this case, the characterization of the value function as a viscosity solution of a HJB equation is intricate and usually requires a delicate interplay between the dynamics of the processes involved and the set of constraints, { see e.g. \cite{BCV}}. First, some \textit{viability} (or \textit{controllability}) conditions have to be satisfied to guarantee the finiteness of the value function, second, specific properties on the set of admissible controls must hold to ensure the continuity of the value function and its PDE characterization. {When the behavior at the boundary of the constrained region is clear, one could think in principle to use the same penalization techniques as in \cite{BCV}. However, when (as in our case) one does not have natural boundary conditions, this approach would not characterize the solution univoquely.} This often makes the problem not tractable by the classical dynamic programming techniques.

In this paper we follow the alternative approach developed in \cite{BPZ16} to provide a fully characterization of the value function and optimal strategy associated to the optimal control problem in a general framework.  We pass by a suitable auxiliary reformulation of the problem which allows a simplified treatment of the state constraints. This is achieved by the use of  the so called  \textit{level-set method}, built to permit a treatment of state constraints by an exact penalization technique.  Initially introduced by Osher and Sethian in \cite{OS88} to model some deterministic front  propagation phenomena, the level-set approach has been used in many applications related to  controlled systems (see e.g. \cite{ ABZ13, FGL94, KV06, ML11, ST02b}).

In our case, the level-set method allows to link  the original state constrained problem to an auxiliary optimal control problem, referred as the \textit{level-set problem} defined on an augmented state and control space, but without state constraints. This level-set problem has the great advantage of leading to a complete characterization of the original one and of being, at the same time, fully treatable by classical dynamic programming argument under very mild assumptions. 

The rest of the paper is organized as follows. We introduce the optimal control problem and the main assumptions in Section \ref{sec:problem}. In Section \ref{sec:HJB} we provide a HJB  characterization of the associated value function under suitable controllabilty conditions on the system dynamics.  Then, under a simplified model, we discuss  in Section \ref{sec:issues} the main difficulties arising from the presence of state constraints when such assumptions are not satisfied. In Section \ref{sec:2dam} we present the level set method and provide the main results  of the paper. A numerical validation of the proposed approach is provided in Section \ref{sec:tests}.

\section{Formulation of the problem and main assumptions}\label{sec:problem}

%{\color{blue}
% AP: Ho messo in questa sessione anche modelli di prezzi e caso una diga:\\}
Let $(\Omega, \mathcal F, \mathbb P)$ be a probability space supporting a one-dimensional Brownian motion $B_\cdot$ and let $\mathbb F$ be the filtration generated by $B$.
We consider the electricity price for the period $[t,T]$ governed by the following stochastic differential equation in $\R$:
\be\label{eq:demand}
\mathrm{d} X_s=b(s,X_s)\mathrm{d}s+\sigma(s,X_s)\mathrm{d} B_s\qquad s\in[t,T],\qquad X_t=x.
\ee
We work under the following assumption: 
\begin{itemize}
\item[\bf{(H1)}] $b:[0,T]\times \R\to\R$ and $\sigma:[0,T]\times \R\to \R$ and  there exists $C_0\geq 0$ such that for any $x, \bar x\in \R$, $t\in [0,T]$ one has
 \begin{eqnarray*}
 & &|b(t,x)-b(t,\bar x)|+|\sigma(t,x)-\sigma(t,\bar x)|\leq C_0  |x-\bar x|,\\
 & &|b(t,x)|+|\sigma(t,x)|\leq C_0 (1+ |x|).
 \end{eqnarray*}
\end{itemize}
Under this assumption there exists a unique strong solution, denoted by $X^{t,x}_\cdot$, to equation \eqref{eq:demand}. 
%To be consistent with the demand modeling, w
We also assume the following non negativity condition:
\begin{itemize}
\item[\bf{(H2)}] for any $t\in [0,T], x\geq 0$ one has $X^{t,x}_s\geq 0, \forall s\geq t$ a.s. .
\end{itemize}

Different price models can be taken into account. In order to guarantee assumptions (H1) and (H2) being satisfied we will focus on the following two dynamics:
\begin{itemize}
\item[a)] Price modeled as a Geometric Brownian Motion (GBM): $b(t,x) = b x$ and $\sigma(t,x) = \sigma x$ for some $b, \sigma\geq 0$;
\item[b)] Price modeled as an Inhomogeneous Geometric Brownian Motion (IGBM): $b(t,x) = a - b x$ and $\sigma(t,x) = \sigma x$ for some $a, b, \sigma\geq 0$;
\end{itemize}
Both models provide non-negative prices. In particular, the second model is used in many financial applications where one wants a process exhibiting non-negativity and mean-reversion, for example when modeling interest rates, default intensities, volatilities, etc., see e.g. \cite{Capriotti} and references therein. 

We assume to have a water storage, composed of two reservoirs. The two reservoirs  are filled with  rate $\beta_1: [0,T]\to \R$ and $\beta_2:[0,T]\to \R$ respectively with $\beta_1, \beta_2$ given positive bounded functions.
The volume $Y_{1,\cdot}$ of water remaining at every time in the first reservoir is controlled by $u_1$ and  it satisfies
\be\label{eq:reserve1}
\mathrm{d} Y_{1,s}=(\beta_1(s)-u_{1,s})\mathrm{d} s \qquad s\in[t,T],\qquad Y_{1,t}=y_1,
\ee
where $u_1$ can either be negative (pumping water  up) or positive (pulling water down).
While for the second reservoir one has  
\be\label{eq:reserve2}
\mathrm{d} Y_{2,s}=(\beta_2(s)+u_{1,s} -u_{2,s})\mathrm{d} s \qquad s\in[t,T],\qquad Y_{2,t}=y_2,
\ee
for some positive $u_2$.  We will assume that the couple $(u_1,u_2)$, that is our control, is a progressively measurable process  taking values in $U:=[- { \underline u_1},\bar u_1]\times [0,\bar u_2]$. We denote by $\mathcal U$ the set of these controls\footnote{{Physically, a more sophisticated model could be like that in \cite{TDR}, who perform a fine modeling of the physical constraints for the water flows between the higher and the lower basin (this one still having infinite capacity), and consider the water inflow of the upper basin as stochastic. Differently from this, for ease of exposition we treat both the water inflow into the two basins, as well as the constraints on the outflow, as deterministic quantities not depending on water height. Our simplification has no practical consequences when the total height of the higher basin is negligible with respect to its relative height with respect to the lower one; if this is not the case, however, these constraints could be easily included in our model. }}. For any $t\in [0,T]$, $y\equiv (y_1,y_2)$ and any choice of $u\equiv(u_1,u_2)\in \mathcal U$, we denote by $Y^{t,y,u}_\cdot\equiv(Y^{t,y,u}_{1,\cdot},Y^{t,y,u}_{2,\cdot})$ the solution to \eqref{eq:reserve1}-\eqref{eq:reserve2}.
For $i\in\{1,2\}$, the value $Y^{t,y,u}_{i,\cdot}$ is required to remain non negative and bounded by a maximum value $\ymax_i$ (the reservoir capacity) along the interval $[t,T]$. This is expressed by the following constraint on the state:
$$
(Y^{t,y,u}_{1,s},Y^{t,y,u}_{2,s})\in [0,\ymax_1]\times [0,\ymax_2]=: K ,\qquad \forall s\in [t,T]\quad \text{a.s.}\,.
$$  
The amount of electricity instantaneously {produced or consumed} by the reservoirs is modeled by the function $\kappa:U\to \R$ defined as follows{\footnote{ A more realistic model would incorporate in the function $\kappa$ some  parameter for the efficiency of the energy production process. We could for instance consider $\kappa(u) = \eta_2 u_2 + \eta_1 c(u_1)$ for some $\eta_1,\eta_2\leq 1$. To  simplify the notation here we assumed $\eta_1=\eta_2=1$.}   
$$
\kappa(u) :=  \left(u_2 + c(u_1)\right)  \qquad \text{with } c(u)=\begin{cases} u \quad\text{ if } u \geq 0\\ \gamma u \quad\text{ if } u<0\end{cases} \quad (\gamma >1)
$$
(more generally, one can consider any concave function $c$ such that $c(0)=0$). 
The aim of the controller is to maximize the cash flow obtained by selling the produced electricity $\kappa(u)$ at the price $x$. This results in the following state constrained optimal control problem
\be\label{eq:def_V}
V(t,x,y):=\underset{u\in \Uad}\sup\;\mathbb E\left[\int^T_t L(X^{t,x}_s, u_s)\mathrm{d} s \right]
\ee
with $L:[0,+\infty)\times U\to\R$ given by
$$
L(x, u) := x \kappa(u)
$$
and 
$ \Uad$ the set of progressively measurable processes in $\mathcal U$ such that  $Y^{t,y,u}_s\in K$ for any $s\in [t,T]$ a.s..
We refer at the function $V:[0,T]\times\R\times K\to \R$ as the value function of the problem. 
In the sequel, whenever $\Uad= \emptyset$ we use the convention $V(t,x,y)=-\infty$.
\medskip

Along the paper we will also consider a simplified problem of one single reservoir with level managed by $u_1$ taking values in $U=[0,\bar u]$. We then restrict to the two-dimensional dynamics $(X_s, Y_{1,s})$ and consider 
$$
\kappa(u_1):= u_1.
$$
To simplify the notation,  in this case we will drop the subscripts indices. More precisely, one  gets the optimal control problem \eqref{eq:def_V} with $L(x,u):=x u$ subject to \eqref{eq:demand} and
\beno
\mathrm{d} Y_{s}=(\beta(s) -u_{s})\mathrm{d} s \qquad s\in[t,T],\qquad Y_{t}=y
\eeno
under the state constraint $Y^{t,y,u}_s\in K,\, \forall s\in [t,T]$ a.s. with $K:=[0,\ymax]$.

\section{HJB characterization under controllability assumption}\label{sec:HJB}

Let us start considering the single reservoir model. 
To have a meaningful problem we impose that 
$$
\int^T_0 \beta(s)\mathrm{d} s \leq \bar u T
$$
i.e. the entire amount of {water entering the basin} in the period $[0,T]$ is not greater than the amount that can be withdrawn.\footnote{We are not assuming, as instead is in real dams, that we have the possibility to eliminate the excess water in the reservoir, i.e. to be able to withdraw from the reservoir, if it is close to be full, via a safety discharge, thus not producing electricity. This however typically results in a waste of water, i.e. of potential electricity production, thus of potential profit (recall that we assume the electricity price being nonnegative). Thus, this should always be regarded as suboptimal, and is not modeled here.}}
In what follows we denote by $Q:=[0,+\infty)\times (0,\ymax)$ the spatial domain and by $\overline Q$ its closure.

In order to have a finite value function  on $[0,T]\times \overline Q$  one needs $\Uad\neq \emptyset$ for any $t\in [0,T], y\in [0,\ymax]$. The following assumption ensures this.
\begin{itemize}
\item[\bf{(H3)}]  there exists $\eta>0$ such that $\eta \leq \beta(t) \leq \bar u -\eta , \forall t\in [0,T]$. 
\end{itemize}
%\begin{figure}
%\includegraphics[width=0.5\textwidth]{grafico3.png}
%\fbox{figure}
%\caption{\cblue{We can think to include a realistic shape for the function $\beta$}.}
%\end{figure}
\begin{rem}
To guarantee that $\Uad \neq \emptyset$ it is sufficient that $0 \leq \beta(t) \leq  \bar u,  \forall t\in [0,T]$ since, in this case, the control $u_t\equiv\beta(t), \forall t\in [0,T]$ is always admissible. However, the stronger assumption (H3) is what one needs for the well posedness of the state constrained problem given by Theorem \ref{teo:with_h3} below. 
\end{rem}

We start with the following result concerning some regularity and growth properties of $V$.
\begin{lem}\label{lem:growth}
Let assumptions (H1)-(H3) be satisfied. Then, there exists a constant $C\geq 0$ such that for any $t\in [0,T], x, \bar x\in [0,+\infty), y\in [0,\ymax]$ one has
$$
0\leq V(t,x,y)\leq C(1+x)\qquad
\text{ and }\qquad 
|V(t,x,y)-V(t,\bar x,y)|\leq C|x-\bar x|. 
$$
Moreover, for any $(x,y)\in \overline Q$
$$
\lim_{t\to T} V(t,x,y) =0.
$$
\end{lem}
\begin{proof}
Thanks to assumption (H3), $V$ is finite on $[0,T]\times \overline Q$, and its non negativity follows directly by the definition of $L$ and by assumption (H2). 
Under assumption (H1) it is well known that for any $x, \bar x\in \R$ 
\begin{align*}
\mathbb E\left[ \sup_{s\in [t,T]} (X^{t,x}_s)^2 \right] \leq C (1+x^2)\qquad \text{and}\qquad \mathbb E\left[ \sup_{s\in [t,T]} \left|X^{t,x}_s - X^{t,\bar x}_s\right|^2 \right] \leq C |x-\bar x|^2,
\end{align*}
where the constant $C$ only depends on $T$ and the constants appearing in assumption (H1).
By the very definition of $L$, it follows that 
\begin{align*}
V(t,x,y) & \leq  \bar u\, \mathbb E \left[ \int^T_t  X^{t,x}_s \mathrm{d} s \right]\leq \bar u (T-t)\,   \mathbb E \left[  \sup_{s\in [t,T]} X^{t,x}_s \right] \\
|V(t,x,y)-V(t,\bar x,y)| & \leq \sup_{\Uad}\, \mathbb E \left[\int^T_t |L(X^{t,x}_s,u_s)-L(X^{t,\bar x}_s,u_s)| \mathrm{d}s \right]\leq \bar u (T-t) \,  \mathbb E \left[\sup_{s\in [t,T] } |X^{t,x}_s-X^{t,\bar x}_s| \right]
\end{align*}
from which the first two statements follow.
 Let $h>0$, from the previous estimate one  also has 
 \begin{align*}
 0\leq V(T-h,x,y) \leq C h (1+x)
 \end{align*}
 which gives the last result.
\end{proof}
The following theorem allows us to characterize the value function $V$  as the unique solution of a suitable HJB partial differential equation. Due to the strong degeneracy of the diffusion term we do not expect to have classical (i.e. $C^{1,2}([0,T]\times \overline Q)$) solutions of the equation. For this reason we use the notion of viscosity solutions.
% For completeness the main definitions are included in Appendix \ref{app:proof_h3} (we remand to \cite{CIL92} for an overview on the subject). 
\begin{theo}\label{teo:with_h3}
Let assumptions (H1)-(H3) be satisfied. Then $V$ is the unique continuous viscosity solution with linear growth of the following state constrained HJB equation
\begin{subequations}
\begin{align}
& -V_t - b(t,x) V_x -\frac{1}{2} \sigma^2(t,x) V_{xx}  +\inf_{u\in U} \left\{  -(\beta(t) - u) V_y - L(x,u)   \right\} = 0  & \qquad \text{on } (0,T)\times \overline Q \label{eq:HJB_constr_price}\\
& V(T,x,y) =0 & \qquad   \text{on }  \overline Q
\end{align}
\end{subequations}
\end{theo}
{The proof follows by the dynamic programming and comparison principle for viscosity solutions of state-constrained HJB equations. We remand to \cite{BB95} and \cite{CapDolcLions90} for further details.} 

It is possible to check that the analogous of assumption (H3) in the two dams setting is the following requirement: 
\begin{itemize}
\item[\bf(H3)'] there exists $\eta>0$ such that for all $t\in [0,T]$
$$ 
\eta -{\underline u}_1 \leq \beta_1(t) \leq {\bar u}_1 -\eta,\quad   -({\bar u}_1\wedge \underline u_1)+\eta  \leq \beta_2(t) \leq {\bar u}_2 +{\underline u}_1 -\eta, \quad \eta\leq   \beta_1(t)+\beta_2(t)\leq {\bar u}_2 - \eta.
$$
\end{itemize}
which leads to the following HJB equation 
\begin{equation} \tag{5a'}
-V_t  - b(t,x) V_x -\frac{1}{2} \sigma^2(t,x) V_{xx} +\inf_{u\in U} \Big\{  -(\beta_1(t) - u_1) V_{y_1} -(\beta_2(t) + u_1-u_2) V_{y_2}  - L(x,u)   \Big\} = 0
\end{equation}

\begin{figure}
\includegraphics[width=0.6\textwidth]{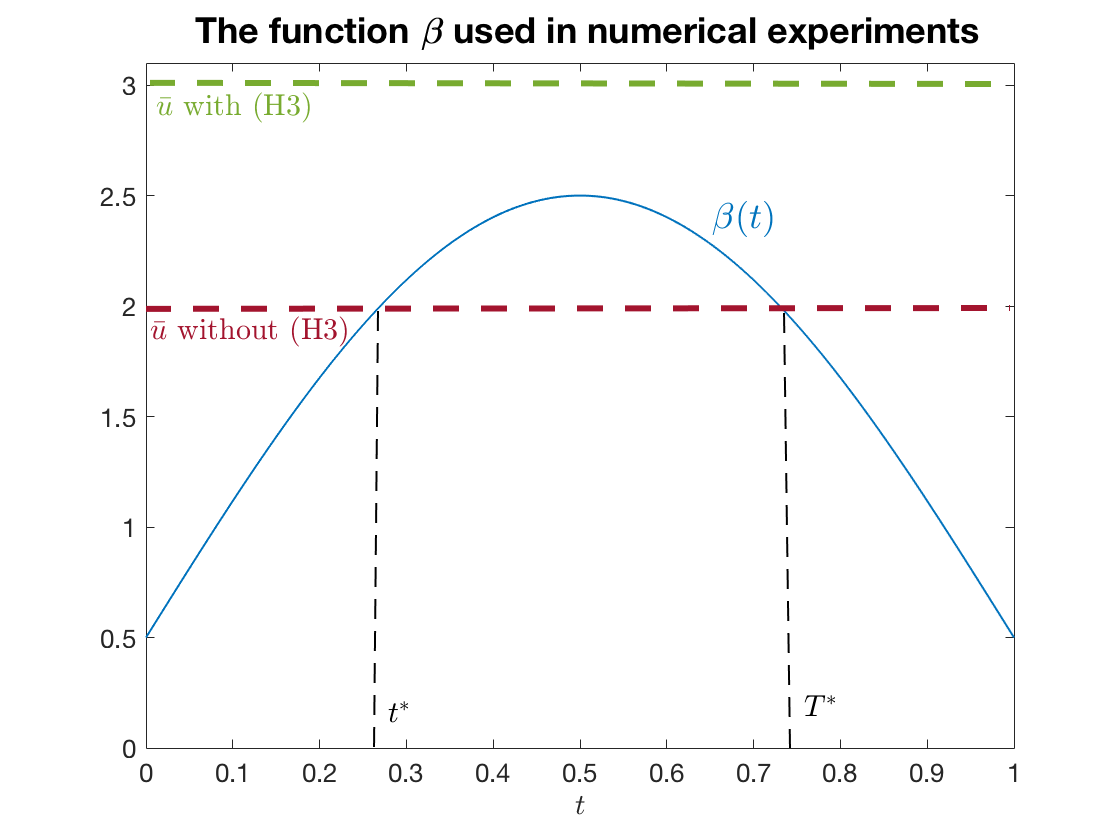}
\caption{The function $\beta(t) = 2\sin(\pi t) +0.5$ used in our numerical tests.}\label{fig:beta}
\end{figure}

Problem \eqref{eq:HJB_constr_price} does not admit, in general, an explicit solution.  
We then need to approximate its solution (i.e. the value function and the optimal feedback) numerically.
Here, we use a semi-Lagrangian scheme, see  e.g. \cite{ForChe07, ForChe08}. 
In our numerical experiments, we fix the time horizon $T=1$, the maximal capacity of the reservoir $\bar y=1$ 
and take the following function 
$$\beta(t) = 2\sin(\pi t)+0.5$$
as the filling rate of the reservoir. The function $\beta(\cdot)$ is plotted in Figure \ref{fig:beta}. In this section we assume that the dam under consideration admits maximal discharge rate $\bar u = 3$. It is immediate to observe that this value is sufficiently big to ensure that  assumption (H3) is satisfied, which allows us to use Theorem \ref{teo:with_h3} to characterize the solution of our problem. 
\\
\\
\begin{figure}
\includegraphics[width=0.45\textwidth]{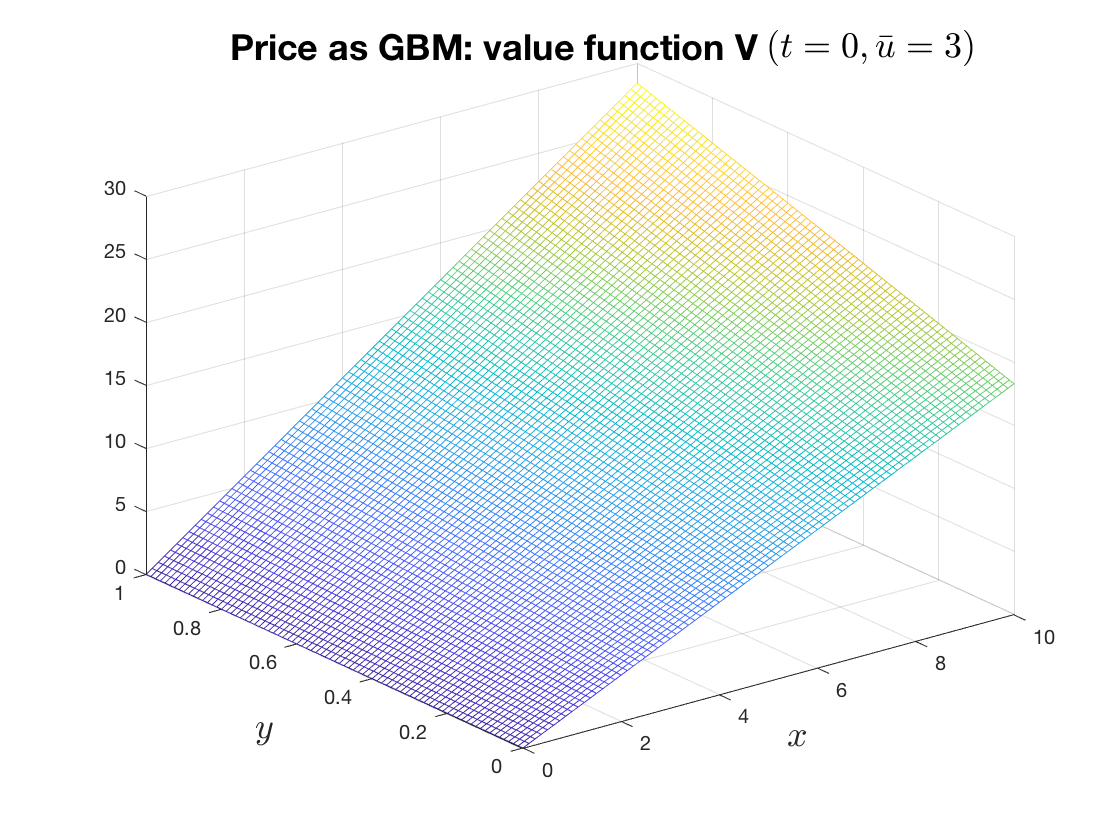}
\includegraphics[width=0.45\textwidth]{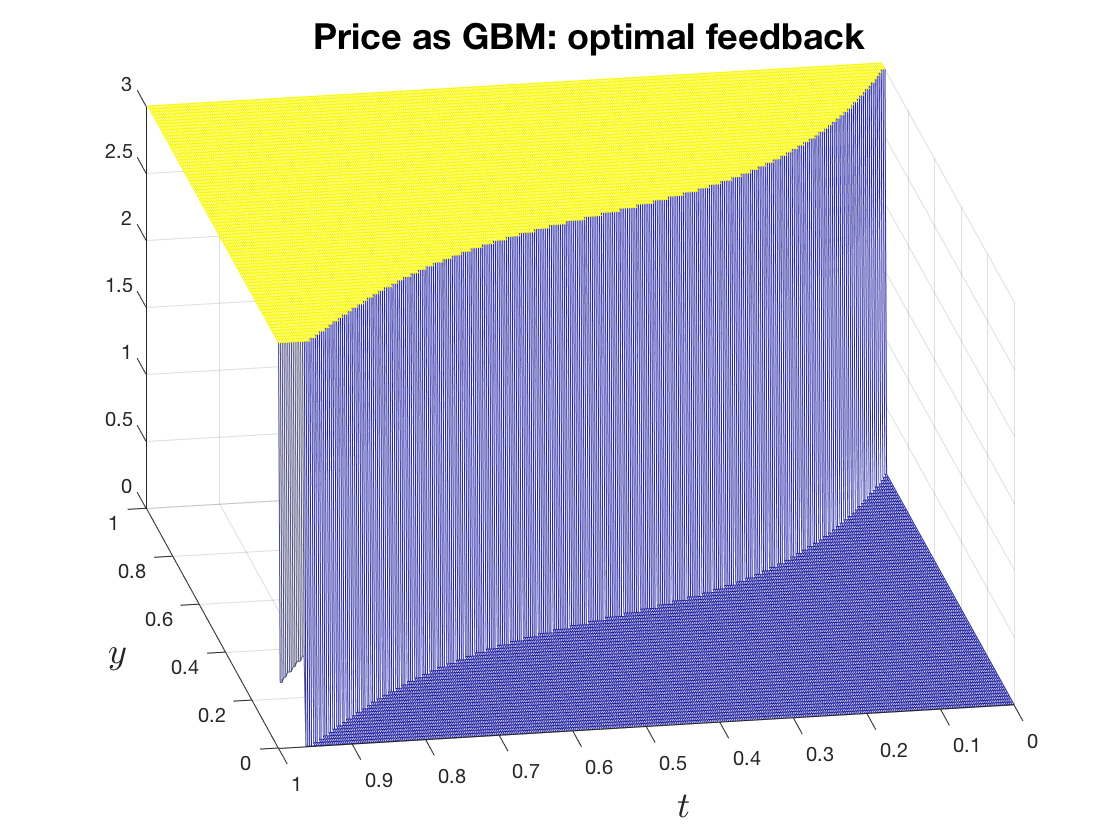}
\caption{Single dam, price modeled as a GBM with $b(t,x)=0.05 x$ and $\sigma(t,x)=0.1 x$, maximal discharge rate $\bar u=3$ (assumption (H3) is satisfied). Left: numerical approximation of the value function $V$ at $t=0$ . Right: optimal feedback as a function of $(t,y)$}\label{fig:1dam_price}
\end{figure}
In Figure \ref{fig:1dam_price} we plot a numerical approximation of the value function $V$ at time $t=0$ (left) and the optimal feedback control (right)
obtained  solving \eqref{eq:HJB_constr_price} for a price  process following a GBM. 
Notice that when the price is modeled as a GBM the value function can be factorized, as the electricity price depends on the initial value $X_t = x$ just via a multiplication factor, thus we have
\begin{equation} \label{linearinx}
V(t,x,y) = \sup_{u\in \Uad} x\, \mathbb E\left[ \int^T_t e^{(b-\frac12 \sigma^2) s + \sigma B_s} u_s\mathrm d s\right] =: x \,v(t,y) 
\end{equation}
where $v$ satisfies a simplified version of the HJB equation, given by
$$ -v_t(t,x) + \inf_{u\in \Uad} \left\{ - b  v(t,x) - (\beta(t) - u) v_y(t,x) - u \right\} = 0 $$
still with terminal condition $v(T,y) = 0$. Given this multiplicative decomposition, here we have the optimal control which ends up being just a feedback control of $(t,y)$, without explicit dependence on $x$.
In terms of the dependency on the $y$ variable, we can see that the graphics reflects quite well the expectations that we have on this model. Indeed, the value function is increasing with respect to the reservoir level $y$, and the dependence here seems almost linear: this is maybe due to the fact that, regardless on the current level $y$, we always have the maximum possible flexibility in the control $u$, and one can let the dam fill or use its flow always at maximum speed in either direction.
\\
\\
 The case of the electricity price modeled as an IGBM is given in Figure \ref{fig:1dam_IGBMprice}. 
With this process, we cannot use the multiplicative decomposition of the value function seen in the GBM case, thus the optimal feedback will in general depend on state $x$. We notice that the value function exhibits a linear growth both in $x$ as in $y$, as in the GBM case (even if here, differently from the GBM case, we have $V(t,0,y) > 0$). 
Here the interpretation is the same as in the GBM case: regardless on the current level $y$, we always have the maximum possible flexibility in the control $u$. Also, the optimal control is always of bang-bang type and increases with respect to $x$, which is quite natural, as it is optimal to produce more and more as the electricity price increases; conversely, when the price is low, it is optimal to delay production and wait for the price to increase again. 

In the two dams model we plot the results obtained under a GBM and IGBM model in Figure \ref{fig:2dam_price} and \ref{fig:2dam_priceIGBM}, respectively. Here, we take $\beta_1=\beta_2=\beta$, $\gamma=1.5$, $\bar y_1=\bar y_2=1$, $u_1\in [-1, 3]$ and  $u_2\in [0, 5.5]$ which ensures $(H3)'$ being satisfied. We can observe similar results as the single dam model: if the price is modeled by a GBM one gets an optimal feedback independent of $x$ and  with the same qualitative behavior observed in the case of a single dam. When the price if modeled by an IGBM, if the price is above the long-term mean (which is 5 here), the optimal strategy is to discharge independently of the reservoir level.  Conversely, if  the price is below, it is optimal to let the reservoirs be filled; observe that in this case, for very low levels of the price (see the case $x=0.5$ in Figure \ref{fig:2dam_priceIGBM}), the water is also pumped back to the upper reservoir.

%The parameters used in the numerical experiments are given in tables

% {\color{red} [ATTENZIONE: NELLA STORIA E' PIU' PLAUSIBILE CHE $\bar u$ SIA COSTANTE MA CAMBI $\beta$, PARLARE CON ATHENA]}. ,
 %\begin{table}[!hbtp]
%\centering
%\begin{tabular}{|c|c|c|}
%\hline 
%$T$ & $\bar u$  & $\bar y$ \\
%\hline
%$1$ & $3$ & $1$  \\
%\hline
%\end{tabular}
%\caption{Test 1: Parameters used in numerical experiments.}
%\label{tab:data}
%\end{table}

\begin{figure}
\includegraphics[width=0.5\textwidth]{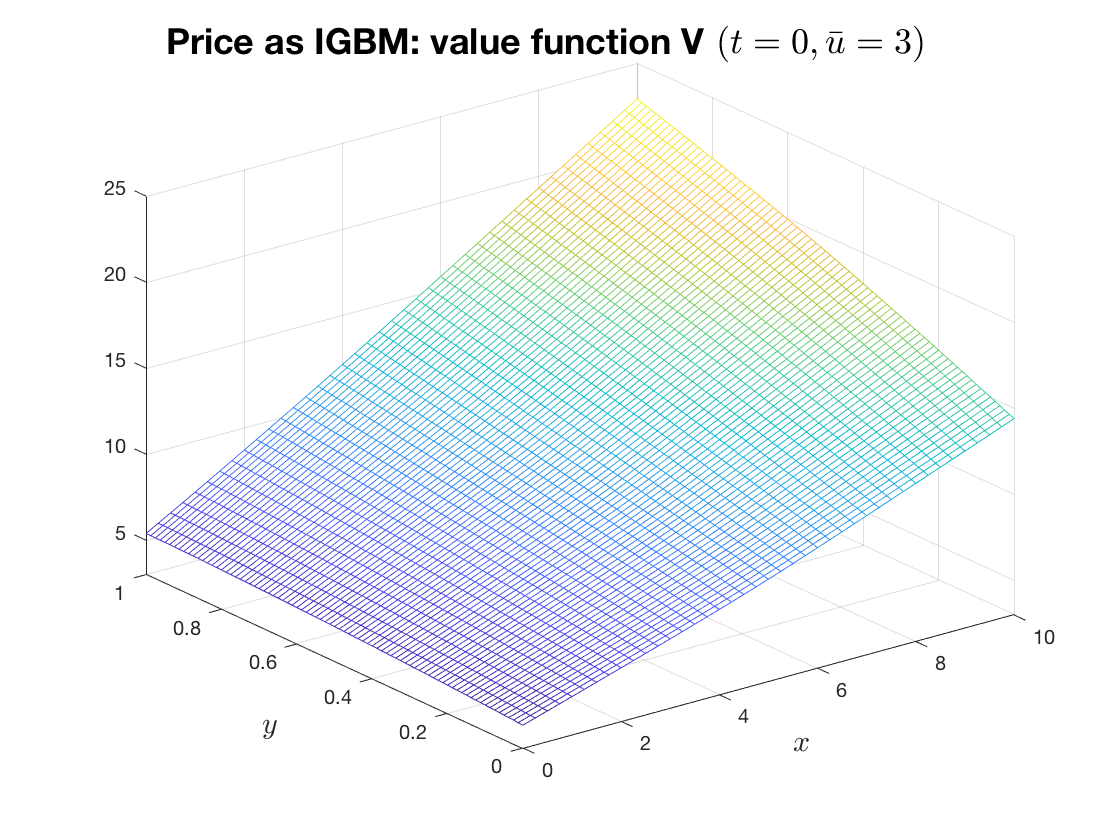}\\
\includegraphics[width=0.24\textwidth]{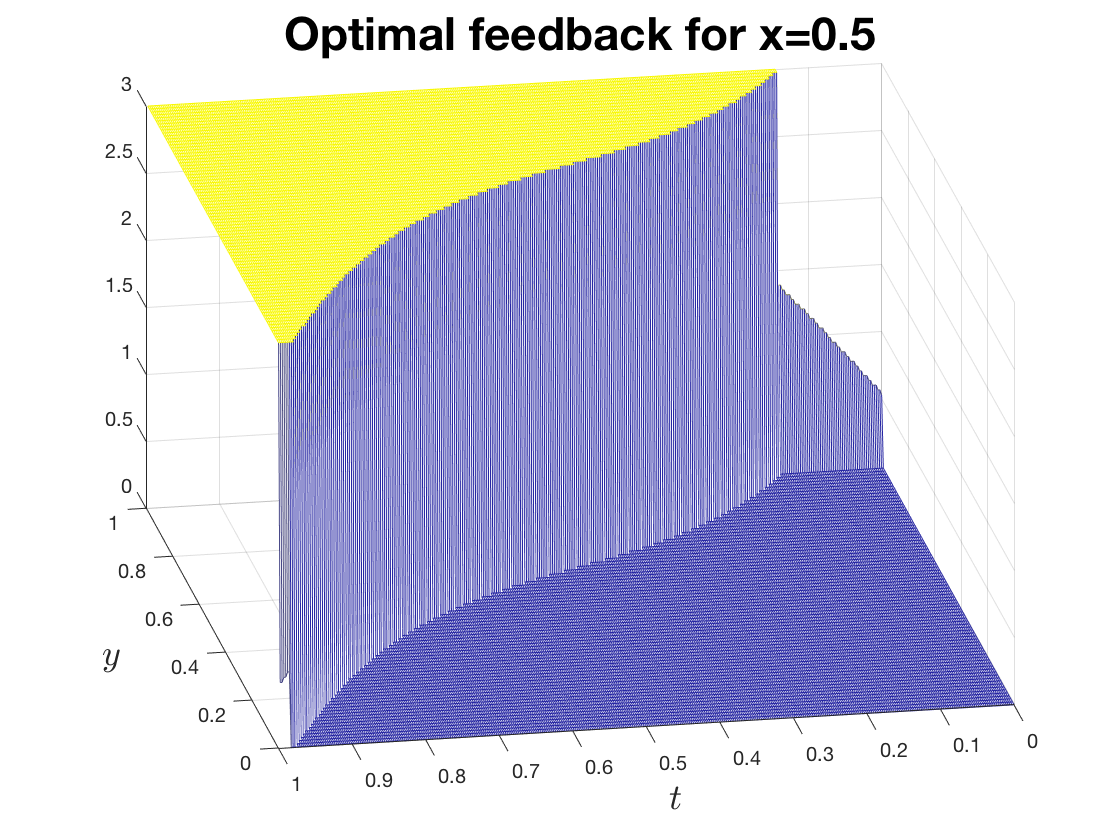}
\includegraphics[width=0.24\textwidth]{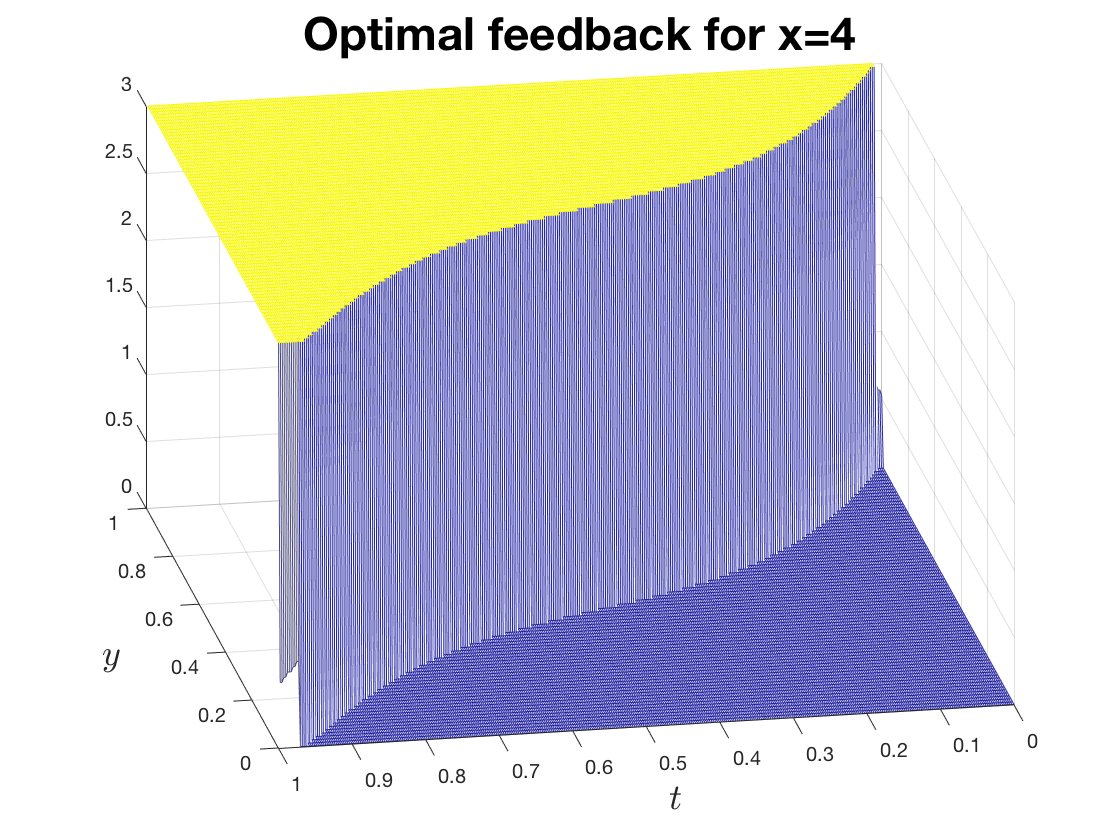}
\includegraphics[width=0.24\textwidth]{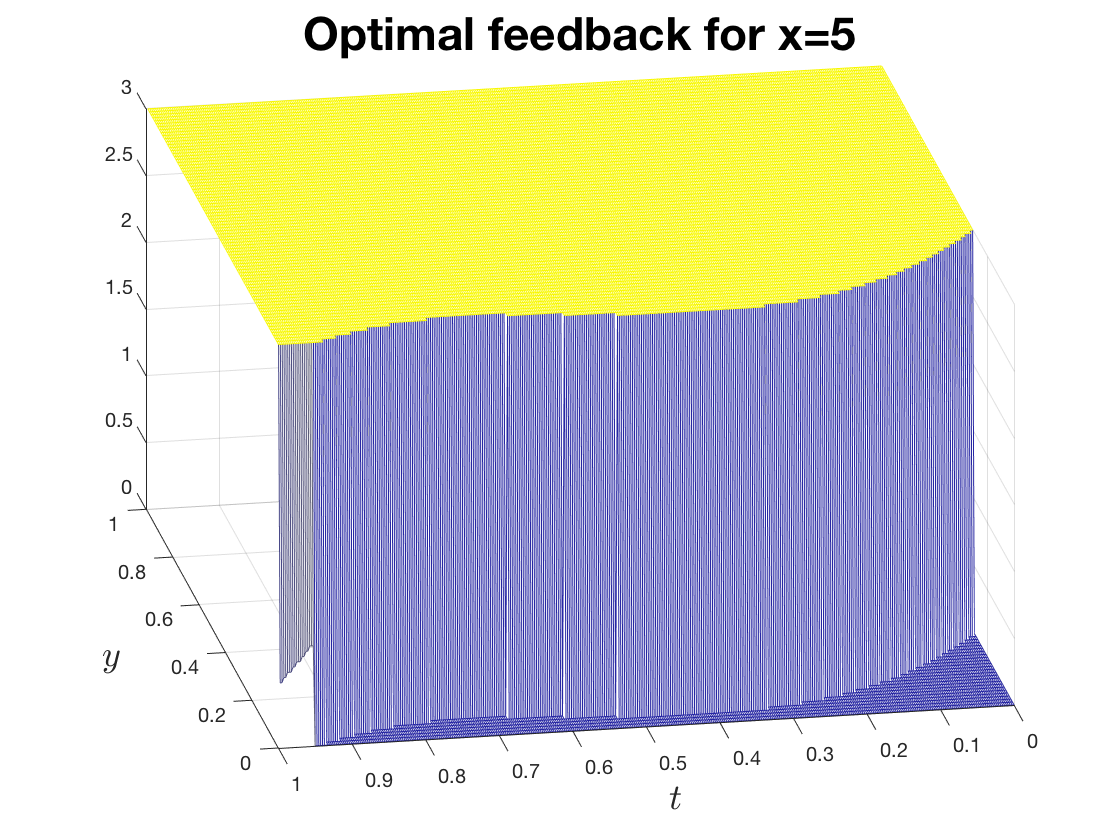}
\includegraphics[width=0.24\textwidth]{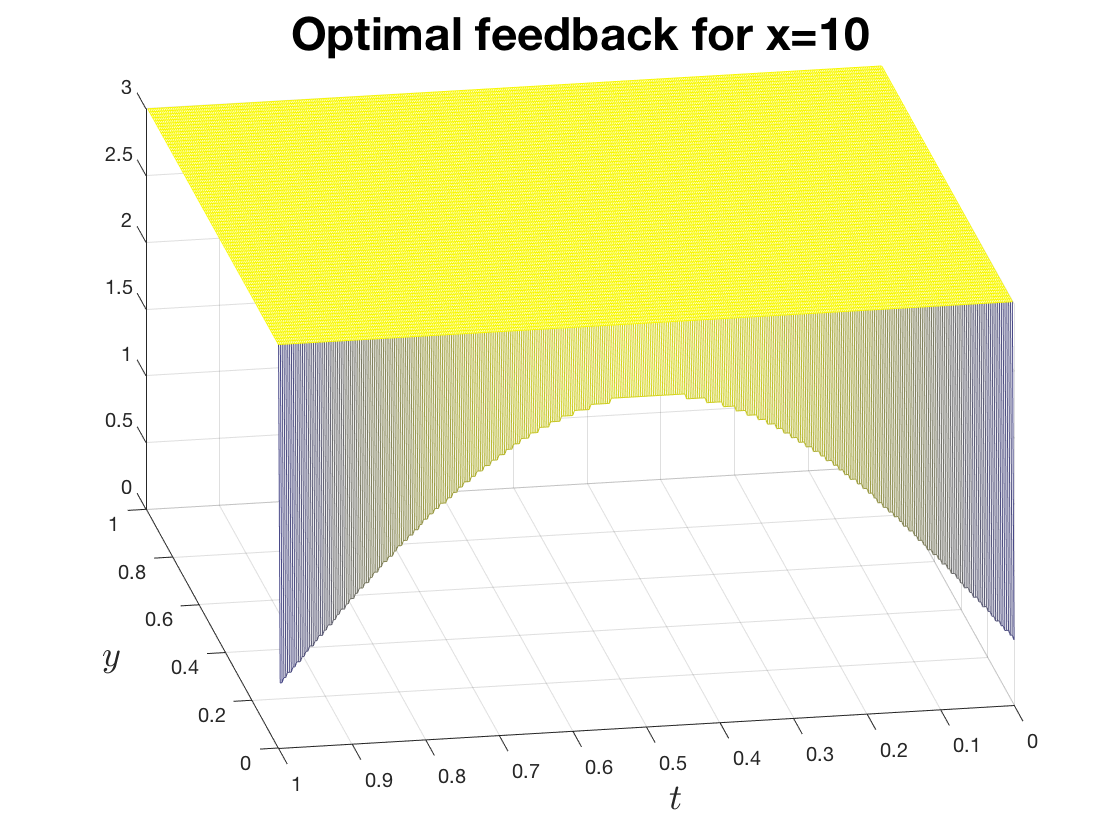}
\caption{Single dam model, price modeled as IGBM with $b(x) = 5-x$ and $\sigma(x) = 0.1 x$, maximal discharge rate $\bar u=3$ (assumption (H3) is satisfied). Top: numerical approximation of the value function $V$ at $t=0$. Bottom: optimal feedback as a function of $(t,y)$ for different values of  $x=0.5, 4,  5 ,10$ (from left to right).} \label{fig:1dam_IGBMprice}
\end{figure}

\begin{figure}
\includegraphics[width=0.5\textwidth]{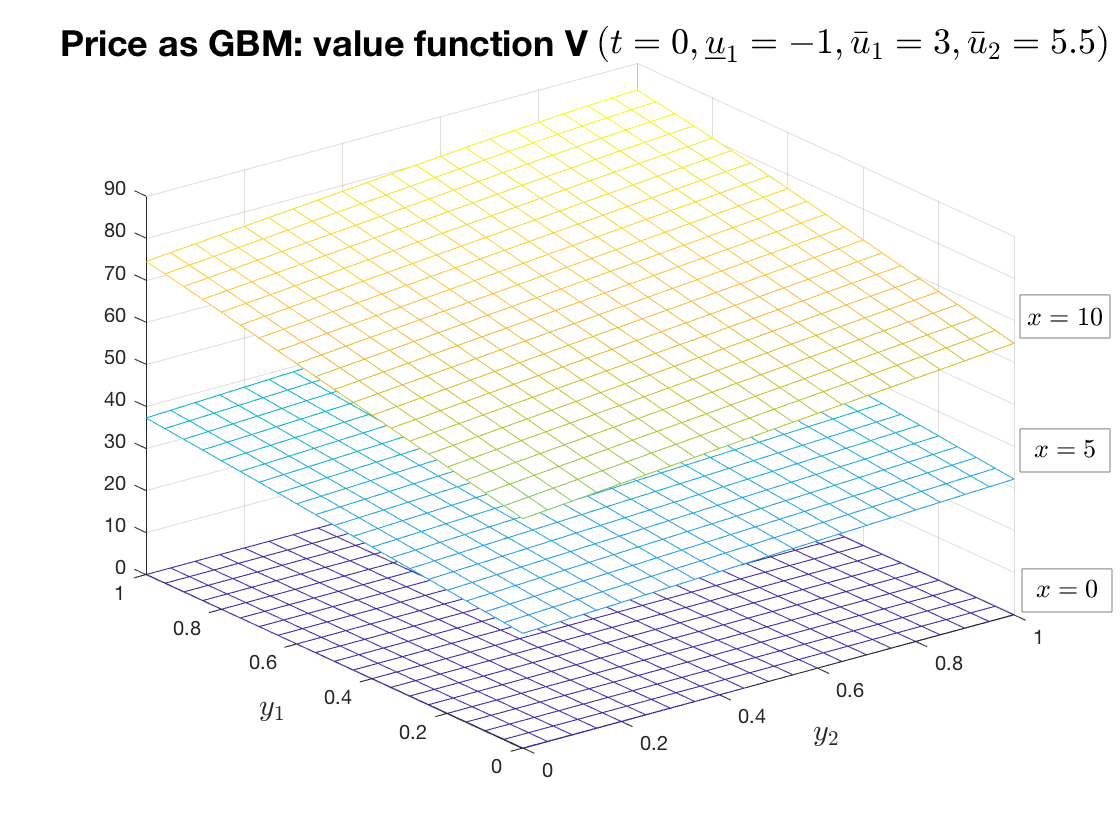}\\
\includegraphics[width=0.3\textwidth]{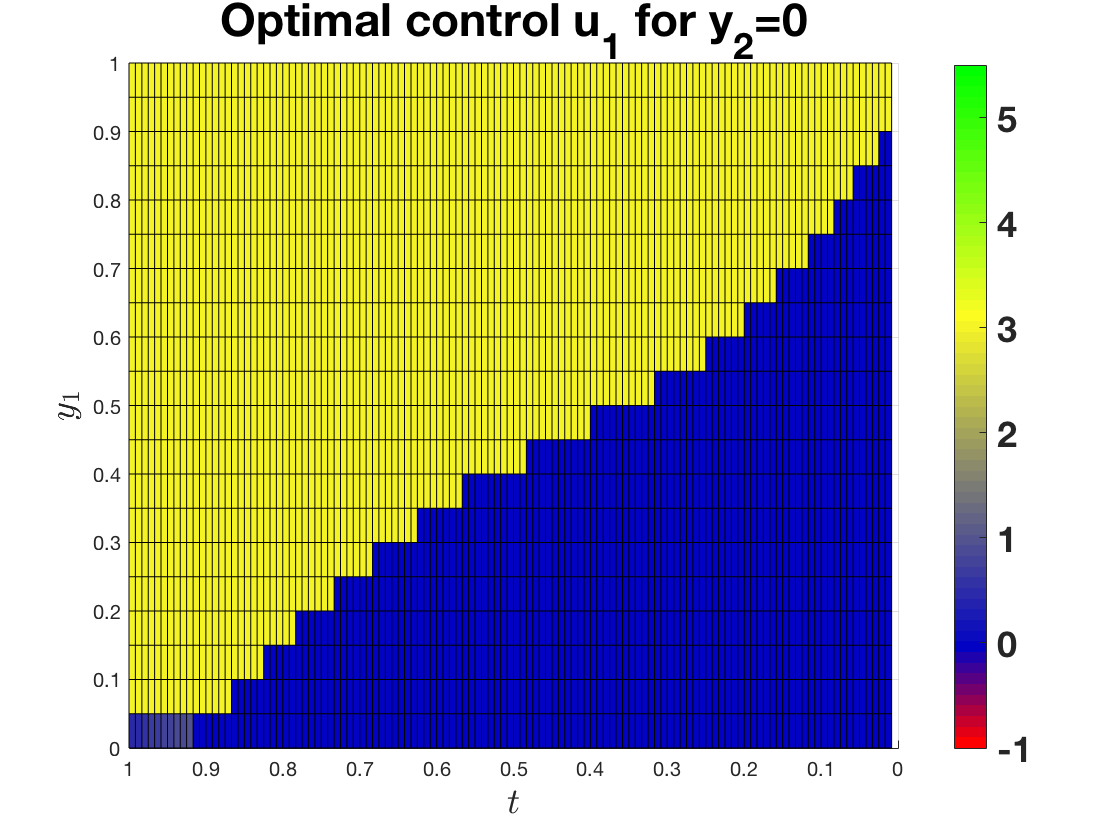}
\includegraphics[width=0.3\textwidth]{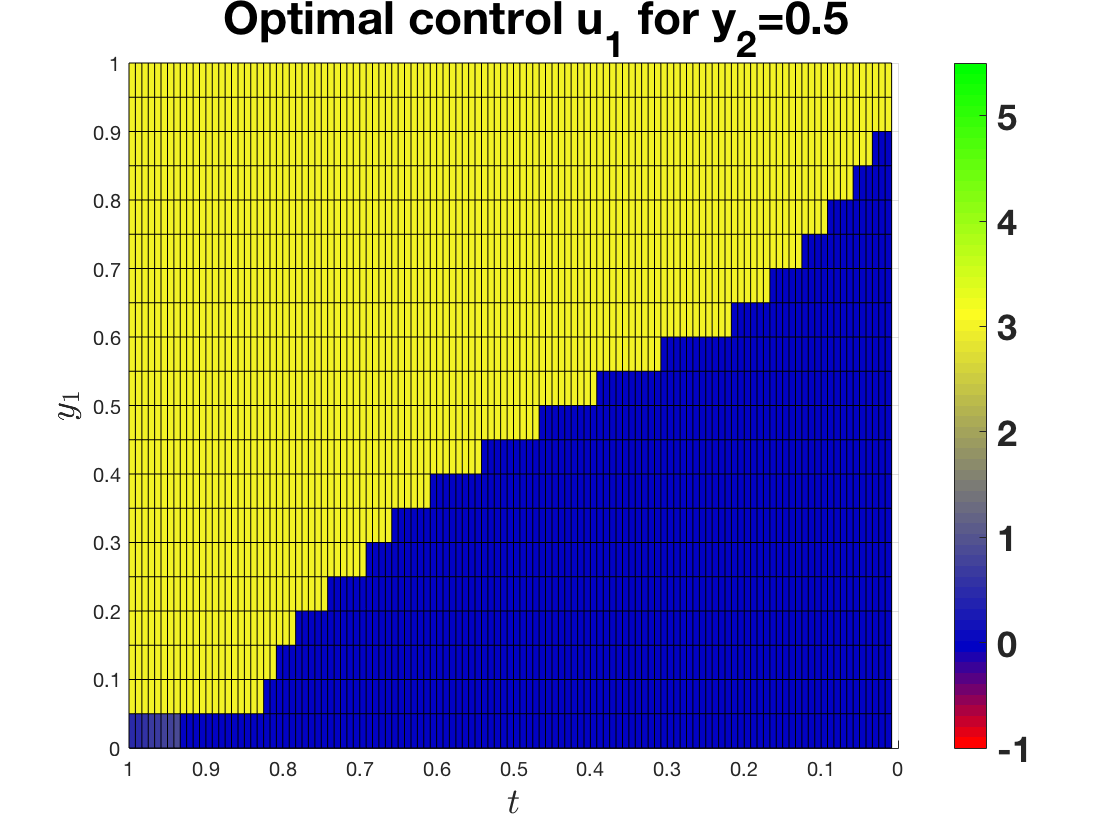}
\includegraphics[width=0.3\textwidth]{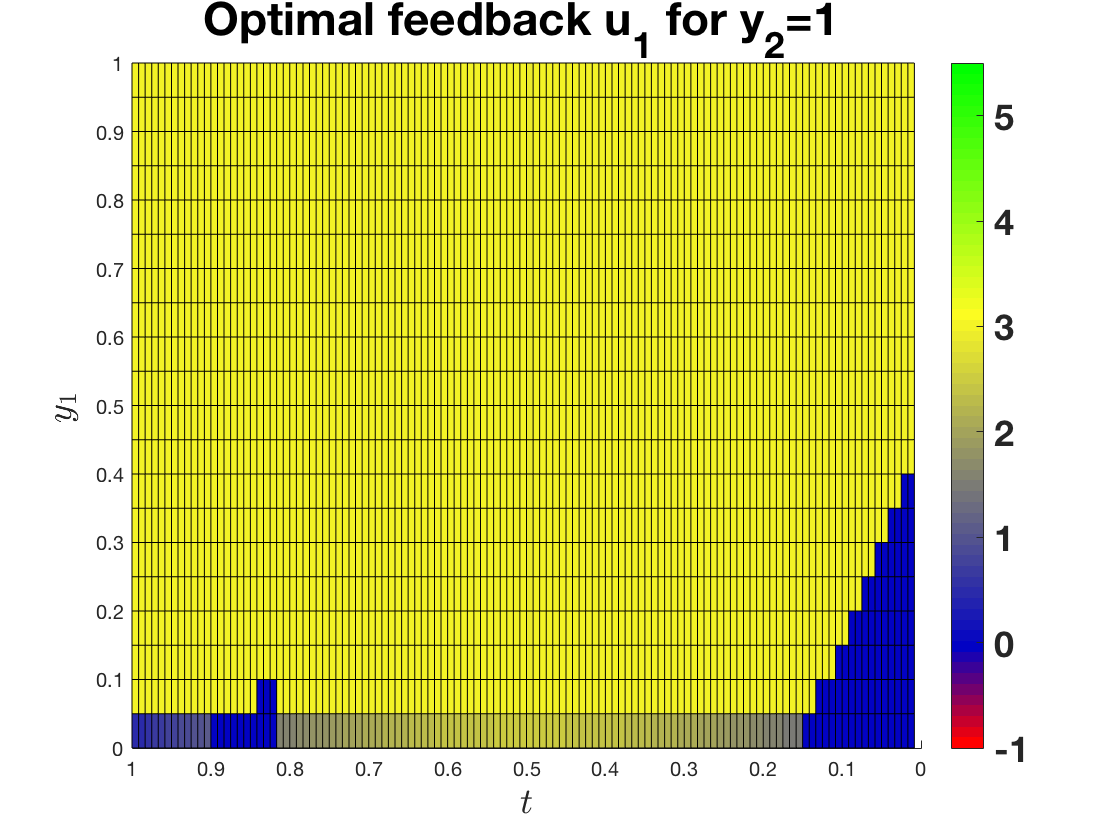}\\
\includegraphics[width=0.3\textwidth]{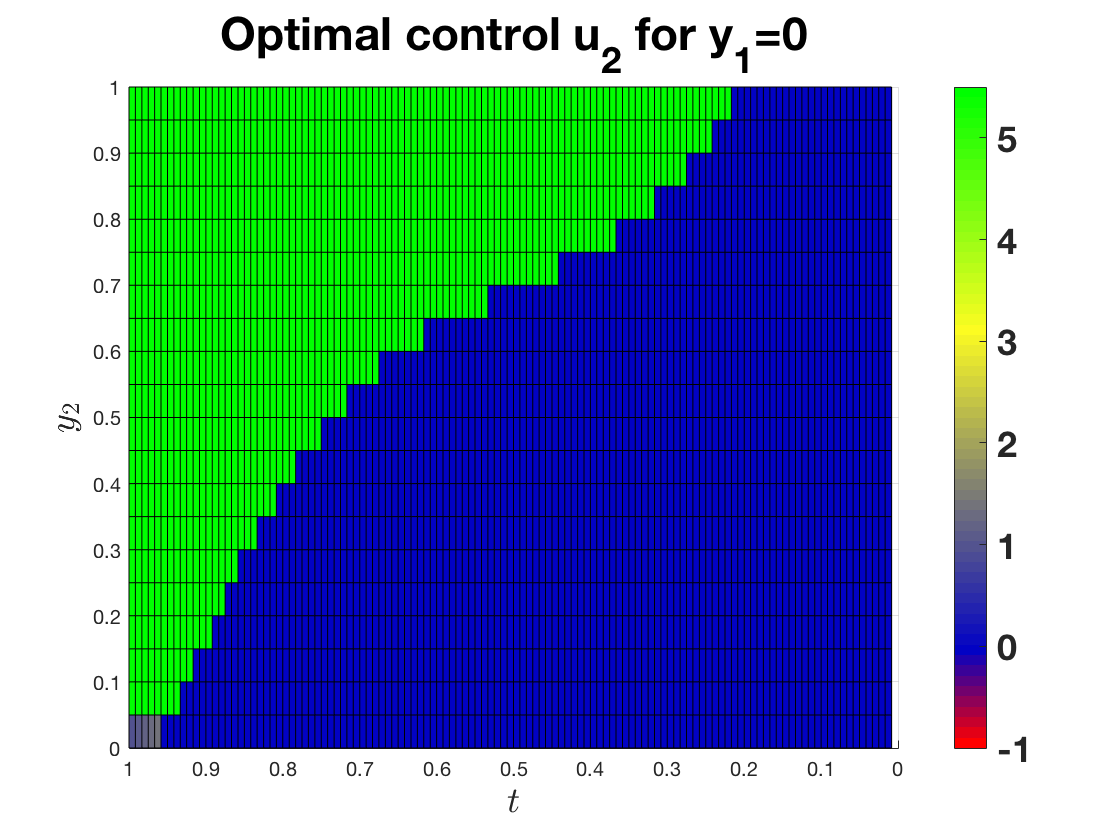}
\includegraphics[width=0.3\textwidth]{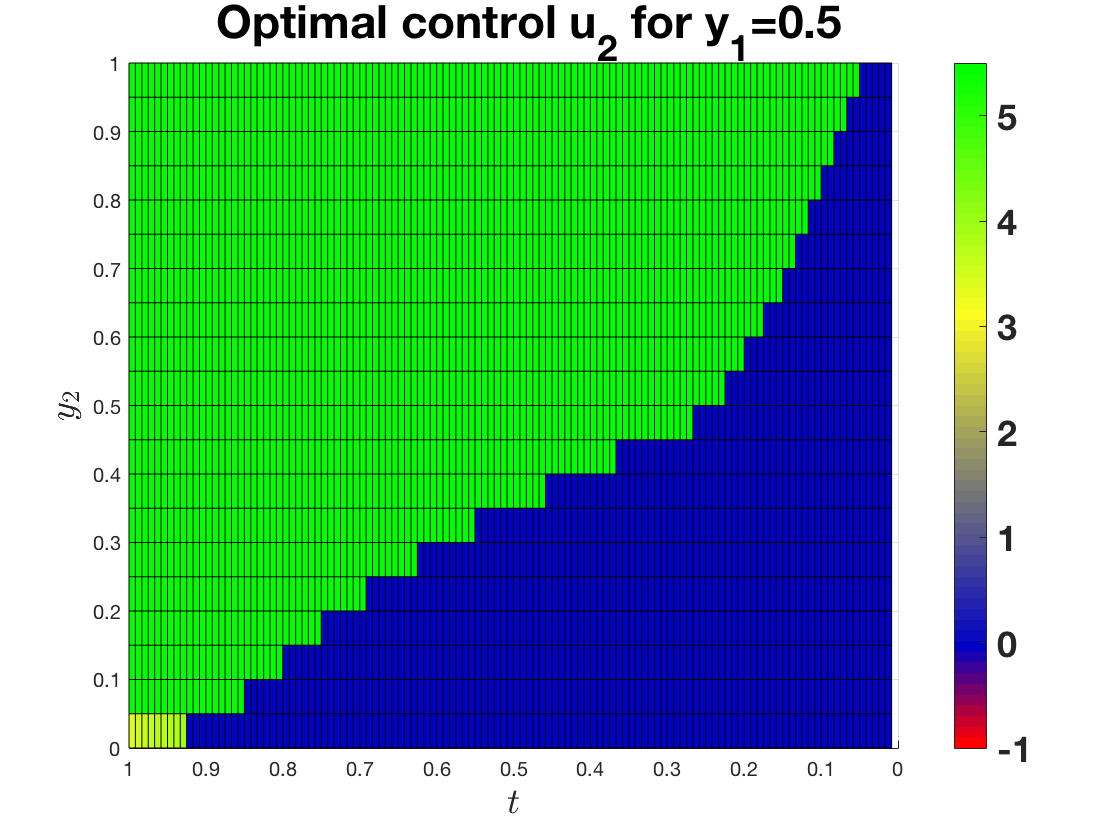}
\includegraphics[width=0.3\textwidth]{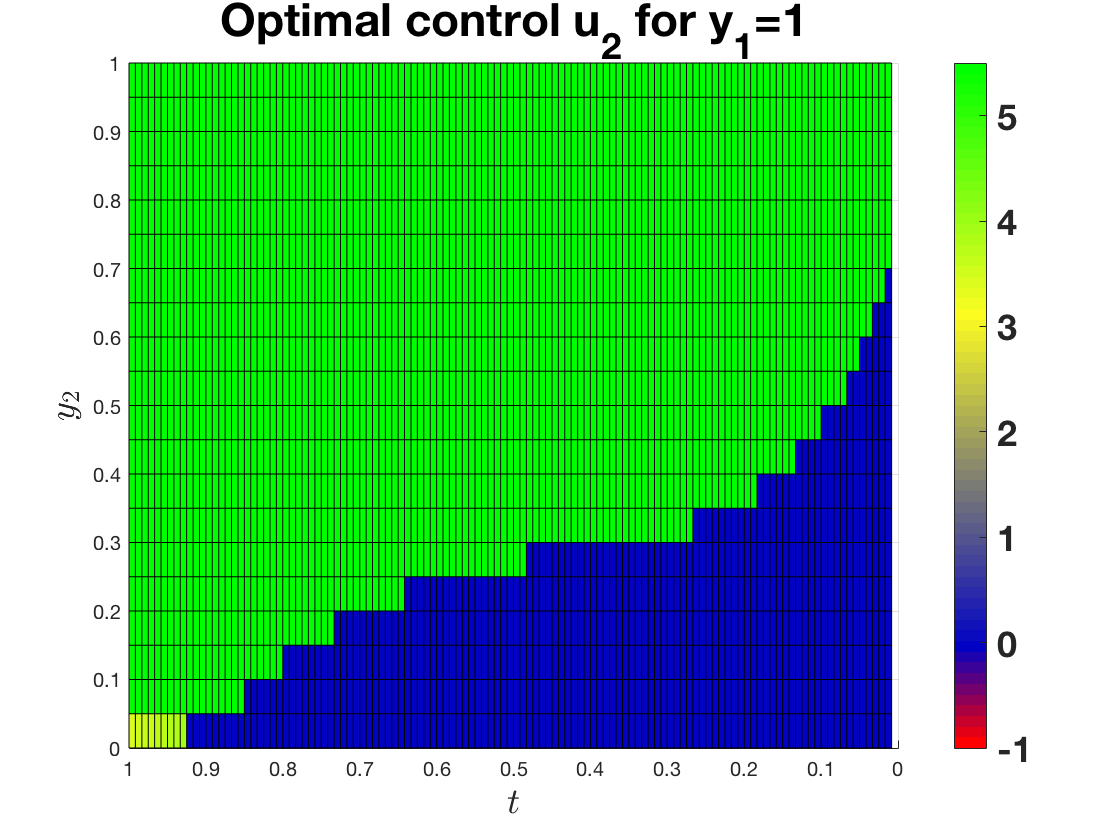}
\caption{Two dams, price modeled as GBM with $b(t,x)=0.05 x$ and $\sigma(t,x)=0.1 x$, maximal/minimal discharge rate $\underline u_1 = -1, \bar u_1=3$ and $\bar u_2 = 5.5$ (assumption (H3)' is satisfied). Top: numerical approximation of the value function $V$ at $t=0$. Second (resp. third) line: optimal feedback $u_1$ (resp. $u_2$) as a function of $(t,y_1)$ (resp. $(t, y_2)$) for different values of $y_2$ (resp. $y_1$).}
 %{\color{red} Cose da controllare: plot di $u_2$ per $y_1=0$ non viene rate di scarico $\beta_1+\beta_2$? plot di $u_1$ per $y_2=1$ me lo sarei aspettato senza "buco" in mezzo...}}
 \label{fig:2dam_price}
\end{figure}

\medskip

\begin{figure}
\includegraphics[width=0.5\textwidth]{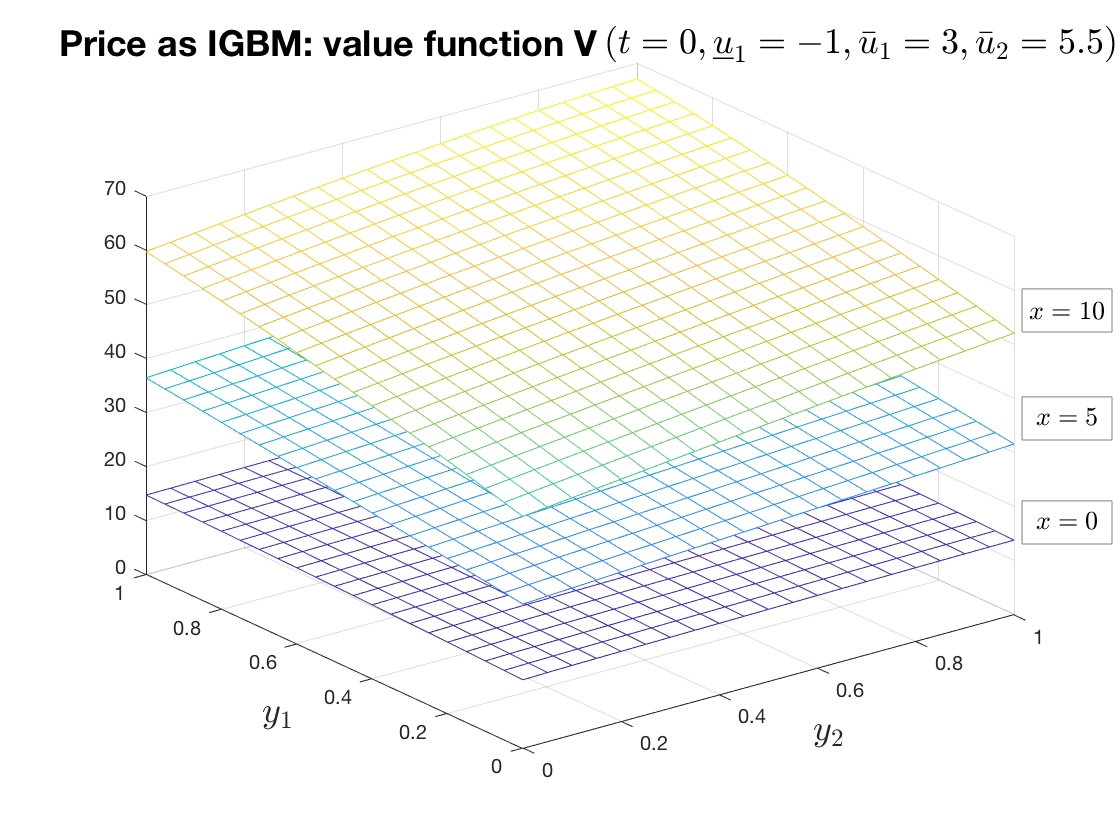}\\
\begin{minipage}[h]{0.08\textwidth}\vspace{-5cm}\fbox{$x=0.5$}\end{minipage}
\begin{minipage}[h]{0.9\textwidth}
\includegraphics[width=0.28\textwidth]{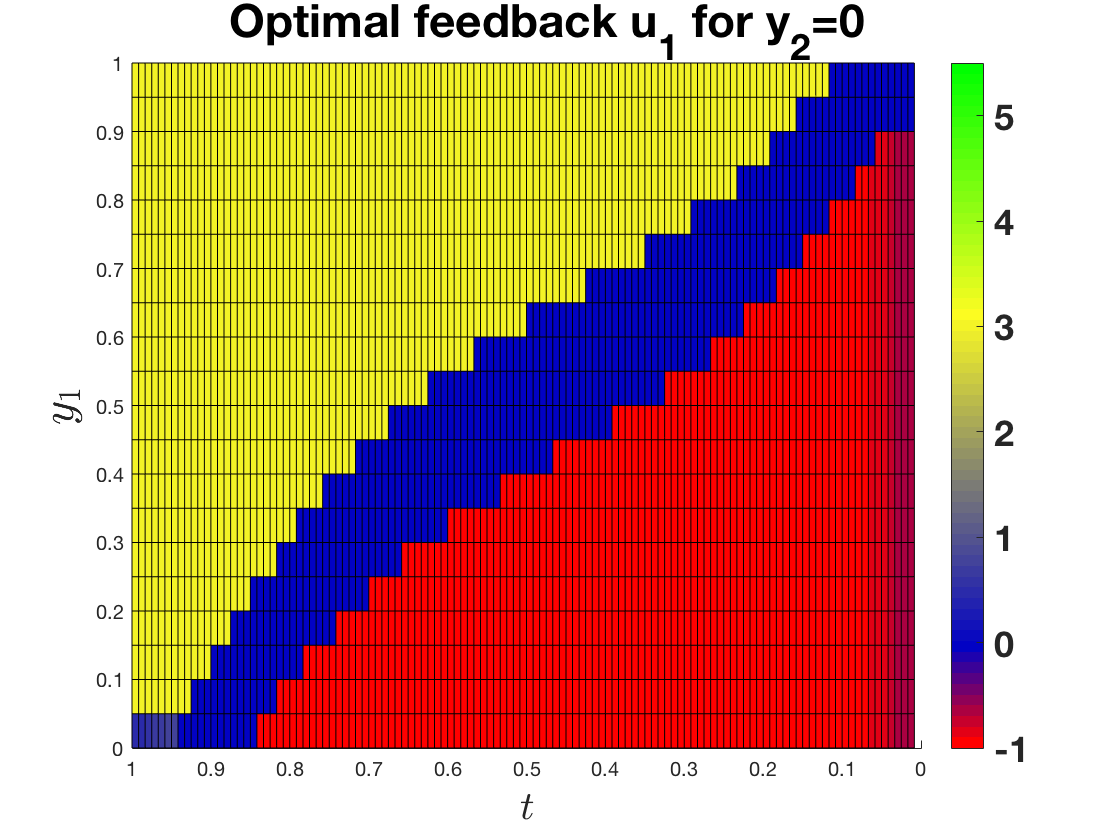}
\includegraphics[width=0.28\textwidth]{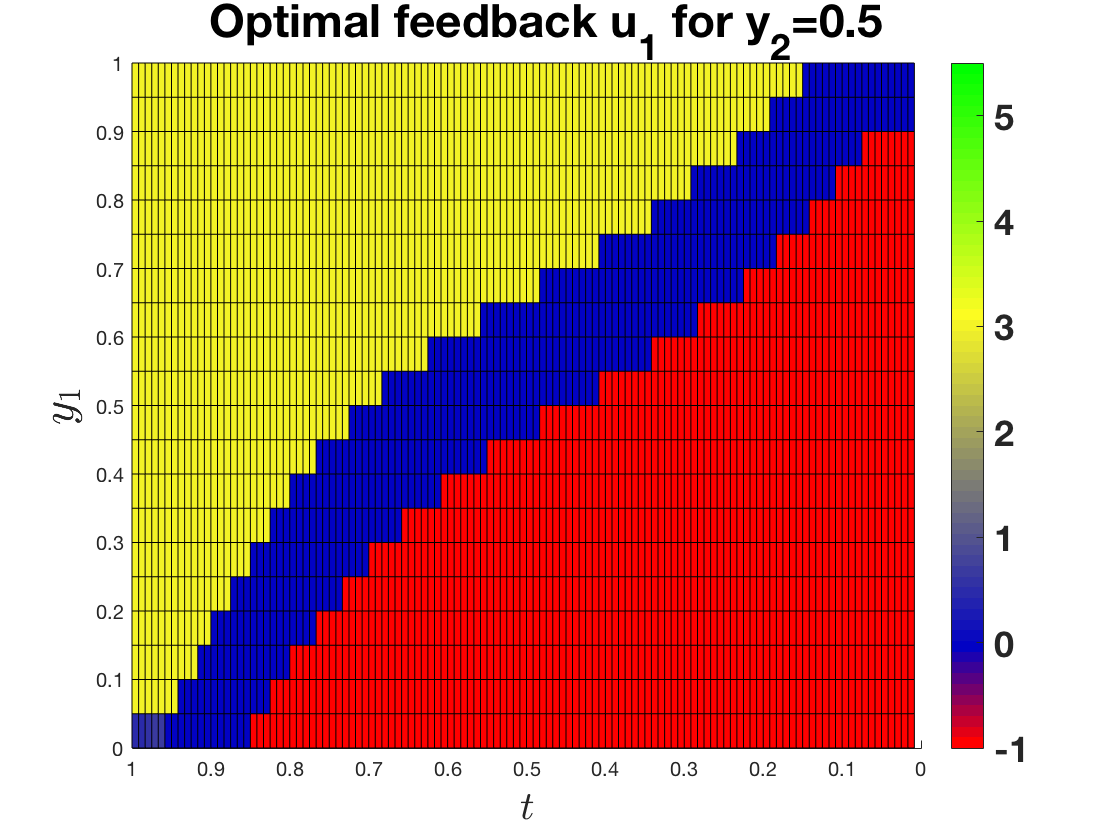}
\includegraphics[width=0.28\textwidth]{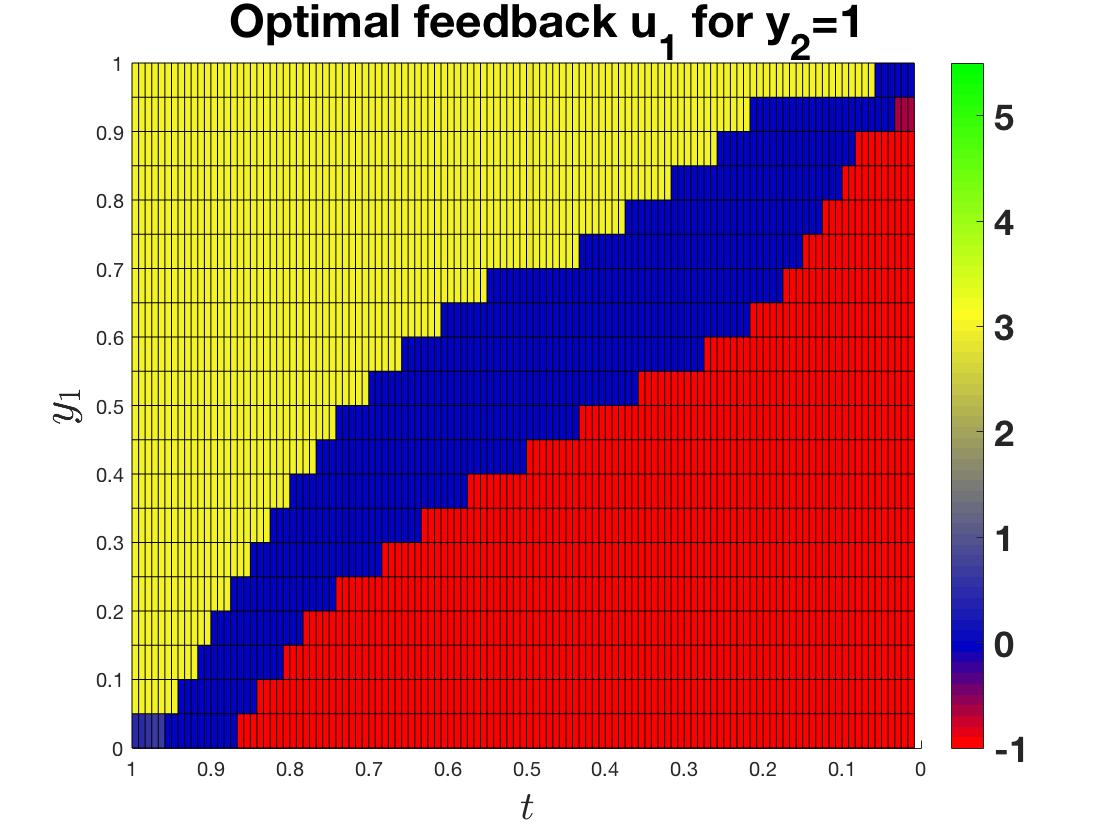}\\
\includegraphics[width=0.28\textwidth]{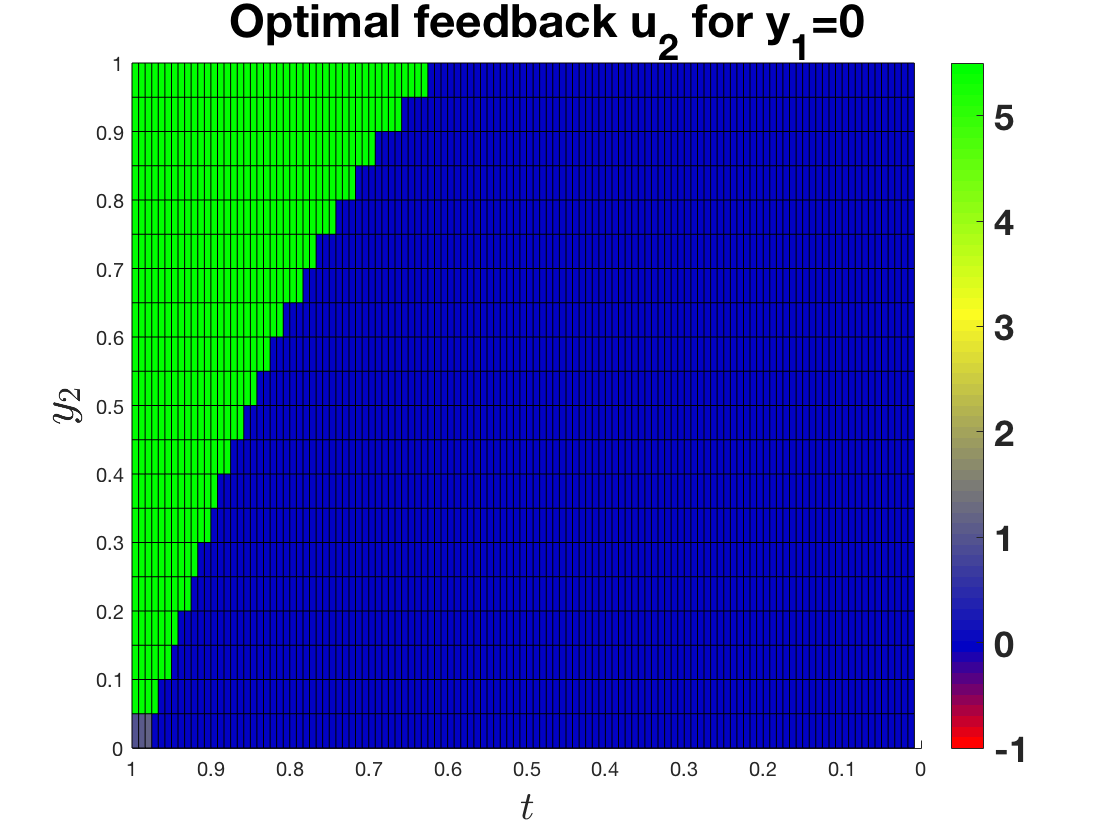}
\includegraphics[width=0.28\textwidth]{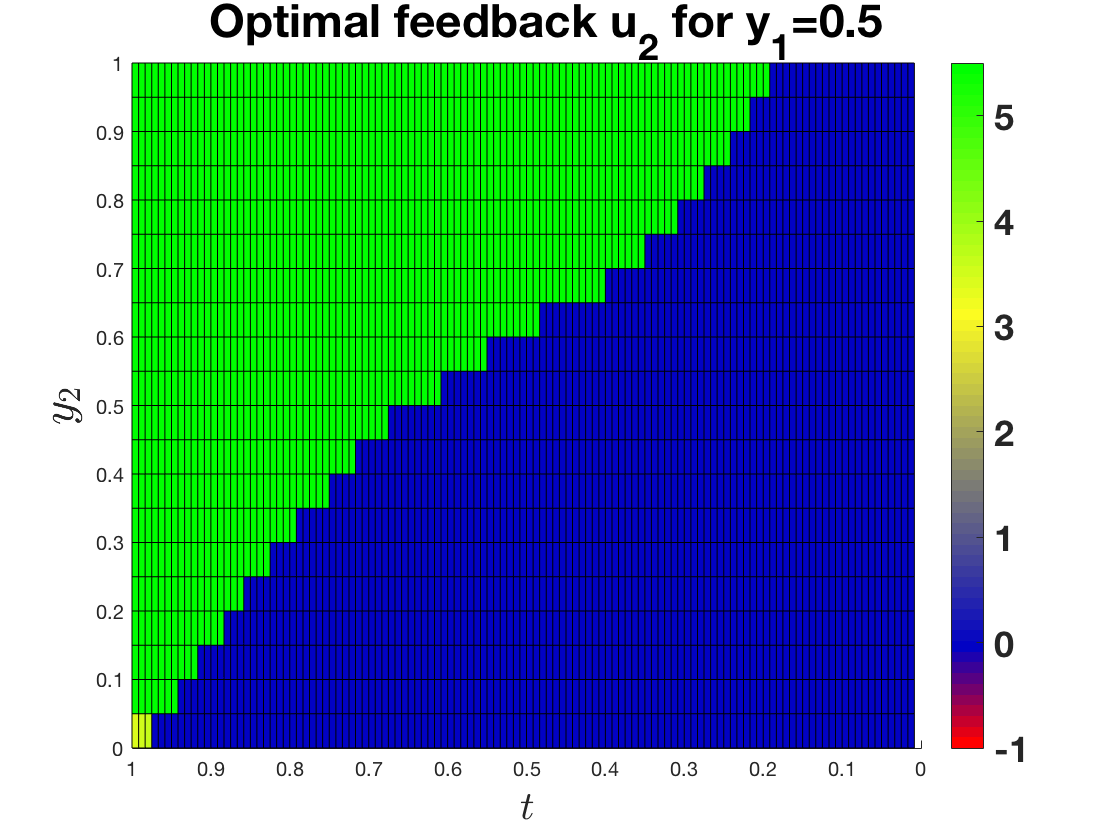}
\includegraphics[width=0.28\textwidth]{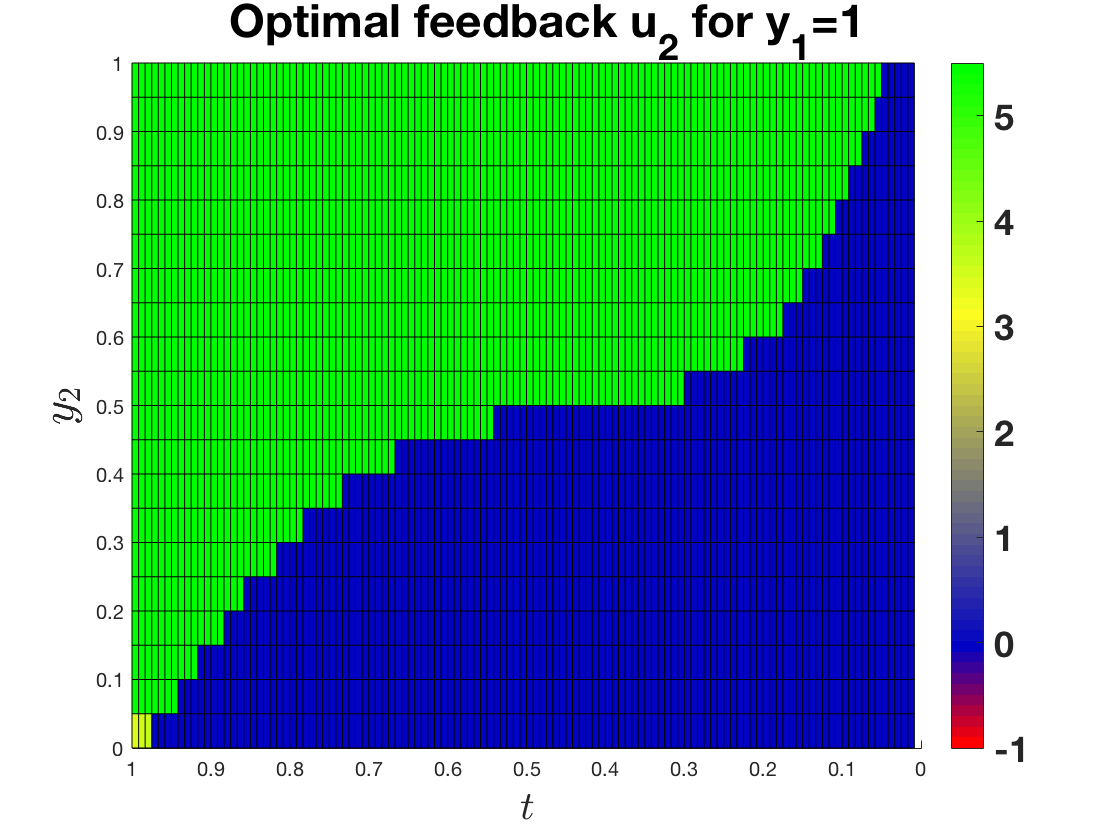}
\end{minipage}\\
\begin{minipage}[h]{0.08\textwidth}\vspace{-5cm}\fbox{$x=4$}\end{minipage}
\begin{minipage}[h]{0.9\textwidth}
\includegraphics[width=0.28\textwidth]{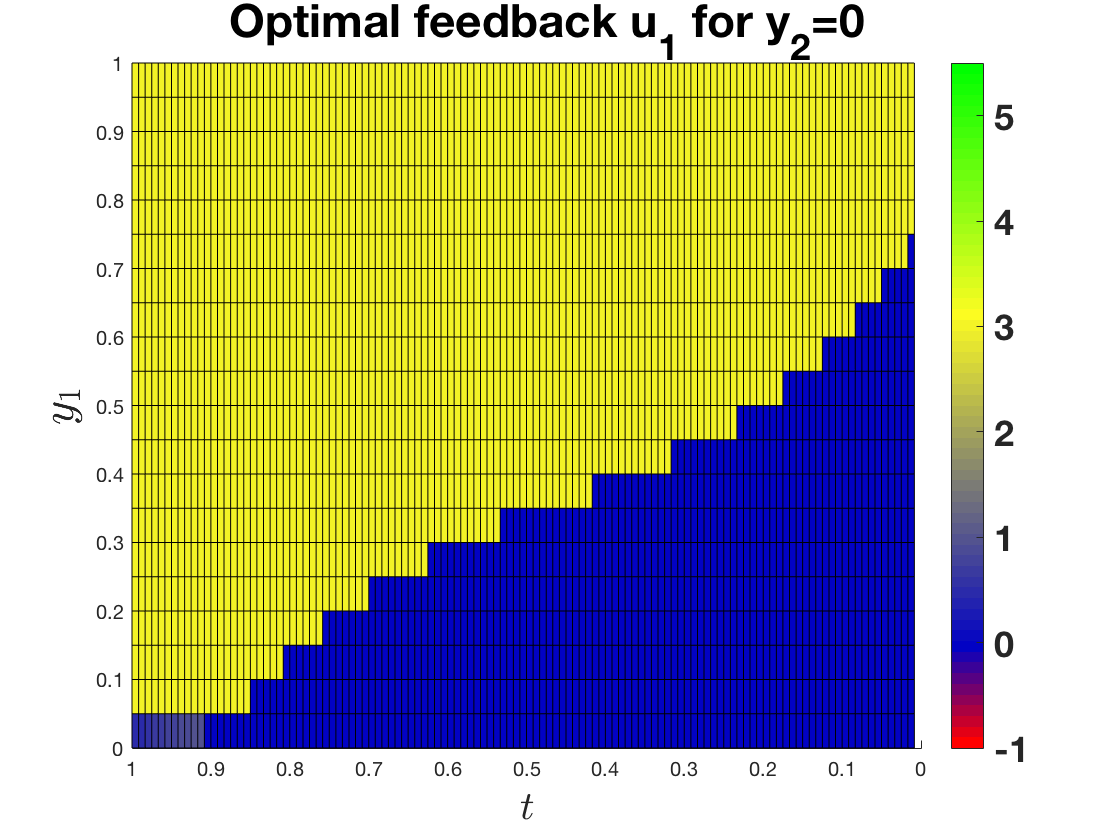}
\includegraphics[width=0.28\textwidth]{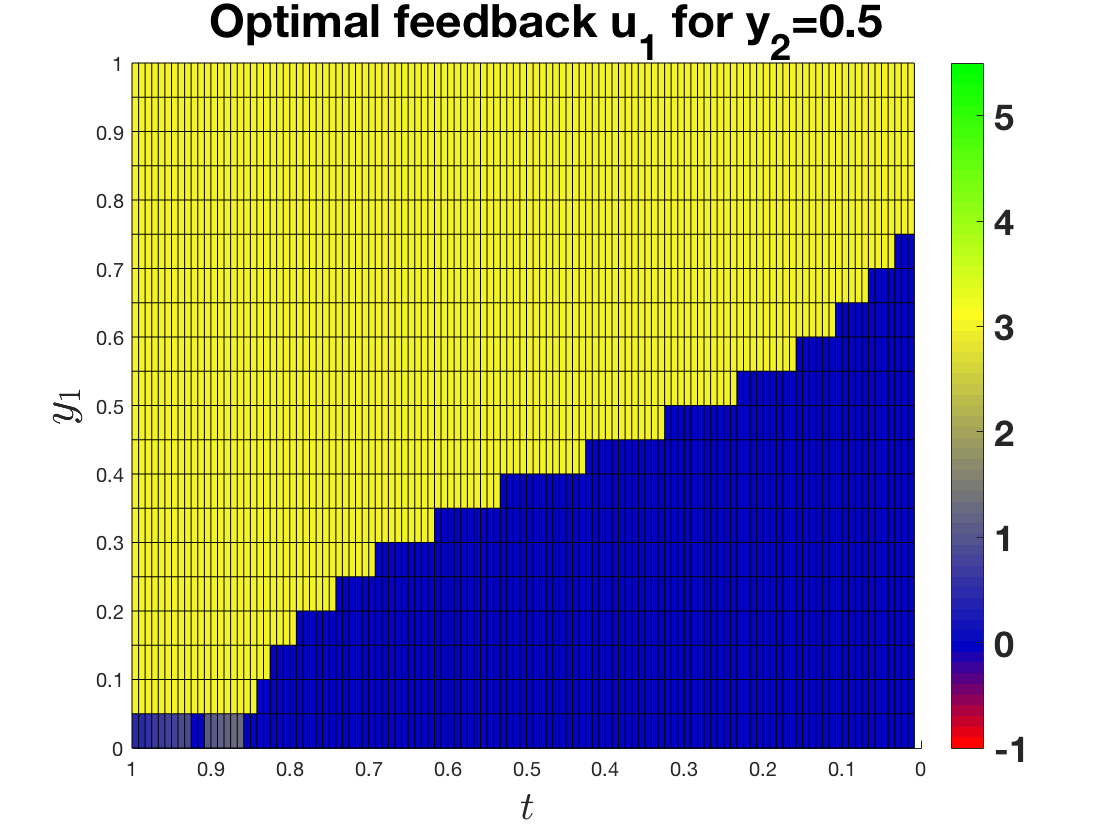}
\includegraphics[width=0.28\textwidth]{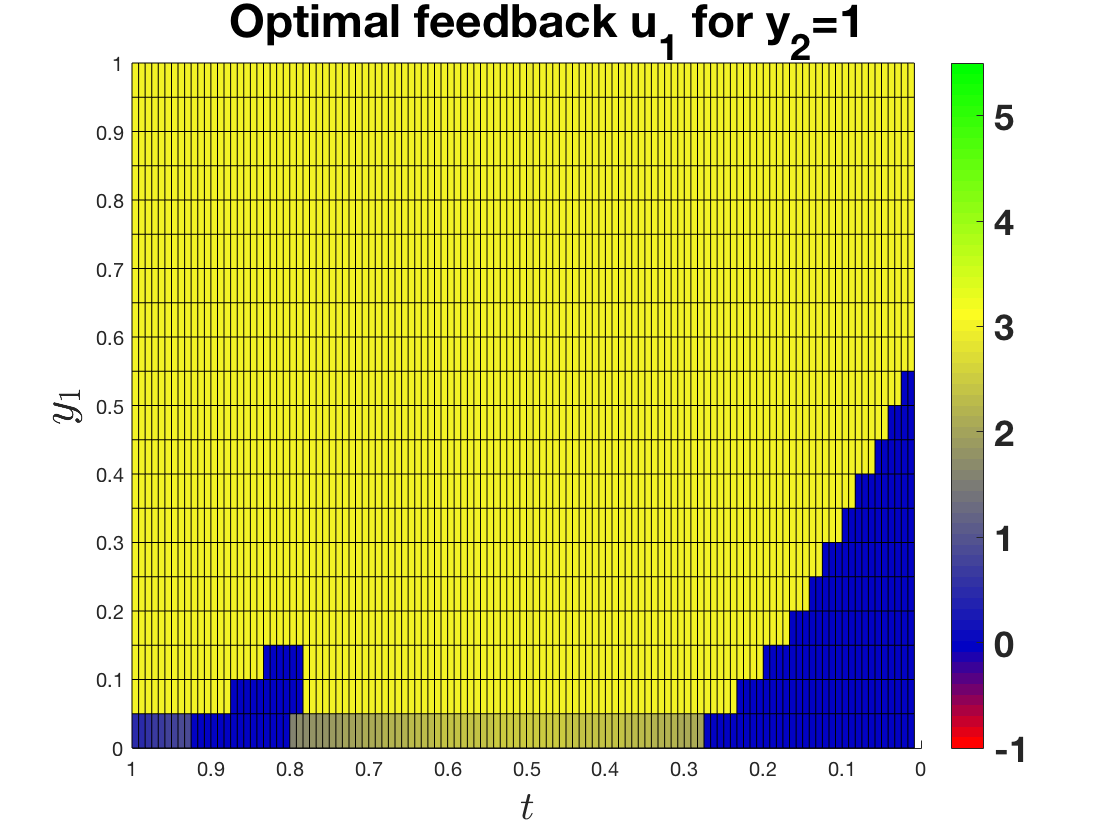}\\
\includegraphics[width=0.28\textwidth]{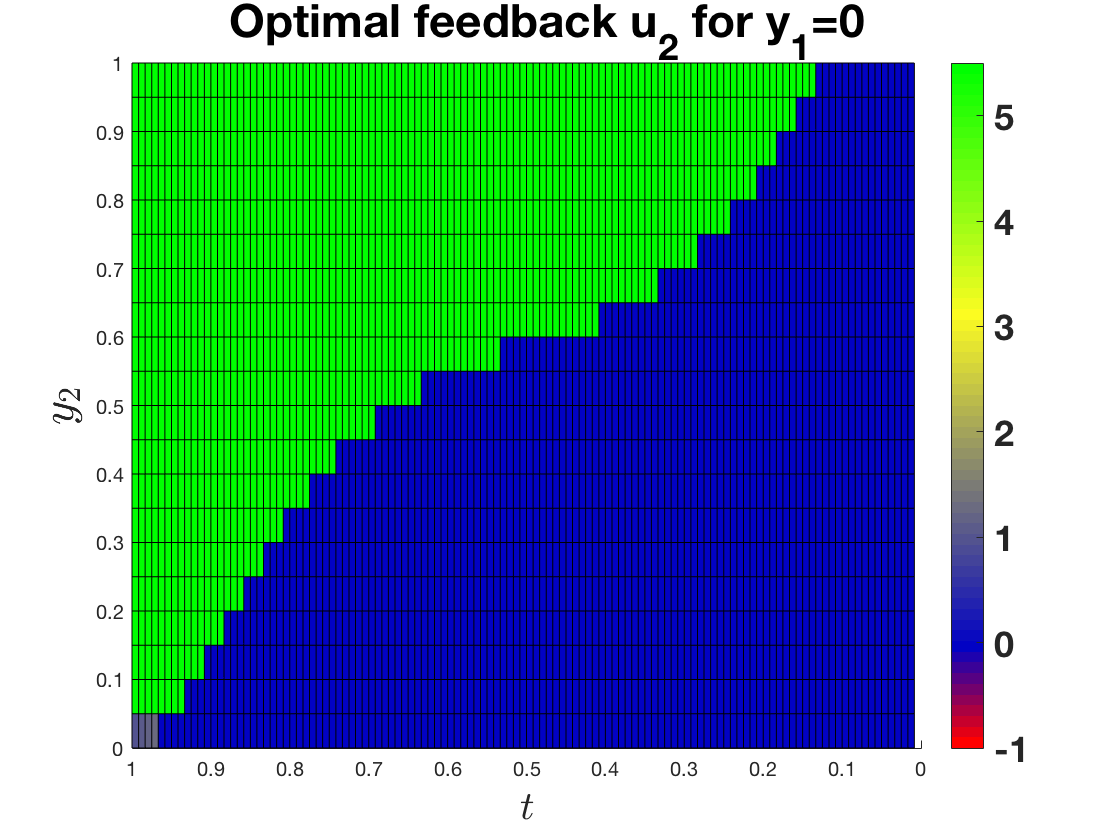}
\includegraphics[width=0.28\textwidth]{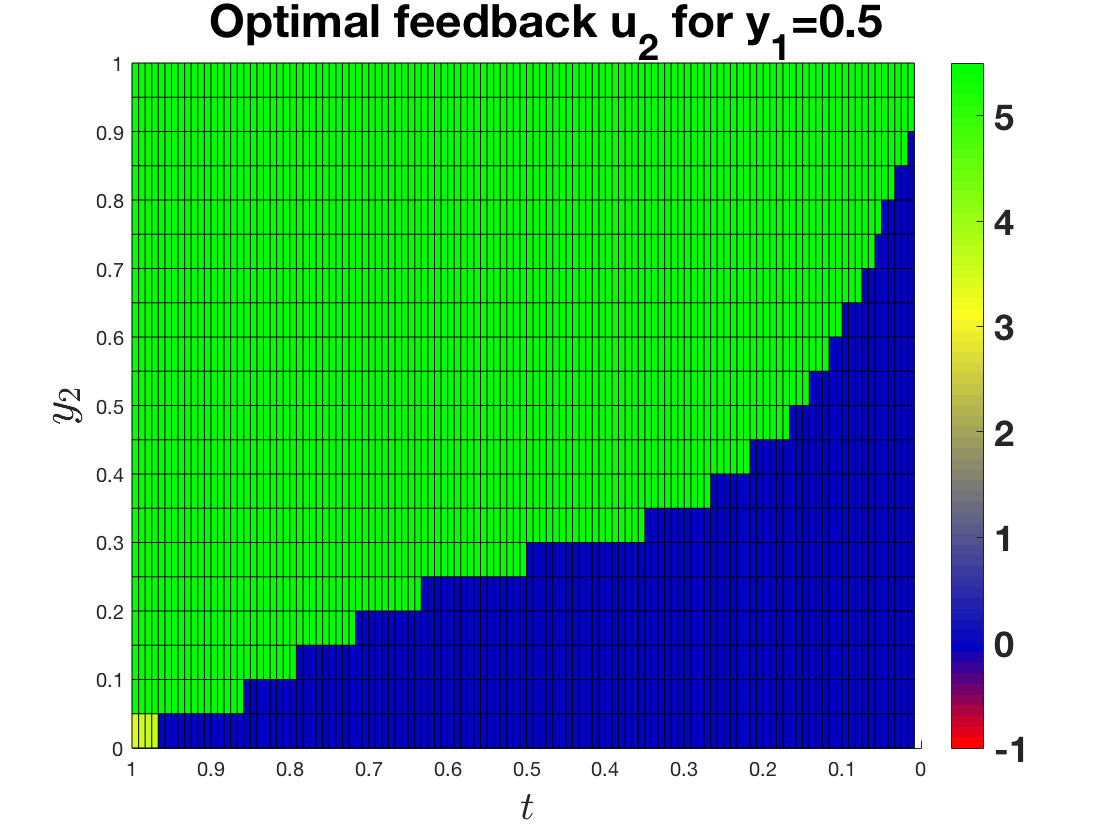}
\includegraphics[width=0.28\textwidth]{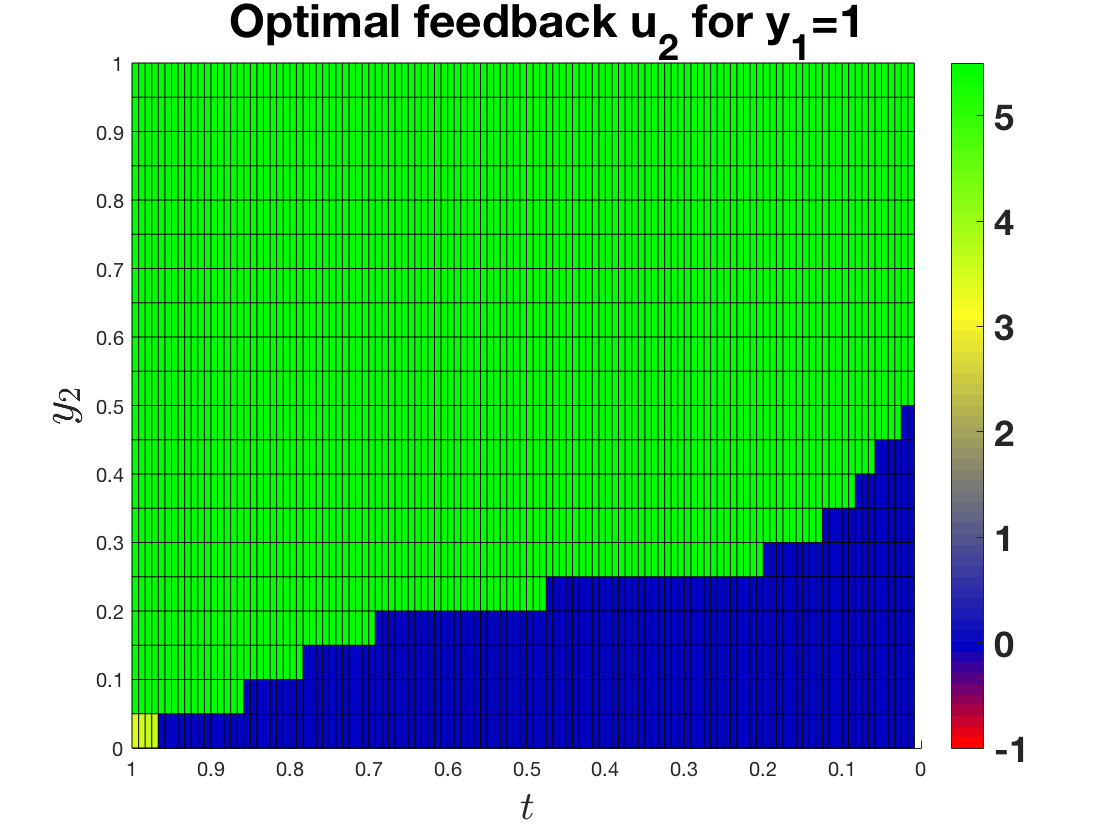}
\end{minipage}\\
%\begin{minipage}[h]{0.08\textwidth}\vspace{-5cm}\fbox{$x=5$}\\ {\color{blue} metto $x=4$ (nicer) o $x=5$??}\end{minipage}
%\begin{minipage}[h]{0.9\textwidth}
%\includegraphics[width=0.28\textwidth]{surf_uu1star_y2_0_x_5_new.png}
%\includegraphics[width=0.28\textwidth]{surf_uu1star_y2_05_x_5_new.png}
%\includegraphics[width=0.28\textwidth]{surf_uu1star_y2_1_x_5_new.png}\\
%\includegraphics[width=0.28\textwidth]{surf_uu2star_y1_0_x_5_new.png}
%\includegraphics[width=0.28\textwidth]{surf_uu2star_y1_05_x_5_new.png}
%\includegraphics[width=0.28\textwidth]{surf_uu2star_y1_1_x_5_new.png}
%\end{minipage}\\
\begin{minipage}[h]{0.08\textwidth}\vspace{-5cm}\fbox{$x=10$} %{\color{blue} ne metterei solo una per $u_1$ e una per $u_2$}
\end{minipage}
\begin{minipage}[h]{0.9\textwidth}
\includegraphics[width=0.28\textwidth]{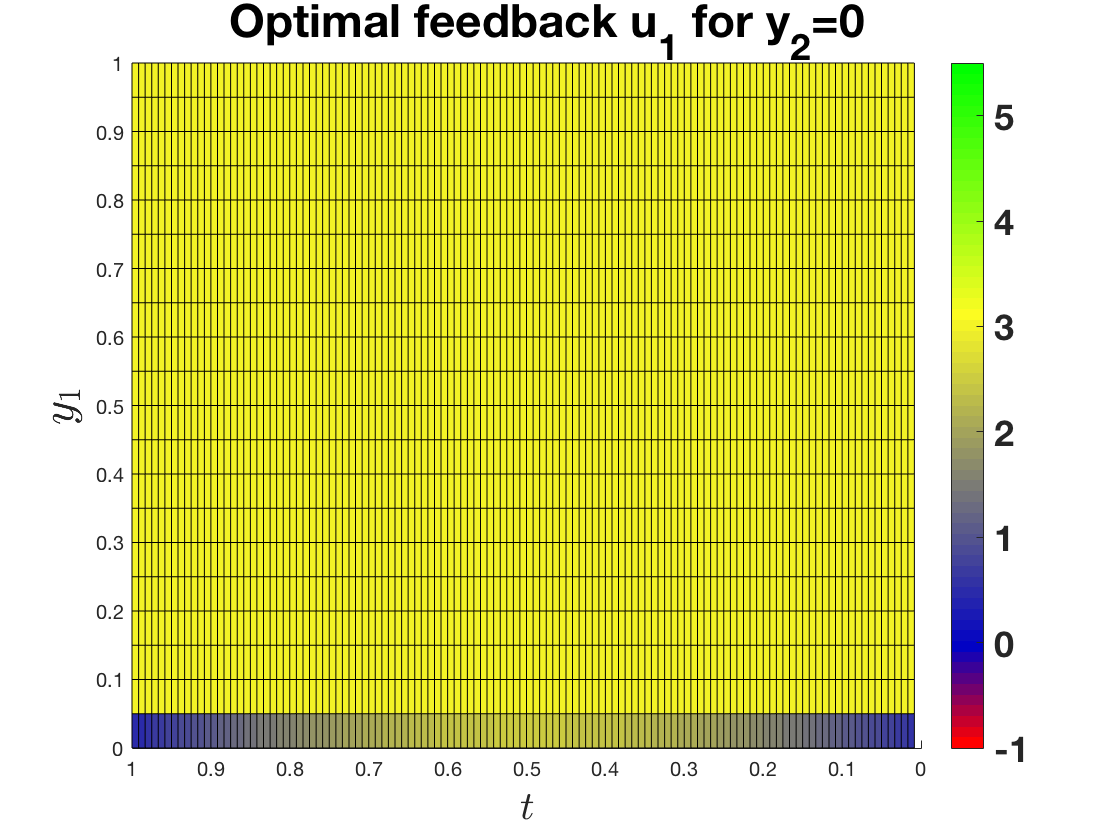}
\includegraphics[width=0.28\textwidth]{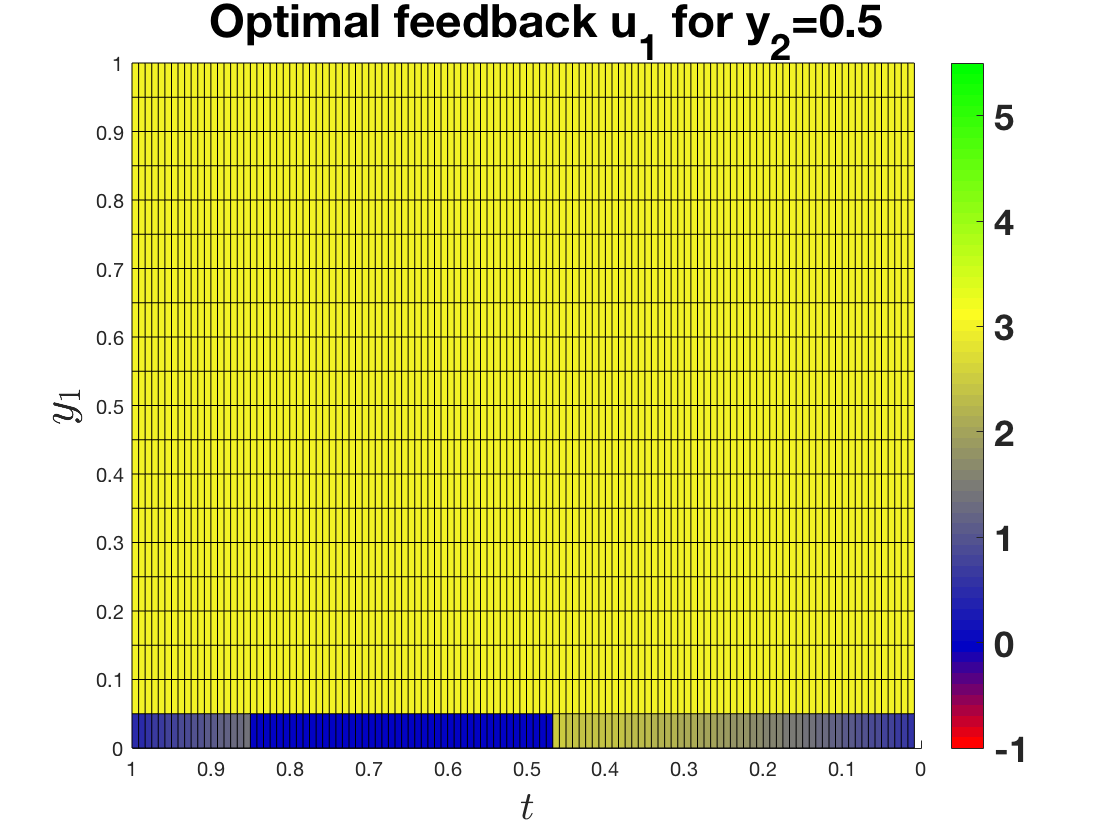}
\includegraphics[width=0.28\textwidth]{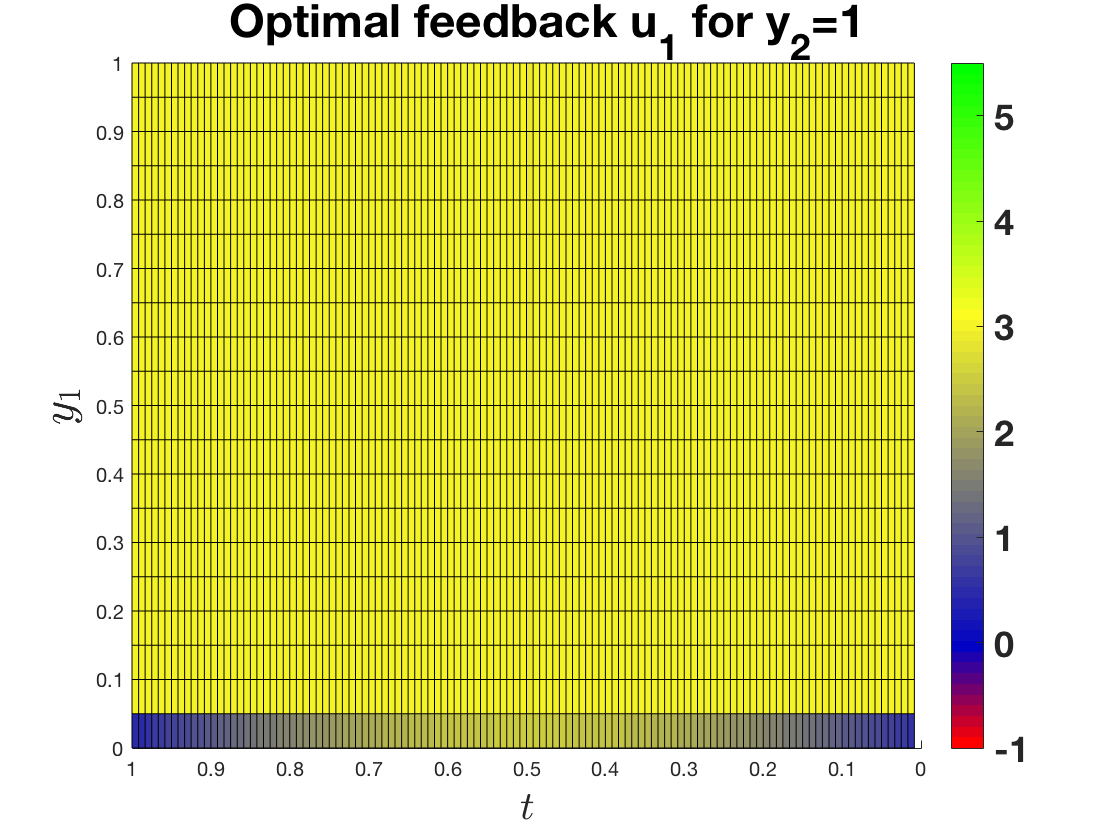}\\
\includegraphics[width=0.28\textwidth]{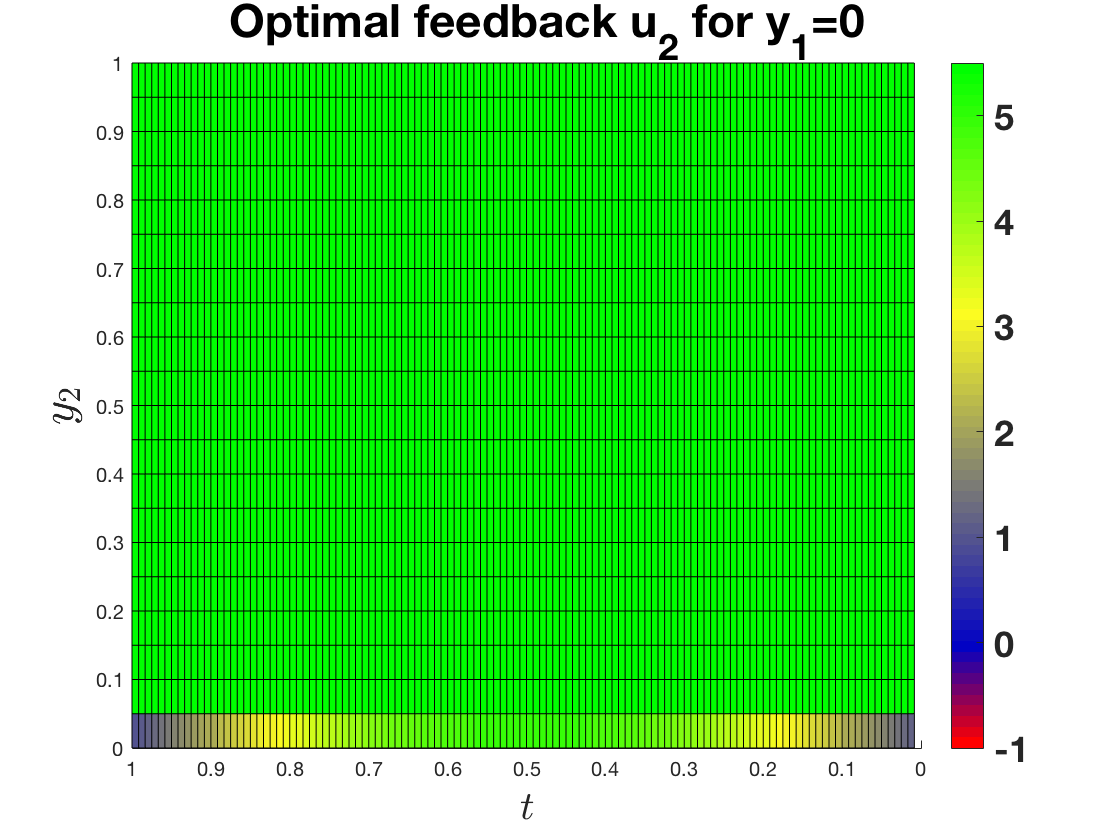}
\includegraphics[width=0.28\textwidth]{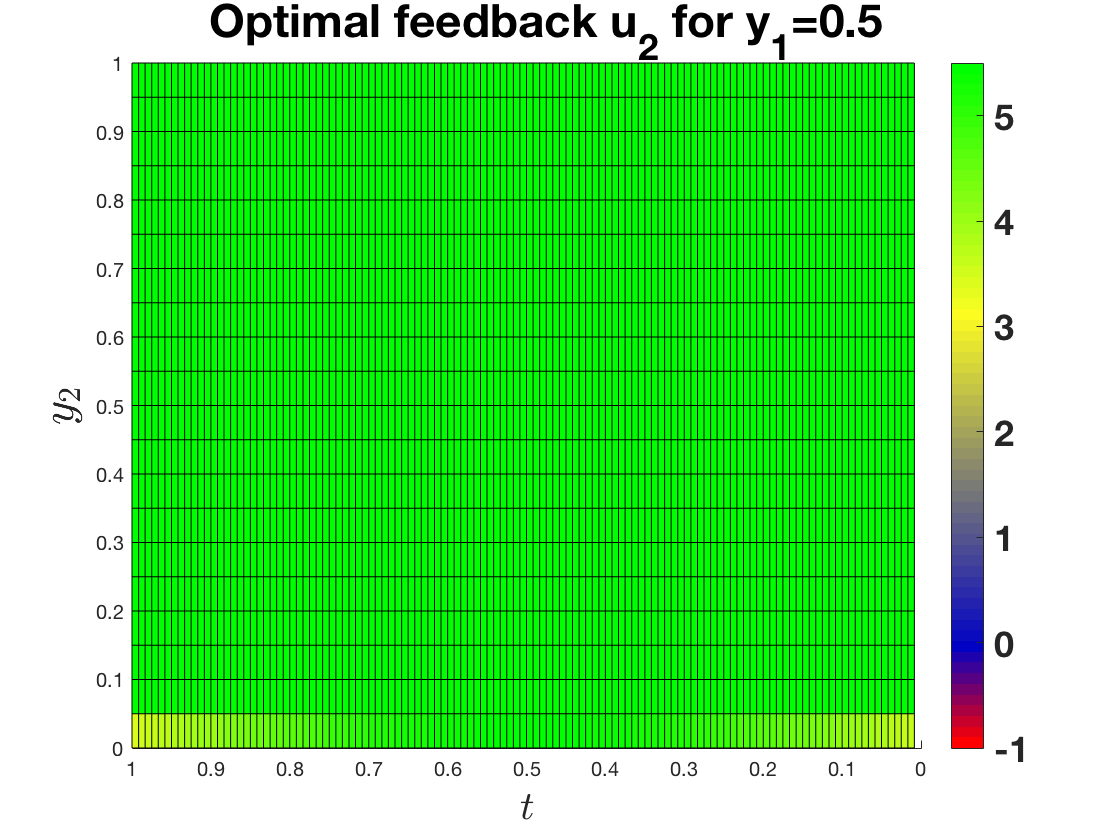}
\includegraphics[width=0.28\textwidth]{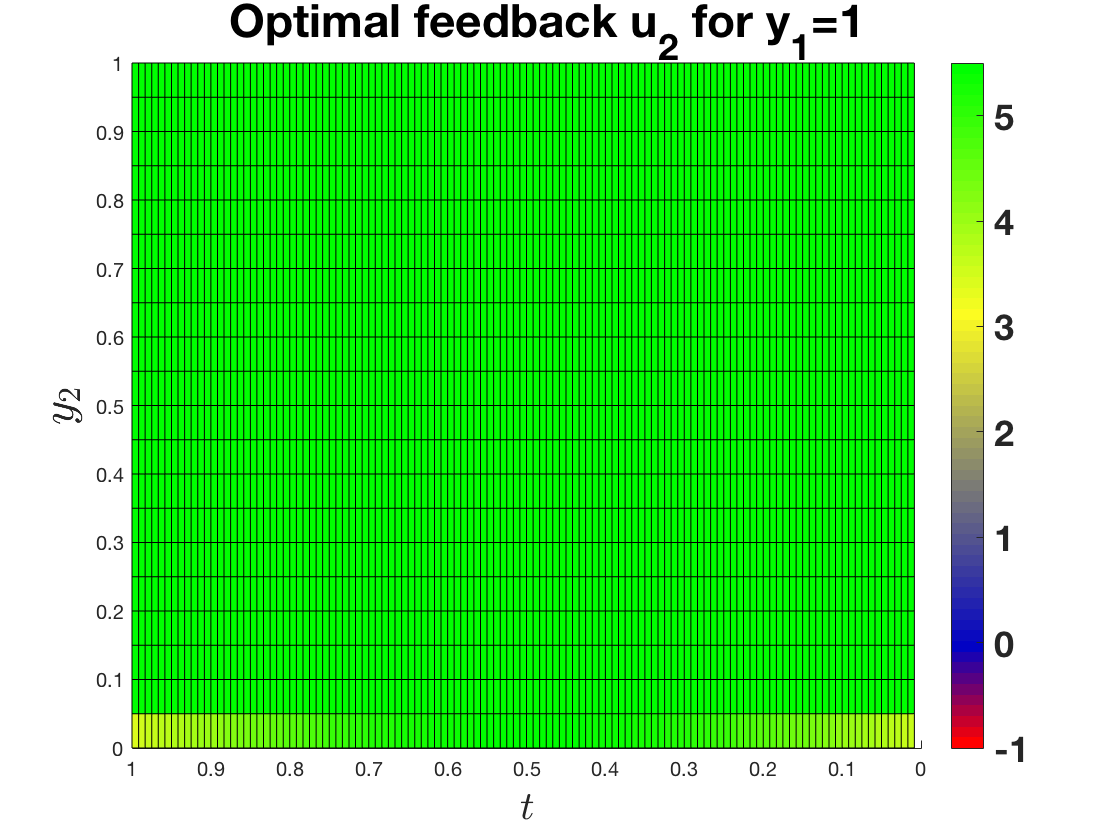}
\end{minipage}\\
\caption{Two dams, price modeled as IGBM with $b(t,x)=5 - x$ and $\sigma(t,x)=0.1 x$, maximal/minimal discharge rate $\underline u_1 = -1, \bar u_1=3$ and $\bar u_2 = 5.5$ (assumption (H3)' is satisfied). Top: numerical approximation of the value function $V$ at $t=0$. From second line to bottom: optimal feedbacks $u_1$ and $u_2$.}
 \label{fig:2dam_priceIGBM}
\end{figure}

\section{State constrained optimal control: controllability issues}\label{sec:issues}

The function $\beta$ is exogenously given and many structural reasons may prevent assumption (H3) to be satisfied. { A typical situation could be that the dam was built in the past with a maximum outflow level $\bar u$ as suitable to a given inflow intensity $\tilde \beta$ which used to be observed in the past, but nowadays, due to climate changes,  the inflow intensity  has changed from $\tilde \beta$ to $\beta$, which could potentially be greater than $\bar u$ for some given times $t \in [0,T]$. As an example, when the inflow intensity was $\tilde \beta(t) = \sin (\pi t) + 0.5$ for $t \in [0,1]$, the dam could have been built with $\bar u = 2$, thus satisfying Assumption (H3). If we now assume that the current inflow intensity is instead $\beta(t) = 2 \sin (\pi t) + 0.5$, we can easily see from Figure \ref{fig:beta} that Assumption (H3) is now violated} 

In order to treat the problem in its full generality we aim in what follows to remove assumption (H3). Considering again the simpler single dam model, let us then assume that $U=[0,\bar u]$ with possibly $\bar u < \max_{t\in [0,T]} \beta(t)$ (we keep $\beta(t)\geq 0, \forall t\in [0,T]$, for simplicity).
The first difficulty one faces in this scenario is to determine the maximal ``controllable region'', i.e. the set $\mathcal D_t\subseteq K$ such that 
$$
\Uad \neq \emptyset \Leftrightarrow y\in \mathcal D_t.
$$
Roughly  speaking $\mathcal D_t$ is the biggest subset of $K$ where at time $t$ the value function is well defined. %finite. 
Let us assume the function $\beta$ is { qualitatively as the one reported in Figure \ref{fig:beta}, i.e. such that} $\beta(t)\geq \bar u$ in the interval $[t^*, T^*]$, for some $0\leq t^*\leq  T^*\leq T$. 
%Notice that  $T^*=0$ means that assumption (H3) is satisfied. 

\begin{prop}\label{prop:domain}
Let
$$
\hat y_t := \bar y - \max\left( \int^{T^*\vee t}_t (\beta(s)-\bar u)\mathrm d s, 0\right).  
$$
 Then, for any $t\in [0,T]$ one has 
$$
\Uad \neq\emptyset \Leftrightarrow y\in [0,\hat y_t].
$$
\end{prop}

\begin{proof}
The result is straightforward if  $t\geq T^*$. Let $t^*\leq t < T^*$.  If there exists a control $u\in \mathcal U$ such that $Y^{t,y,u}_s\in [0,\bar y], \forall s\in [t,T]$ a.s., this certainly implies (taking $s=t$) that $y\in [0,\ymax]$. Moreover, taking $s=T^*$ (recalling we are in the case $t<T^*$)   it also gives
$$
y  \leq \ymax  -  \int^{T^*}_t (\beta(r) - u_r) \mathrm d r \leq \ymax  -  \int^{T^*}_t (\beta(r) - \bar u) \mathrm d r.
$$
On the other hand, let $y\in [0,\hat y_t]$: defining the control 
$$
\hat u_r := \bar u\mathbbm 1_{[t,\tau)\cup [t^*,T^*)} + \beta(r) \mathbbm 1_{[\tau,t^*)\cup [T^*,T]} \quad a.s., 
$$
where $\tau = \inf\{r\in [t, t^*] : Y^{t,y,\bar u}_r <0 \}\wedge t^*$,
one has $\hat u\in \mathcal U$ (because $\beta(r)\in U$ for $r\notin [t^*,T^*]$) and $Y^{t,y,\hat u}_s\in K \,\forall s\in [t,T]$ a.s. thanks to the construction of $\hat u$.
%Let us assume $t<t^*$ (the case $\geq t^*$ is easier). For $s\in [t,T]$
%\begin{align*}
%Y^{t,y,\hat u}_s & = y + \int^{s\wedge T^*}_t (\beta(r) -\bar u) \mathrm d r  \leq \ymax -  \int^{T^*}_{s\wedge T^*} (\beta(r)-\bar u)\mathrm d r
% \leq \ymax\\
%Y^{t,y,\hat u}_s & =  y + \int^{s\wedge T^*}_t (\beta(r) -\bar u) \mathrm d r \geq 0.
%\end{align*}

\end{proof}

\begin{figure}
\includegraphics[width=0.45\textwidth]{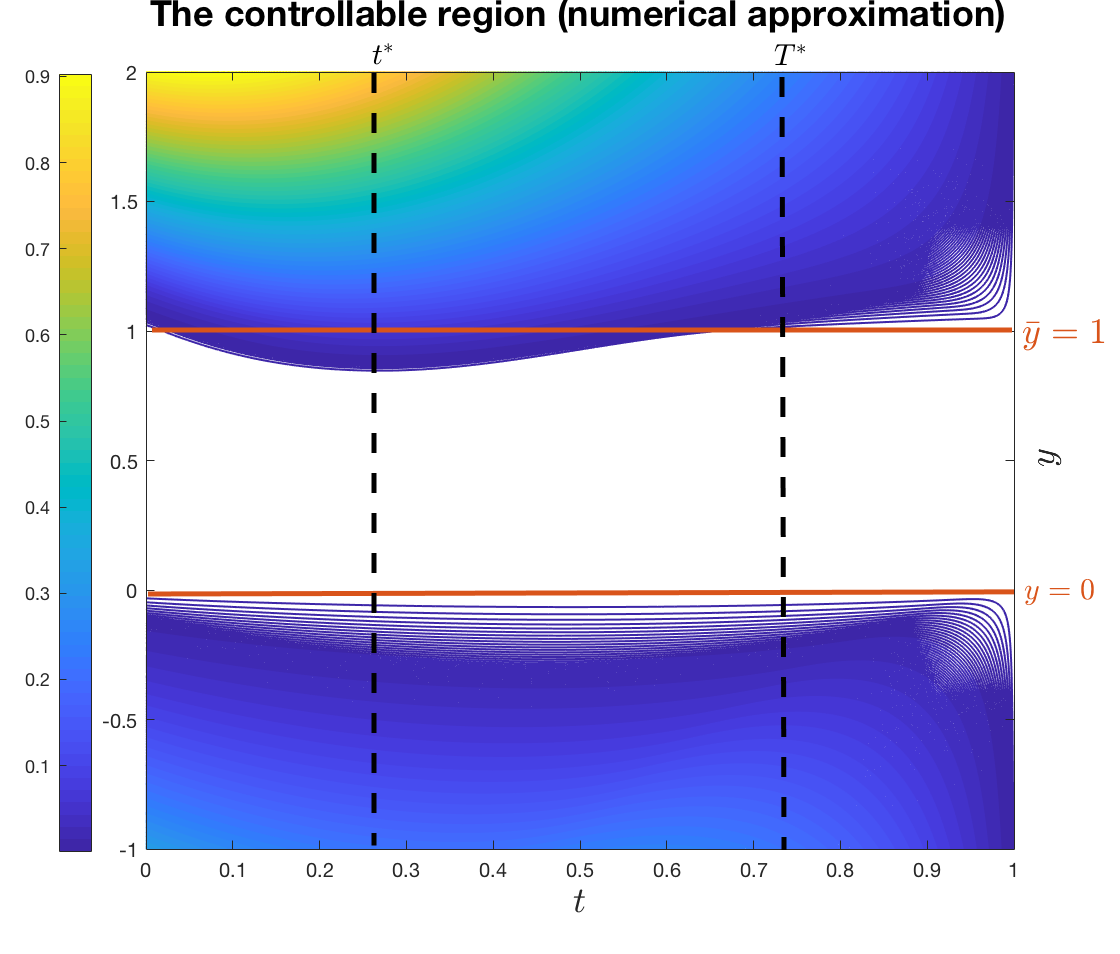}
\includegraphics[width=0.45\textwidth]{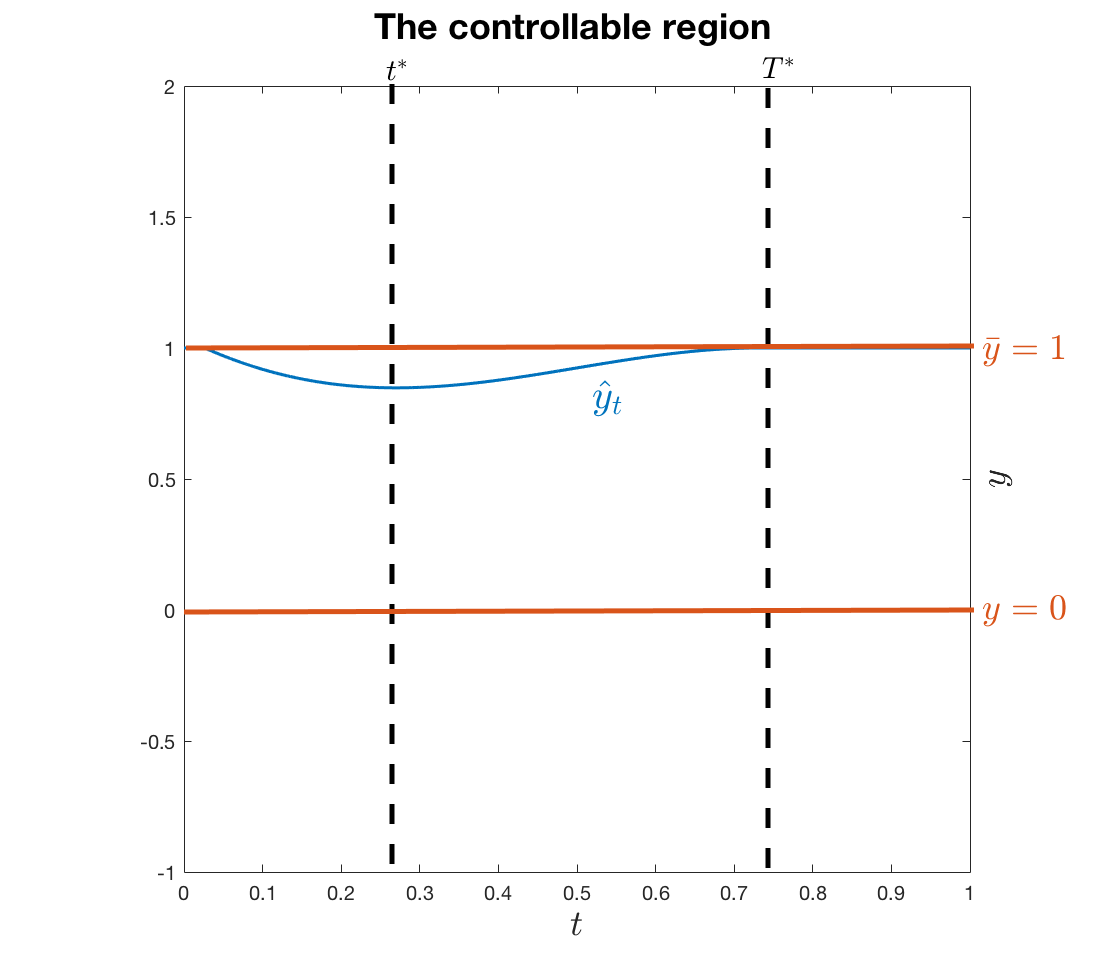}
\caption{The set $\{\mathcal D_t, t\in [0,T]\}$ with $\beta$ as in Figure \ref{fig:beta}. Left: level sets of the function $\vartheta$ (the white region corresponds to the set of level zero). Right: plot of the function $\hat y_t$ explicitly defining $\mathcal D_t$ in Proposition \ref{prop:domain}.}\label{fig:domain}
\end{figure}

Finding explicitly the sets $\{\mathcal D_t, t\in [0,T]\}$  can be in general very complex. An alternative characterization of $\{\mathcal D_t, t\in [0,T]\}$ can be obtained using the so called \textit{level set approach}, see e.g. \cite{BCS14, BFZ10, BPZ15, FGL94, KV06}. Indeed, it is possible to define a function  $\vartheta:[0,T]\times \R\to\R$ such that for any $t\in [0,T]$
\be\label{eq:level}
y\in \mathcal D_t \Leftrightarrow \vartheta(t,y)=0.
\ee
 A possible choice of $\vartheta$ is
$$
\vartheta (t,y) = \inf_{u\in \mathcal U} \int^T_{t}  d_K(Y^{t,y,u}_s) \mathrm d s,
$$
where  $d_K(y)$ denotes the (positive) distance of $y$ to the set $K$.
%
% In alternative, it is possible to characterize it as the level set of a suitable function. For instance we can consider $\vartheta:[0,T]\times \R\to\R$ defined by 
%$$
%\vartheta (t,y) = \inf_{u\in \mathcal U} \int^T_{t}  d_K(Y^{t,y,u}_s) \mathrm d s,
%$$
%where  $d_K(y)$ denotes the (positive) distance of $y$ to the set $K$. 
%It is possible to show that (see e.g. \cite{BCS14, BFZ10, BPZ15, FGL94, KV06}) for any $t\in [0,T]$
%\be\label{eq:level}
%y\in \mathcal D_t \Leftrightarrow \vartheta(t,y)=0.
%\ee
Now, $\vartheta$ is the value function of an unconstrained optimal control problem  fully characterized by a first order HJB equation.
Figure \ref{fig:domain} (left) shows the level sets of the numerical approximation of the function $\vartheta(t,\cdot)$ for $t\in [0,T]$. In particular, the white region approximates the set of level zero and therefore represents a numerical approximation of the region $\{\mathcal D_t, t\in [0,T]\}$. Of course, the choice of $\vartheta$ is not unique and many different functions may provide the characterization given by \eqref{eq:level} (see for instance the discussion in \cite{BCS14}).

Once defined the set $\mathcal D_t$, set which is in general not trivial, as we have seen, one has to deal with an optimal control problem settled in a time-dependent domain. For a problem of this form, the regularity properties of the value function on its domain of definition, its HJB characterization as well as its numerical approximation are very complex and non standard in optimal control theory.  This motivates the study presented in the next section where, to avoid the direct treatment of the state constrained optimal control problem by dynamic programming techniques, we pass by a suitable reformulation of the problem that eliminates this difficulty.

\section{Treatment of the state constraints without controllability assumptions}\label{sec:2dam}
Let us go back to the  model with two reservoirs. Motivated by the discussion in the previous section, we follow here the alternative approach presented in \cite{BPZ16} to overcome the issues related to the treatment of state constraints when assumptions such as (H3) (and (H3)') are not satisfied. 
\\

Let us start introducing the following auxiliary \textit{unconstrained} optimal control problem: 
\be\label{eq:W}
W(t,x,y,z):=\inf_{(u,\alpha)\in\mathcal U\times L^2_{\mathbb F}}\; \mathbb E\bigg[\max\Big(Z^{t,x,z,u,\alpha}_T, 0\Big)+\int^T_t d_K(Y^{t,y,u}_s)\mathrm{d} s\bigg],
\ee
where $L^2_{\mathbb F}\equiv L^2_{\mathbb F}([t,T];\mathbb R)$ is the set of adapted processes $\alpha$ such that  $\|\alpha\|_{L^2_{\mathbb F}}:= \mathbb E[\int^T_{t} |\alpha_s|^2\mathrm d s]<+\infty$  and  
$$
Z^{t,x,z,u,\alpha}_\cdot:=z-\int^\cdot_t \tilde L(X^{t,x}_s,u_s)\mathrm d s+\int^\cdot_t \alpha_s \mathrm d B_s
$$
for $\tilde L(x,u) := x(\kappa(u)- \bar \kappa)$ with  $\bar \kappa:= \max_{u\in U} \kappa(u)$.

We refer to $W$ as the \textit{level set function} for problem \eqref{eq:def_V}.

 The following theorem states the fundamental link between the original optimal control problem \eqref{eq:def_V} and the auxiliary problem \eqref{eq:W}.
 
\begin{theo}\label{teo:main}
The following holds 
\be\label{eq:mainV}
V(t,x,y)=\sup \bigg\{z\leq 0: W(t,x,y,z)=0\bigg\} + G(t,x), 
\ee
with $ G(t,x):= \bar \kappa \,\mathbb E\left[\int^T_t X^{t,x}_s\mathrm{d} s\right]$.
\end{theo}

\begin{proof}
We start proving that 
$$
V(t,x,y)\leq \sup \bigg\{z\leq 0: W(t,x,y,z)=0\bigg\} + G(t,x).
$$
If $\Uad=\emptyset$ there is nothing to prove. Let us then assume that  $\Uad\neq \emptyset$. For any  $u\in \Uad$, {define
$$ z^u := \mathbb E\left[ \int^T_t  L(X^{t,x}_s,u_s) \mathrm d  s\right] - G(t,x) = \mathbb E\left[ \int^T_t \tilde L(X^{t,x}_s,u_s) \mathrm d  s\right] \leq 0 $$
Since  $\int^T_t \tilde L(X^{t,x}_s,u_s)\mathrm d s  \in L^2_{\mathbb F}$, }
by the martingale representation theorem there exists $\alpha \in  L^2_{\mathbb F}$ such that a.s.
$$
z^u = \int^T_t \tilde L(X^{t,x}_s,u_s) \mathrm d  s-\int^T_t \alpha_s \mathrm d B_s.
$$
This implies that 
$$
0\leq W(t,x,y,z^u)\leq \mathbb E\bigg[\max\Big(Z^{t,x,z^u,u,\alpha}_T, 0\Big)+\int^T_t d_K(Y^{t,y,u}_s)\mathrm{d} s\bigg] =0
$$
which gives the following inclusion 
$$
\bigg\{ { z \leq 0 : z = \mathbb E\Big[ \int^T_t  L(X^{t,x}_s,u_s) \mathrm d  s\Big]  -G(t,x)}, u\in \Uad\bigg\} \subseteq\bigg\{z\leq 0: W(t,x,y,z)=0\bigg\} 
$$
from which the desired inequality follows.

We now prove that 
$$
V(t,x,y)\geq \sup \bigg\{z\leq 0: W(t,x,y,z)=0\bigg\} + G(t,x).
$$

Let $z\leq 0$ such that $W(t,x,y,z)=0$ (if such $z$ does not exists there is nothing to prove). 
Let $(u^k,\alpha^k)\in \mathcal U\times  L^2_{\mathbb F}$ be a minimizing sequence for $W(t,x,y,z)$. Therefore for any $\varepsilon>0$  there exists $k_0$ such that $\forall k\geq k_0$ one has 
$$
 \mathbb E\bigg[\max\Big(Z^{t,x,z,u^k,\alpha^k}_T, 0\Big)+\int^T_t d_K(Y^{t,y,u^k}_s)ds\bigg] \leq W(t,x,y,z) + \varepsilon    = \varepsilon.
$$
Being $u^k$ uniformly bounded in norm $L^2_{\mathbb F}$ (because $u$ takes values in a compact set), there exists a weakly convergent subsequence (for simplicity still indexed by $k$), i.e.
$$
u^k\to \hat u\qquad \text{weakly in $L^2_{\mathbb F}$},
$$
 for some $\hat u\in \mathcal U$.
Applying Mazur's theorem one has that there exists $\tilde u^k= \sum_{i\geq 0} \lambda_i u^{i+k}$ with $\lambda_i\geq 0$ and $\sum_{i\geq 0}\lambda_i =1$ such that 
$$
 \tilde u^k\to \hat u\qquad \text{strongly in $L^2_{\mathbb F}$}.
$$
We then consider $(\tilde u^k, \tilde \alpha^k)\equiv \sum_{i\geq 0} \lambda_i (u^{i+k},\alpha^{i+k})$. Observe that $(\tilde u^k, \tilde \alpha^k)$ still belongs to $\mathcal U\times  L^2_{\mathbb F}$ because $U$ is convex and $ L^2_{\mathbb F}$ is a convex space.  
 Let us now consider $(Z^{\tilde u^k,\tilde \alpha^k}, Y^{\tilde u^k})$. One has 
 \begin{align*}
 \mathbb E\left[ \left| Z^{t,x,z,\hat u,\tilde \alpha^k}_T - Z^{t,x,z,\tilde u^k,\tilde \alpha^k}_T\right|^2\right] & = \mathbb E\Big[\Big|\int^T_t X^{t,x}_s\big( \tilde u^k_{2,s} -\hat u_{2,s} + c(\tilde u^k_{1,s}) - c( \hat u_{1,s})\big) \mathrm d s \Big|^2\Big]\\
  & \leq T \mathbb E\Big[\int^T_t \Big|X^{t,x}_s\big( \tilde u^k_{2,s} -\hat u_{2,s} + c(\tilde u^k_{1,s}) - c( \hat u_{1,s})\big) \Big|^2\mathrm d s  \Big]\\
 & \leq T \mathbb E\Big[\int^T_t (X^{t,x}_s)^2 \big( 2|\tilde u^k_{2,s} -\hat u_{2,s} |^2 + 2|c(\tilde u^k_{1,s}) - c( \hat u_{1,s})|^2\big)\mathrm d s  \Big]\\
  & \leq 2 T \mathbb E\Big[\int^T_t (X^{t,x}_s)^2 \big( |\tilde u^k_{2,s} -\hat u_{2,s} |^2 + L^2_c|\tilde u^k_{1,s} -  \hat u_{1,s}|^2\big)\mathrm d s  \Big].
 \end{align*}
{ where $L_c (= \gamma)$ is the Lipschitz constant of the function $c$.} 
 Moreover, from the strong convergence in $L^2_{\mathbb F}$-norm it follows that there exists a subsequence $\tilde u^{k_n}$ such that $|\tilde u^{k_n} - \hat u|^2\to 0$ a.e.. 
 One also has 
 $$
  (X^{t,x}_s)^2 \left( 2|\tilde u^k_{2,s} -\hat u_{2,s} |^2 + 2L^2_c|\tilde u^k_{1,s} -  \hat u_{1,s}|^2\right) \leq C X_s^2 \qquad\text{a.s.}
 $$
for some constant $C$ depending on the uniform bound of elements in $\mathcal U$. Being $X_s^2$ integrable under the classical assumption (H1), we can apply the dominate convergence theorem on the subsequence $u^{k_n}$ to get 
$$
\lim_{n\to\infty} \mathbb E\Big[\int^T_t (X^{t,x}_s)^2 \big( |\tilde u^{k_n}_{2,s} -\hat u_{2,s} |^2 + |\tilde u^{k_n}_{1,s} -  \hat u_{1,s}|^2\big)\mathrm d s  \Big] = \mathbb E\Big[\int^T_t \lim_{n\to\infty}(X^{t,x}_s)^2 \big( |\tilde u^{k_n}_{2,s} -\hat u_{2,s} |^2 + |\tilde u^{k_n}_{1,s} -  \hat u_{1,s}|^2\big)\mathrm d s  \Big] = 0.
$$ 
 In conclusion, there exists a suitable subsequence (indexed with $k$ for simplicity) such that  
 $$
  \mathbb E\left[ \left| Z^{t,x,z,\tilde u^k,\tilde \alpha^k}_T - Z^{t,x,z,\hat u,\tilde \alpha^k}_T\right|^2\right] \to 0.
 $$
 Similarly one has 
 $$
  \mathbb E\left[\sup_{s\in [t,T]} \left| Y^{t,y,\tilde u^k}_s - Y^{t,y,\hat u}_s\right|^2\right] \to 0.
 $$
 Then, for any $\varepsilon >0$ there exists $k_1$ such that $\forall k\geq k_1$ one has  
$$
\left| \mathbb E\bigg[\max\Big(Z^{t,x,z,\tilde u^k,\tilde \alpha^k}_T, 0\Big)+\int^T_t d_K(Y^{t,y,\tilde u^k}_s)\mathrm{d} s\bigg] - \mathbb E\bigg[\max\Big(Z^{t,x,z,\hat u,\tilde \alpha^k}_T, 0\Big)+\int^T_t d_K(Y^{t,y,\hat u}_s)\mathrm ds\bigg] \right| \leq \varepsilon.
$$
 Moreover, one has (using the fact that $c(\cdot)$ is concave and $X^{t,x}_s\geq 0$):
 \begin{align*}
 Z^{t,x,z,\tilde u^k,\tilde \alpha^k}_T 
 & = z -\int^T_t X^{t,x}_s\big( \tilde u^k_{2,s} + c(\tilde u^k_{1,s}) -\bar\kappa\big) \mathrm d s + \int^T_t \tilde \alpha^k_{s} \mathrm d B_s \\
 & = z -\int^T_t X^{t,x}_s\Big( \sum_{i\geq 0} \lambda_i u^{i+k}_{2,s} + c\Big(\sum_{i\geq 0} \lambda_i u^{i+k}_{1,s}\Big)  -\bar\kappa\Big) \mathrm d s + \int^T_t \sum_{i\geq 0} \lambda_i  \alpha^{i+k}_s \mathrm d B_s\\
 & \leq z -\int^T_t X^{t,x}_s\Big( \sum_{i\geq 0} \lambda_i u^{i+k}_{2,s} + \sum_{i\geq 0} \lambda_i c(u^{i+k}_{1,s})  -\bar\kappa\Big) \mathrm d s + \int^T_t \sum_{i\geq 0} \lambda_i  \alpha^{i+k}_s \mathrm d B_s\\
 &  = z -  \sum_{i\geq 0} \lambda_i  \int^T_t X^{t,x}_s\Big(u^{i+k}_{2,s} +  c(u^{i+k}_{1,s}) -\bar\kappa\Big) \mathrm d s + \sum_{i\geq 0} \lambda_i \int^T_t   \alpha^{i+k}_s \mathrm d B_s \\
 & = \sum_{i\geq 0} \lambda_i Z^{t,x,z,u^{i+k}, \alpha^{i+k}}_T 
 \end{align*}
 and (being $Y$ linear in $u$)
 $$
 Y^{t,y,\tilde u^k}_s = \sum_{i\geq 0} \lambda_i  Y^{t,y,u^{i+k}}_s.
 $$
 Putting these things together (observing that $\max(z,0)$ is convex and that being $K$ convex also $d_K$ is) one has for $k\geq \max(k_0,k_1)$:
 \begin{align*}
 \mathbb E\bigg[\max\Big(Z^{t,x,z,\hat u,\tilde \alpha^k}_T, 0\Big) & + \int^T_t d_K(Y^{t,y,\hat u}_s)\mathrm d s\bigg] \leq \mathbb E\bigg[\max\Big(Z^{t,x,z,\tilde u^k,\tilde \alpha^k}_T, 0\Big)+\int^T_t d_K(Y^{t,y,\tilde u^k}_s)\mathrm{d}s\bigg] + \varepsilon\\
 & = \mathbb E\bigg[\max\Big(\sum_{i\geq 0} \lambda_i  Z^{t,x,z,u^{i+k},\alpha^{i+k}}_T, 0\Big)+\int^T_t d_K(\sum_{i\geq 0} \lambda_i Y^{t,y,u^{i+k}}_s)\mathrm{d}s\bigg] + \varepsilon\\
 & \leq \sum_{i\geq 0} \lambda_i \mathbb E\bigg[\max\Big(Z^{t,x,z,u^{i+k},\alpha^{i+k}}_T, 0\Big)+\int^T_t d_K(Y^{t,y, u^{i+k}}_s)\mathrm d s\bigg] + \varepsilon\\
 & \leq W(t,x,y,z) +2 \varepsilon\\
 & =2\varepsilon.
 \end{align*}
 From the previous inequality we can immediately conclude that $\hat u\in \Uad$, because for any arbitrary  $\varepsilon>0$ one has $0\leq  \mathbb E\Big[\int^T_t d_K(Y^{t,y,\hat u}_s)\mathrm d s\Big]\leq 2\varepsilon$. Moreover, one also has 
 \begin{align*}
z - \mathbb E\left[ \int^T_t \tilde L(X^{t,x}_s,\hat u_s)\mathrm d s \right] \leq  \mathbb E\bigg[\max\Big(Z^{t,x,z,\hat u,\tilde \alpha^k}_T, 0\Big)+\int^T_t d_K(Y^{t,y,\hat u}_s)\mathrm d s\bigg] & \leq 2\varepsilon
\end{align*}
which gives 
$$
 z \leq \mathbb E\left[ \int^T_t L(X^{t,x}_s,\hat u_s)\mathrm d s \right] +  G(t,x)
$$
and then $z\leq V(t,x,y) - G(t,x)$.

 \end{proof}

%%%%%%%%%%%%%%%%%%%%%%%%%%%%%%%%%%%%%%%%%%%%%%%%%%%
%%%%%%%%%%%% Sketch of Tiziano's proof....not used %%%%%%%%%%%%%%%%%%%%%

%\begin{lem}\label{lem:exist}
%For any $(t,y)$ such that $\Uad \neq \emptyset$ (i.e., if $y\in \mathcal D_t$), 
%%For any $(t,x,y,z)$ (resp. $(t,x,y,p)$)  in $[0,T]\times [0,+\infty) \times \R^2\times \R$, 
%the problem  \eqref{eq:W} (resp.  \eqref{eq:w}) admits an  optimal control for all $z \geq v(t,x,y)$.
%\end{lem}
% \begin{proof}  Since the point $(t,x,y,z)$ is fixed, if not strictly necessary we omit the dependence on it of quantities below.
%
%Let $u^* \in \Uad$ and define $z^* :=  \mathbb E\bigg[ \int_t^T \tilde \ell(X_s,u_s^*)\ ds \bigg] \geq v(t,x,y)$; then, since $\int_t^T \tilde \ell(X_s,U_s^*)\ ds \in L^2(\Omega, \mathcal F, \mathbb P)$, there exists $\alpha^* \in L^2_{\mathbb F}([t,T];\R)$ such that $\int_t^T \tilde \ell(X_s,U_s^*)\ ds - z^* = \int_t^T \alpha^*_s\ dB_s$. Then 
%%
% $$
% \mathbb E\bigg[\max\Big(- Z^{u^*,\alpha^*}_T, 0\Big)+\int^T_t d_K(Y^{u^*}_s)ds\bigg] = 0 = W(t,x,y,z^*)
% $$
%%
%since $W$ obviously cannot be negative. Thus, $(u^*,\alpha^*)$ is an optimal control for the problem  \eqref{eq:W}.  It is also straightforward to see that $W$ is monotonically decreasing with respect to $z$, thus the result holds for all $z \geq z^*$. 
% \end{proof}
 
{
 \begin{cor}
If $W(t,x,y,z) = 0$, there exists an optimal $(\hat u, \hat \alpha)\in \mathcal U \times L^2_{\mathbb F}$ that realizes the infimum in \eqref{eq:W}. 
\end{cor}
\begin{proof}
Given $\hat u\in \Uad$ as in the proof of Theorem \ref{teo:main} and $\hat\alpha  \in L^2_{\mathbb F}$ such that 
 $$
  \int^T_t \tilde L (X^{t,x}_s, \hat u_s)\mathrm d s - \int^T_t \hat \alpha_s \mathrm d B_s =  \mathbb E\left[ \int^T_t \tilde L (X^{t,x}_s, \hat u_s)\mathrm d s \right] \geq z
 $$
 one has 
 $$
 \mathbb E\left[ \max\left(  Z^{t,x,z,\hat u,\hat \alpha}_T, 0\right) + \int^T_t d_K(Y^{t,y,\hat u}_s) \mathrm d s \right] = 0 = W(t,x,y,z).
 $$
 \end{proof}
 }
 \begin{rem}
The use of the modified cost $\tilde L$ instead of $L$ guarantees the non-positivity of the running cost which is a useful property in view of the HJB characterization provided in \cite{BPZ16}. %Instead, being $\ell$ non negative by definition, we do not need any modification to characterize the value function of problem \eqref{eq:def_v}.   
Notice that the computation of the function $G(t,x)$ { typically} does not add any difficulty and we can consider it given by the problem data.
 \end{rem}
 
 The following regularity result turns out to be particularly useful for the PDE characterization of $W$.%(and $w$)
 %. We { send the interested reader} to \cite{BPZ16} for the proof. 
 
 \begin{prop}\label{prop:reg}
 The  function $W$ %(resp. $w$) 
 is uniformly continuous in $(x,y,z)$ %(resp. $(x,y,p)$) 
 and satisfies $\lim_{t\to T} W(t,x,y,z) = \max(z,0)$ %(resp. $\lim_{t\to T} w(t,x,y,p) = \max(-p,0)$) 
 uniformly.\\
 Moreover, for any $z\geq 0$ one has $W(t,x,y,z) = z + \inf_{u\in \mathcal U} \mathbb E\left[ \int^T_t \left(\tilde L (X^{t,x}_s,u_s)+ d_k(Y^{t,y,u}_s)\right)\mathrm d s\right] $.
\end{prop}

\begin{proof}
The uniform continuity of $W$ with respect to $(x,y,z)$ is straightforward. Moreover, by the definition of $W$, the Lipschitz continuity of $d_K$ and $\tilde L$ and classical estimates on the processes $(X_\cdot, Y_\cdot)$, one has for any  $h>0$ such that $T-h\in [t,T]$ and $u\in \mathcal U$,
\begin{align*}
&W({T - h},x,y,z) -\max(z,0) \leq \mathbb E\bigg[\max\Big(Z^{T-h,x,z,u,0}_T, 0\Big)+\int^T_{T-h} d_K(Y^{T-h,y,u}_s)\mathrm{d} s\bigg] -\max(z,0)\\
& \leq C \mathbb E\left[|Z^{T-h,x,z,u,0} - z|\right]+ h \left(1+ C\mathbb E\left[ \sup_{s\in [T-h,T]} |Y^{t,y,u}_s|\right]\right) \leq C h (1+|x|+|y|).
\end{align*}
If $z\leq 0$  by the nonegativity of $W$ one immediately has   $W(t,x,y,z) -\max(z,0)\geq 0$. Otherwise if $z>0$, by the nonegativity of $d_K$, the martingale property of stochastic integrals and classical estimates on the process $X$, one has 
\begin{align*}
W(t,x,y,z) -\max(z,0) \geq \inf_{(u,\alpha)\in\mathcal U\times L^2_{\mathbb F}} \mathbb E\left[Z^{T-h,x,z,u,\alpha}_T -z\right]  =  \inf_{u\in\mathcal U} \mathbb E\left[Z^{T-h,x,z,u,0}_T -z\right] \geq - C  h (1+|x|),
\end{align*}
which concludes the proof { of $\lim_{t\to T} W(t,x,y,z) = \max(z,0)$.} 
\\
Let us now assume that $z\geq 0$. Minorating the positive part by its argument and using the martingale property of stochastic integrals  one always has
$$
W(t,x,y,z)\geq z + \inf_{u\in \mathcal U} \mathbb E\left[ \int^T_t \left(\tilde L (X^{t,x}_s,u_s)+ d_k(Y^{t,y,u}_s)\right)\mathrm d s\right].
$$
The reverse inequality is obtained observing that 
$$
W(t,x,y,z) \leq \inf_{u\in \mathcal U}\mathbb E\bigg[\max\Big(z -\int^T_t \tilde L (X^{t,x}_s,u_s) \mathrm{d} s ,0\Big) +\int^T_{t} d_K(Y^{t,y,u}_s)\mathrm{d} s\bigg]
$$
and using the non negativity of the process $z -\int^T_t \tilde L (X^{t,x}_s,u_s) \mathrm{d} s$.\\
%The proof for $w$ follows by similar arguments.
\end{proof}

\begin{cor}\label{cor:opt_contr}
%{\color{blue} \sout{Let  $V(t,x,y)$  be finite and let  $z^*$  realize the supremum %(resp. $\inf$)  
%in \eqref{eq:mainV}. If} } 
Let  $V(t,x,y)$  be finite. Then, there exists   $z^*\leq 0$ that   realizes the supremum %(resp. $\inf$)  
in \eqref{eq:mainV} and if   $(u^*,\alpha^*)$ is an optimal control  for the level set problem \eqref{eq:W} at point $(t,x,y,z^*)$, then $u^*$ is admissible and an optimal control for the original optimal control problem \eqref{eq:def_V} at $(t,x,y)$.% (resp. \eqref{eq:mainv}) 
\end{cor}
\begin{proof}
Being  the supremum in \eqref{eq:mainV} defined on a nonempty (because $V$ is finite) set closed and bounded from above, one can always find $z^*\leq 0$ realizing this supremum. It follows by the proof of Theorem \ref{teo:main} that there exists $u^*\in \Uad$ such that
 $$
 z^* \leq \mathbb E\left[\int^T_t \tilde L(X^{t,x}_s, u^*_s)\mathrm d s\right] = \mathbb E\left[\int^T_t \tilde L(X^{t,x}_s, u^*_s)\mathrm d s\right] - G(t,x). 
 $$
 Recalling that $z^* = V(t,x,y) - G(t,x)$ one can immediately conclude that  $u^*$ is an optimal control for the original problem. %An analogous argument applies to problem \eqref{eq:def_v}.
\end{proof}
  
% From this follows the existence of $\alpha^*\in  L^2_{\mathbb F}$ such that 
 %$$
 %z^* \leq  \int^T_t  \tilde L(X^{t,x}_s, u^*_s)\mathrm d s - \int^T_t (\alpha^*)^\top_s \mathrm d B_s
 %$$
 %and then such that $(u^*,\alpha^*)$ is optimal for $W(t,x,y,z)$ since 
 %$$
 %\mathbb E\left[ \max\left(z^* - \int^T_t  \tilde L(X^{t,x}_s, u^*_s)\mathrm d s + \int^T_t (\alpha^*)^\top_s \mathrm d B_s, 0\right) + \int^T_t d_K(Y^{t,y,u^*}_s) \mathrm d s\right] = 0.
 %$$

In conclusion,  by Theorem \ref{teo:main} and Corollary \ref{cor:opt_contr},  on its domain of definition the optimal value function  $V(t,x,y)$   and the associated optimal strategy is completely determined  by the unconstrained optimal problem \eqref{eq:W}. 

It is possible to show that the value function $W$ is associated to the following HJB equation: 
$$ \left\{ \begin{array}{l}
-W_t  -b(t,x) W_x - \frac{1}{2}\sigma^2(t,x)W_{xx} -d_{K} (y)+ \underset{\substack{u\in U, \alpha\in \R}}\sup \bigg\{ -(\beta_1(t)-u_1) W_{y_1} -(\beta_2(t)+u_1-u_2) W_{y_2}  \\
\qquad  + \tilde L(x,u) W_z  -  \alpha \sigma(t,x)W_{xz}-\frac{1}{2}\alpha^2 W_{zz}\bigg\}=0\\
 \\
 {W(T,x,y,z) = \max(z,0)} \\
 \end{array} \right. $$
We send the interested readers  to \cite[Section 4]{BPZ16} for the complete characterization of $W$ as the unique continuous viscosity solution of this generalized equation. 

%\begin{align*}
%\text{(resp. }\underset{\substack{(u_1,u_2)\in U,\\ (\xi_1,\xi_2)\in B(0,1)}}\sup \bigg\{\xi_1^2\Big(-w_t -b(t,x) w_x -(\beta_1(t)-u_1) w_{y_1} -(\beta_2(t)+u_1-u_2) w_{y_2} + \ell(x,u) w_z \\
% \qquad \qquad- \frac{1}{2}\sigma^2(t,x)w_{xx} -d_{K} (y) \Big) - \xi_1\xi_2 \sigma(t,x)w_{xz}-\frac{1}{2}\xi_2^2 w_{zz}\bigg\}=0).
% \end{align*}
%
%
%

\section{Numerical simulations}\label{sec:tests}

{
In this section we apply the results of  Theorem \ref{teo:main} and Corollary \ref{cor:opt_contr} to numerically approximate  the original value function $V$ solution of \eqref{eq:def_V} and the associated optimal feedback strategy. For details and discussions concerning the numerical approximation of  this particular type of equations, we send the interested readers e.g. to  \cite[Section 9.4]{DebrJako}  and \cite{BCR19}.
}

%\subsection{One dam model}
We focus on the single dam model as described in Section \ref{sec:problem}.
We recall that in this case { the set of admissible controls is defined as}
$$
\Uad = \left\{U\text{-valued progressively measurable processes : }  Y^{t,y,u}_s\in K \text{ for any } s\in [t,T] \text{ a.s.}\right\},
$$
where $U=[0,\bar u]$, $K=[0,\ymax]$ and
$$
Y^{t,y,u}_\cdot = y +\int^\cdot_t (\beta(s) -u_{s})\mathrm{d} s.
$$
The value function $V:[0,T]\times [0, +\infty)\times K \to \R$  is defined by 
$$
V(t,x,y) = \sup_{u\in \Uad} \mathbb E\left[ \int^T_t L(X^{t,x}_s,u_s)\mathrm d s\right]
$$
with $L(x,u)=x u$.
%In what follows we will use different models for the dynamics of the electricity price $X_\cdot$. 
Along the entire section  we take $\ymax =1$, $T=1$ and
$
\beta(t) = 2\sin(\pi t) +0.5.
$

{\color{red}
\subsubsection{Electricity price modeled as a GBM} 
}
In order to better understand the technique we use let us start considering the electricity price which evolves as a geometric Brownian motion (GBM). In particular we consider the same parameters used in Section \ref{sec:HJB}, i.e.
$$
b(t,x) = b x\quad\text{and}\quad \sigma(t,x) = \sigma x. 
$$
%constant price, i.e. $X^{t,x}_\cdot \equiv x\in \R$. 
with $b := 0.05$ and $\sigma := 0.1$. We point out that in this case we simply have $G(t,x) = \bar u \frac{e^{bT} - e^{bt}}{b} x$ for any $(t,x)\in [0,T]\times\R$. 

In order to validate the  approach presented in the previous section we first suppose that assumption (H3) is satisfied. 
The procedure is outlined in Figure \ref{fig:Wconstant}. As here the dependence on $x$ is very simple { (linear in $x$, recalling equation \eqref{linearinx})}, to visualize the results we can freeze the variable $x$ and concentrate ourselves just on the dependence on $t$, $y$ and $z$.
On the top we show  the  function $W(t,x,y,z)$ (left) and its level sets (right) for  $t=0$, $x=5$ and $(y,z)\in \R\times (-\infty, 0]$. Then, using \eqref{eq:mainV}  
we obtain (bottom, left) the value function $V(t,x,y)$ for $t=0, x=5$ and $y\in \R$. Observe that in  the region outside the interval $K=[0,1]$ the level-set function $W$ is positive which means, using again \eqref{eq:mainV}, that in those points $V$ is equal to $-\infty$, represented by the dashed grey line in the figure. This is aligned with the fact that points outside $K$ are not controllable. The resulting value function $V(t,x,y)$ for $t=0$, $x\geq 0$ and $y\in [0,1]$  is finally visible on the bottom-right. In this case, being assumption (H3) satisfied, Figure \ref{fig:Wconstant} (bottom, right) can be compared with the value function which was directly computed by HJB equation in Figure \ref{fig:1dam_price}, and we can notice that the reconstructed value is very near to the value obtained with the HJB equation.

%\begin{figure}[h]
%\includegraphics[width=0.43\textwidth]{GBM_with_W_T1_200_u3.png}
%\includegraphics[width=0.43\textwidth]{V_GBM_u2.png}
%\caption{One dam, price modeled as GBM: the reconstructed value function  under the controllability assumption (H3) with $\bar u=3$ (left) and without the controllability assumption (H3) with $\bar u=2$ (right).}\label{fig:1GBM_with_W} 
%\end{figure}

\begin{figure}
\includegraphics[width=0.43\textwidth]{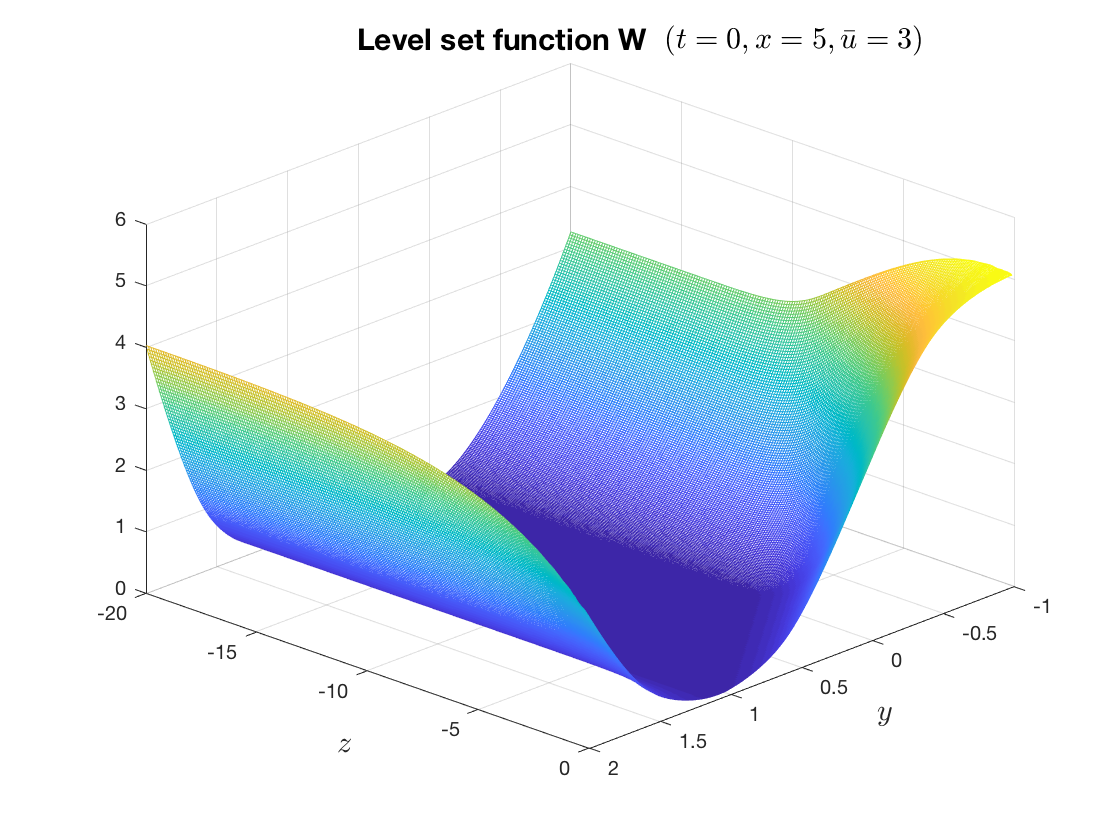} \includegraphics[width=0.43\textwidth]{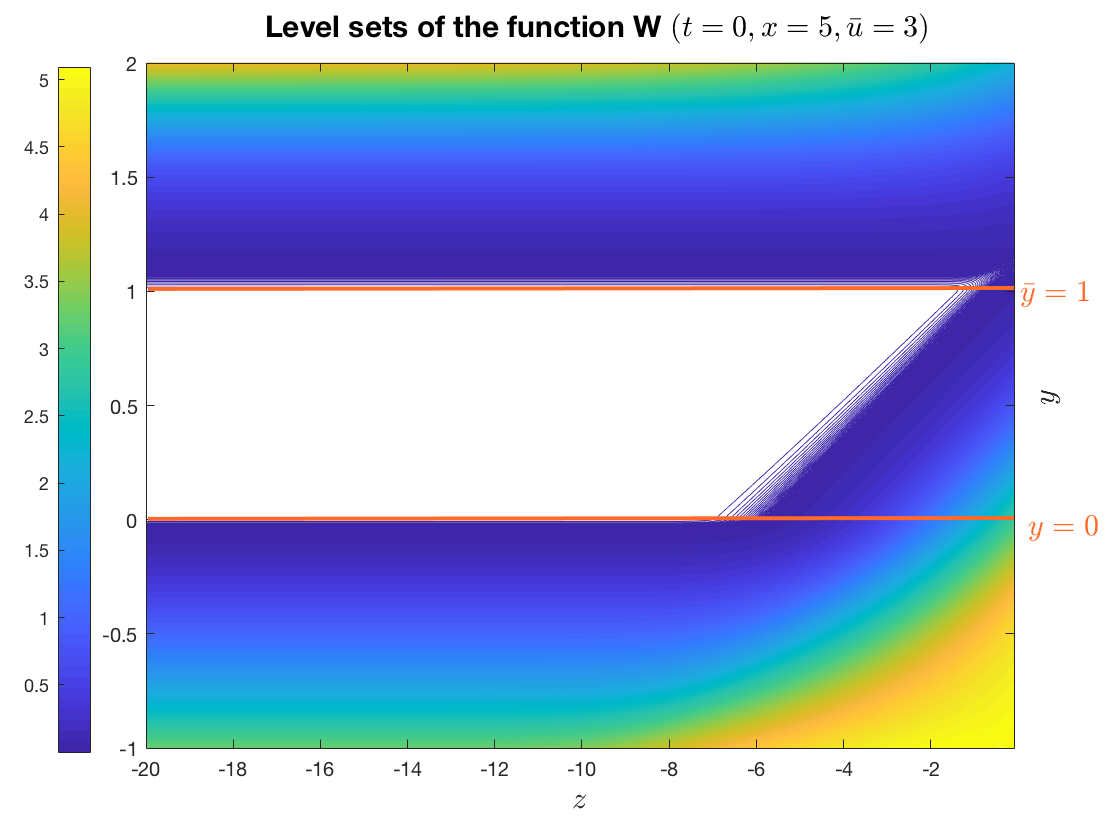} \\\includegraphics[width=0.43\textwidth]{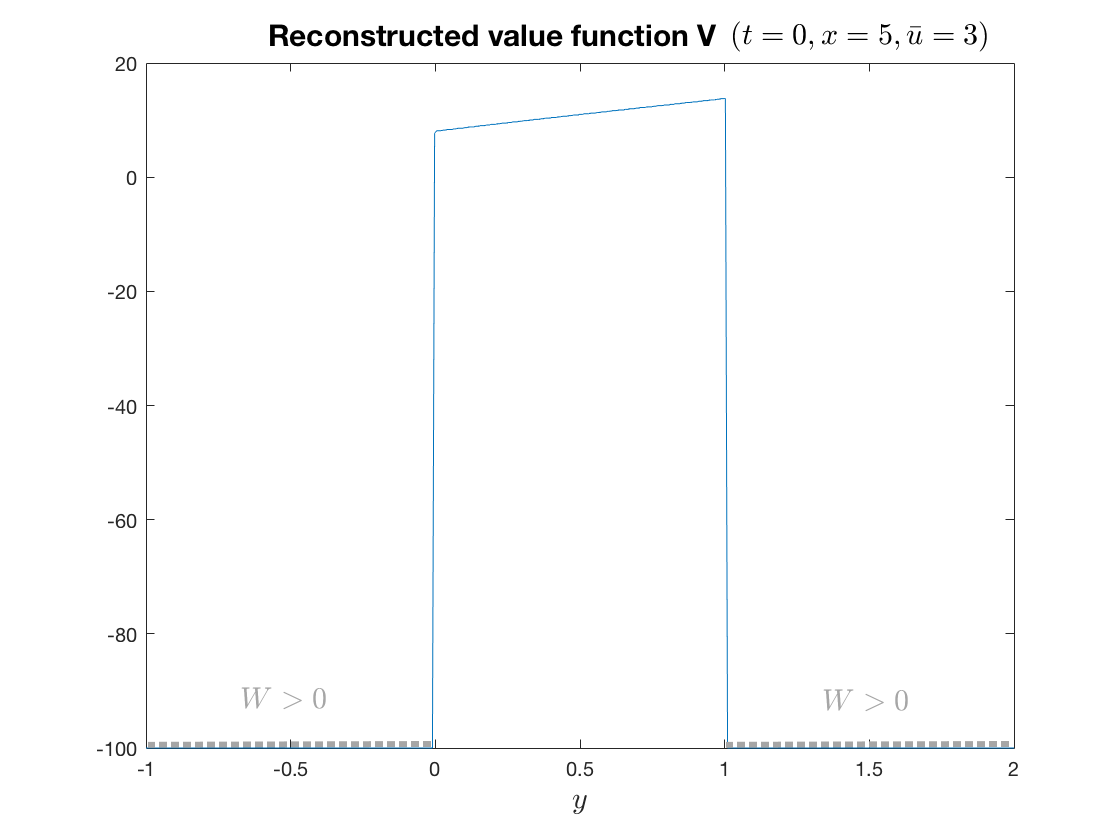}  %\includegraphics[width=0.43\textwidth]{Const_price_with_W}
\includegraphics[width=0.43\textwidth]{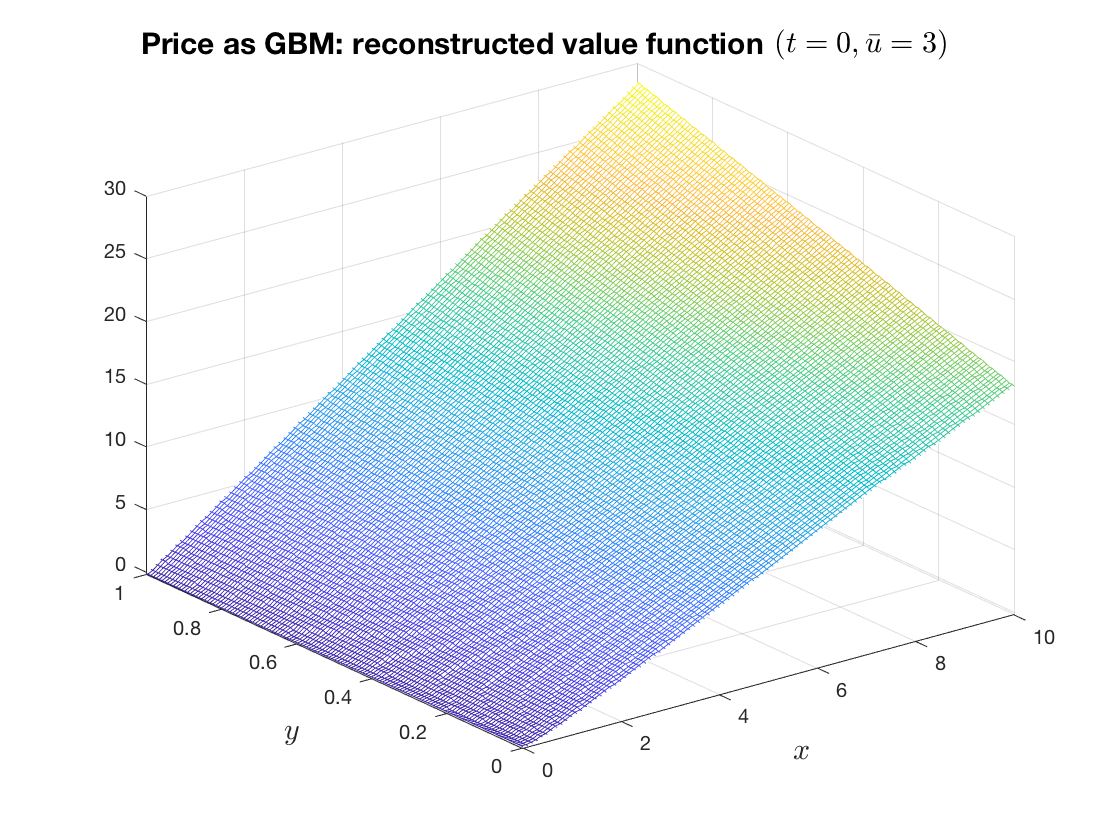}
\caption{Single dam, price modeled as a GBM with $b(t,x) = 0.05 x$ and $\sigma(t,x)=0.1 x$, maximal discharge rate $\bar u=3$ (assumption (H3) is satisfied). Top: the level set function $W$ (left) and its level sets (right) for $x=5$ and $t=0$. Bottom: reconstruction of the value function for $t=0$.}\label{fig:Wconstant}
\end{figure}

We now remove assumption (H3), taking for instance $\bar u=2$. The decomposition $V(t,x,y) = x v(t,y)$ in Equation \eqref{linearinx} holds also in this case, with the consequence that the optimal control is still a feedback control of $(t,y)$ only, but neither $V$ nor $v$ here satisfy an HJB equation (at least with the simple structure of equation \eqref{eq:HJB_constr_price}), and one must reconstruct $V$ via the function $W$ instead.  
We recall that, in this setting, the plot of the controllable region $\{\mathcal D_t, t\in [0,T]\}$ is given in Figure \ref{fig:domain}. We plot the level sets of $W$ and the reconstructed value function at point $x=5$ in the first two lines of Figure \ref{fig:W} on the left and right column, respectively. On the first line we choose $t=0$ and on the second $t=0.3$. We point out that in the second case the finiteness region for $V$, i.e. the region where $W=0$, is reduced to an interval strictly contained in $[0,\bar y]\equiv K$. This corresponds to the finding in Figure \ref{fig:domain}.
The  reconstructed value function is plotted for $t=0$ and $t=0.3$, respectively, at the bottom line of Figure \ref{fig:W}.

\begin{figure}
\includegraphics[width=0.42\textwidth]{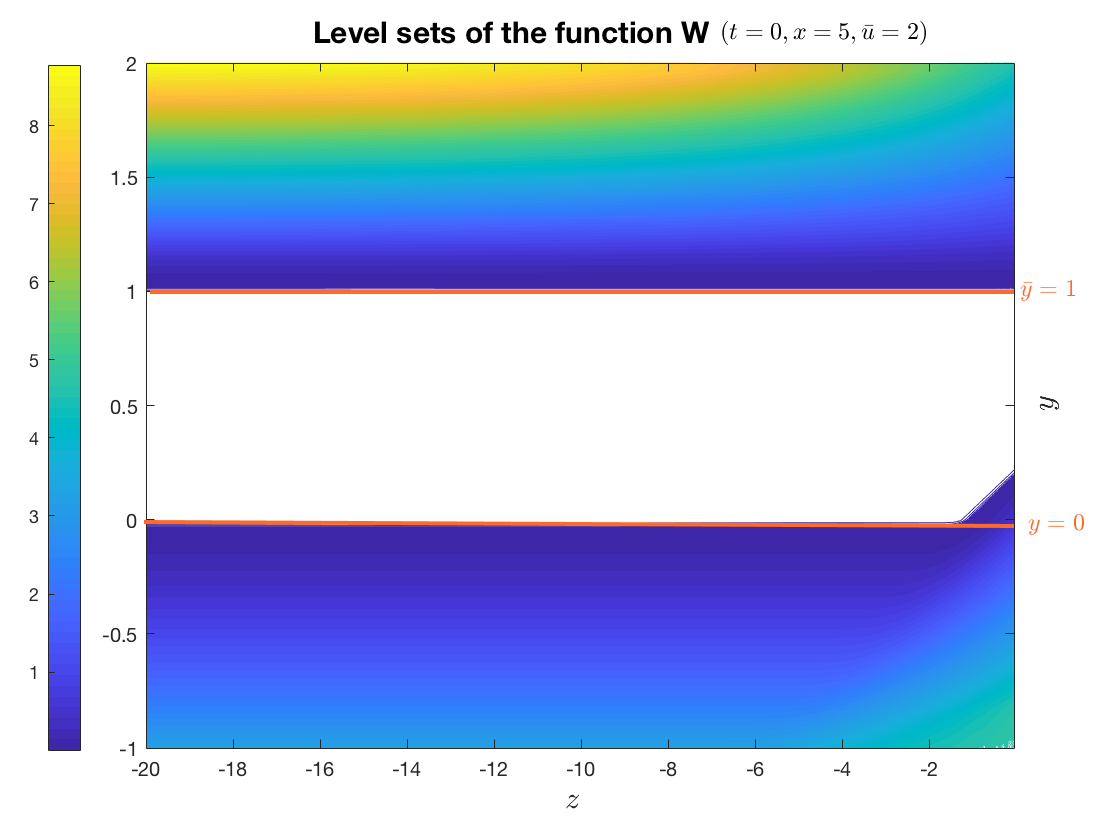}  \includegraphics[width=0.45\textwidth]{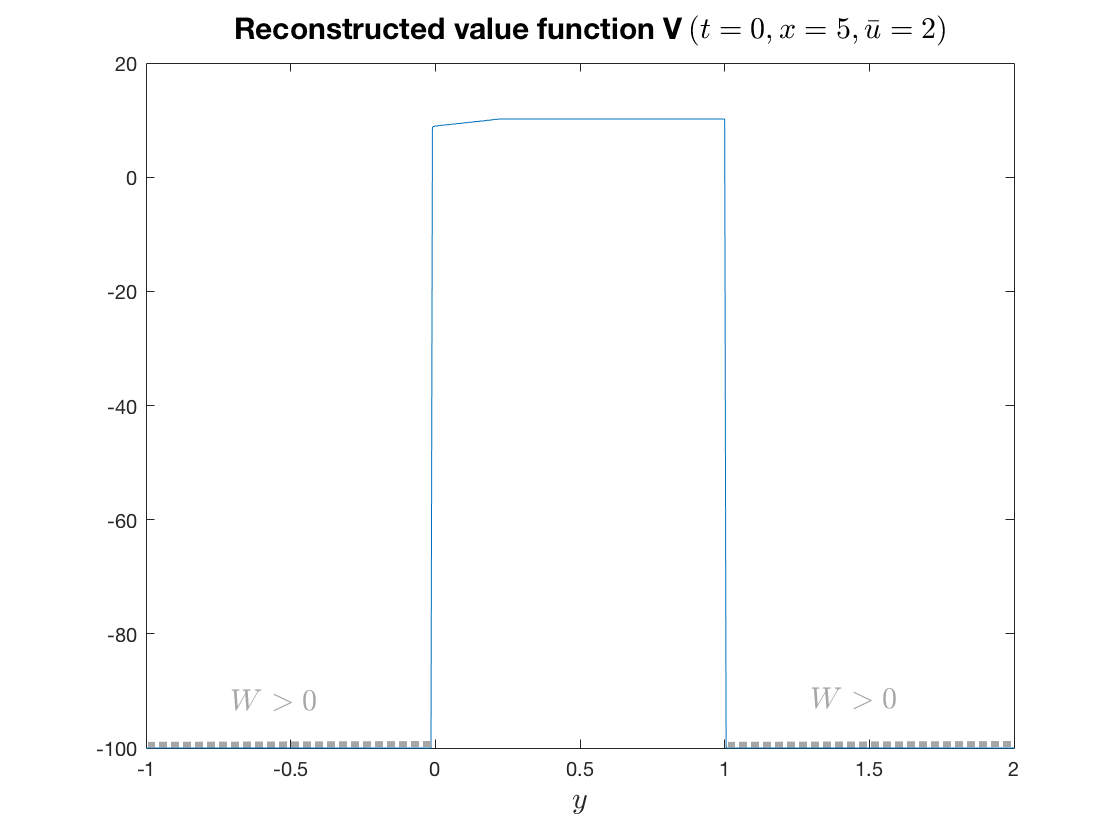}\\
 \includegraphics[width=0.42\textwidth]{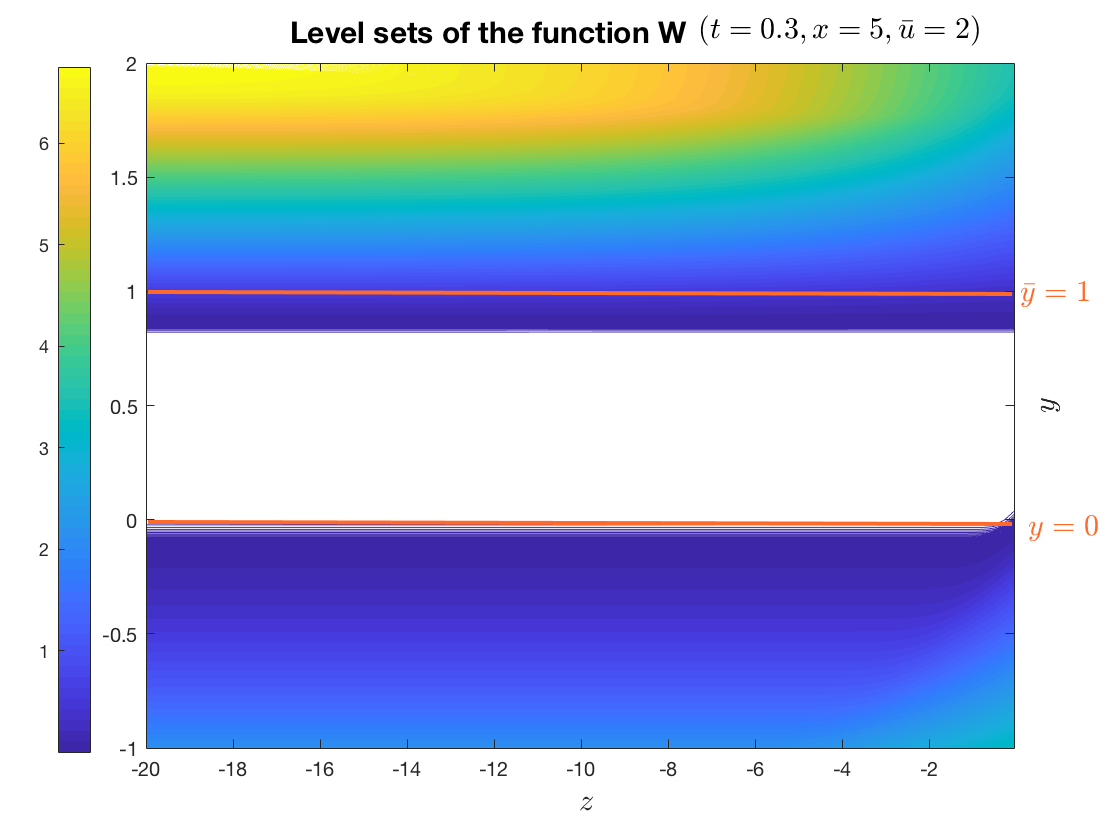}  \includegraphics[width=0.45\textwidth]{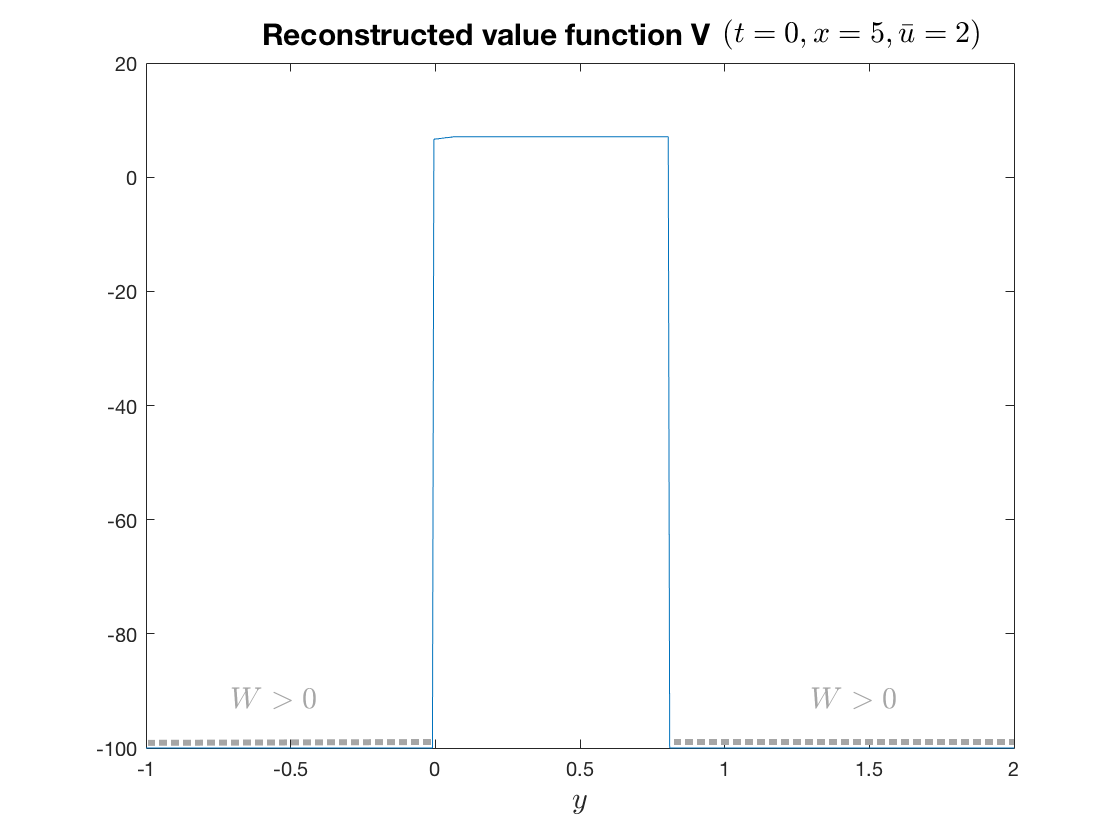}\\
\includegraphics[width=0.43\textwidth]{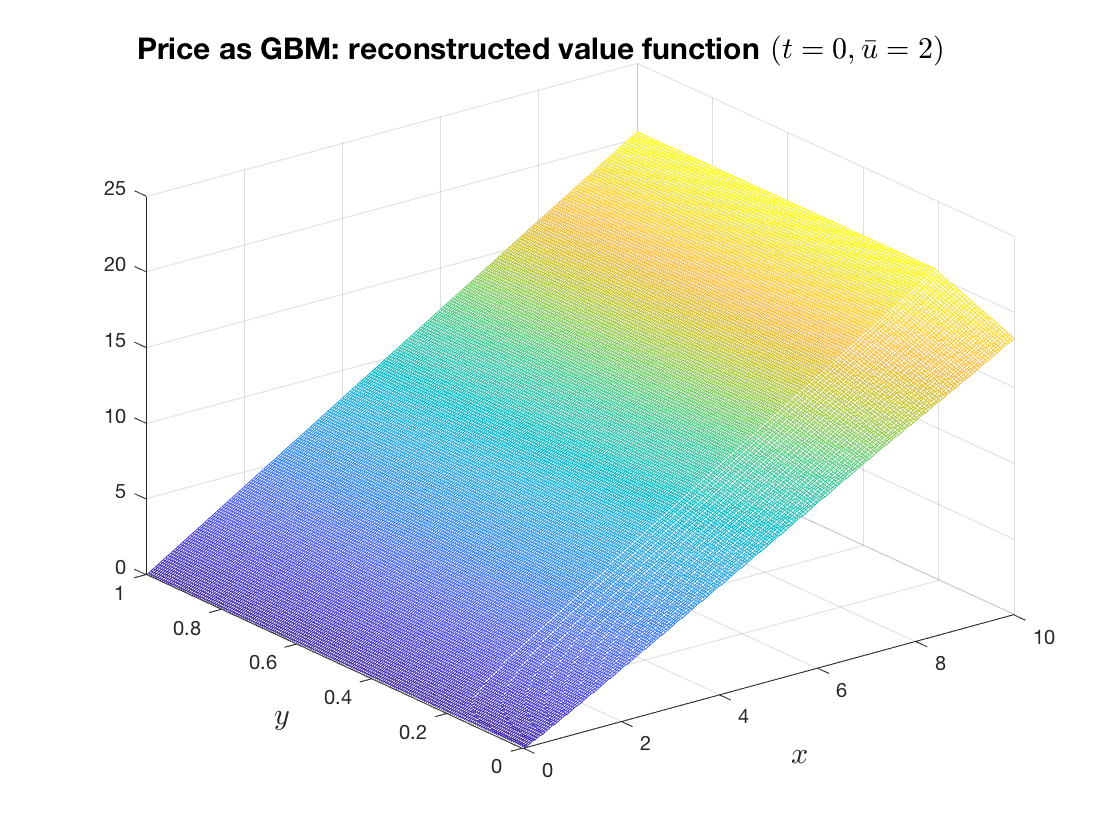}
\includegraphics[width=0.43\textwidth]{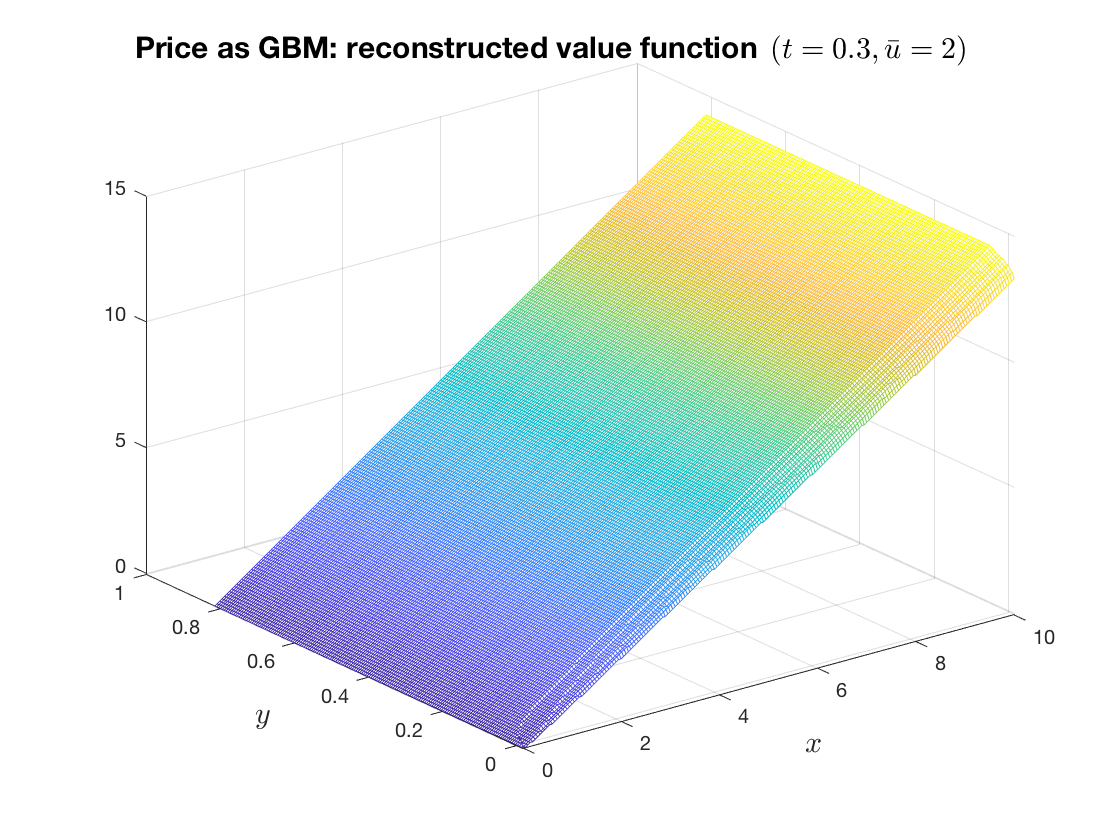}
\caption{Single dam, price modeled as a GBM with $b(t,x) = 0.05 x$ and $\sigma(t,x)=0.1 x$, maximal discharge rate $\bar u=2$ (assumption (H3) is not satisfied).  First and second lines: the level sets of the  function $W$ (left) and the reconstructed value function (right) for $x=5$ at two different time $t=0$ (first line) and $t=0.3$ (second line). Last line: reconstructed value function $V$ for any $x\geq 0$ at different times $t=0$ (left) and $t=0.3$ (right).  }\label{fig:W}
\end{figure}

One can observe that the value function is still increasing with respect to the reservoir level $y$, but the almost linear slope here seems to have a discontinuity for $y \simeq 0.2$: this is maybe due to the fact that here, for levels of $y$ sufficiently high, we lose flexibility in the control $u$, as one must check whether the strategy which was optimal in the previous case now ends up being not admissible: as a result, one could be forced to empty the dam when it would be not optimal. As a consequence, also the ranges of the value functions at $t=0$ are different: the controllable case has a maximum value of almost 30, while in the constrained case it does not surpass 25.

{\color{red}
\subsubsection{Electricity price modeled by  inhomogeneous GBM (IGBM) process}
}
Let us now pass to discuss the case where the electricity price evolves according to an IGBM, obtained by taking
$$
b(t,x) = a - b x\quad\text{and}\quad \sigma(t,x)= \sigma x. 
$$
Considering again the parameters used in Section \ref{fig:1dam_price}, we take $b := 1$, $\sigma := 0.1$ and $a := 5$.
Using the level set method we can approximate the value function and the optimal feedback, in the cases when assumption (H3) holds (again when $\bar u = 3$) and when it does not ($\bar u = 2$). First of all, since $\mathbb E[X^{t,x}_s] = e^{b(s-t)} x + \frac{a}{b} (e^{bs} - e^{b t})$, we have 
$$ G(t,x) = \bar \kappa \left( \frac{1}{b} \left(x + \frac{a}{b}\right) (e^{b(T-t)} - 1) - \frac{a}{b} (T - t) \right) $$ 
The reconstructed value function and optimal feedback are plotted in Figures \ref{fig:1IGBM_u3} and \ref{fig:1IGBM_u2} in the two cases when assumption (H3) holds ($\bar u = 3$) and does not hold ($\bar u = 2$), respectively. 

\begin{figure}[!h]
\includegraphics[width=0.6\textwidth]{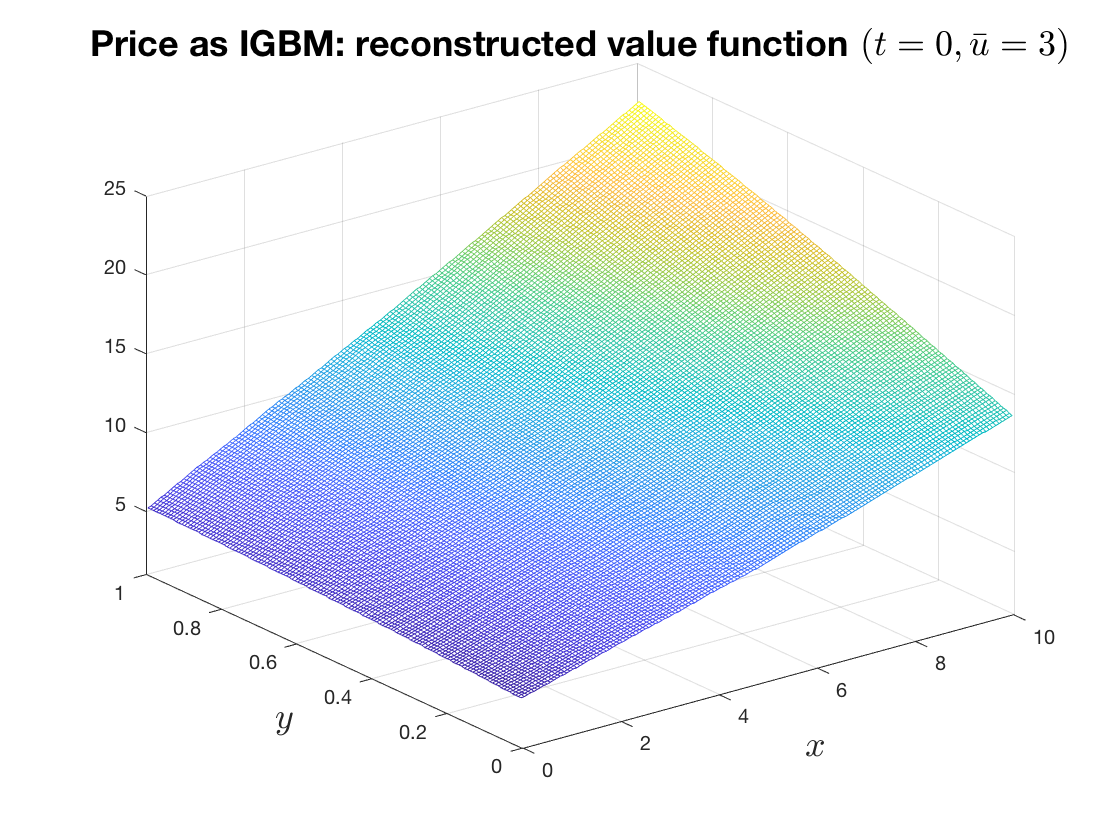} \\
\includegraphics[width=0.24\textwidth]{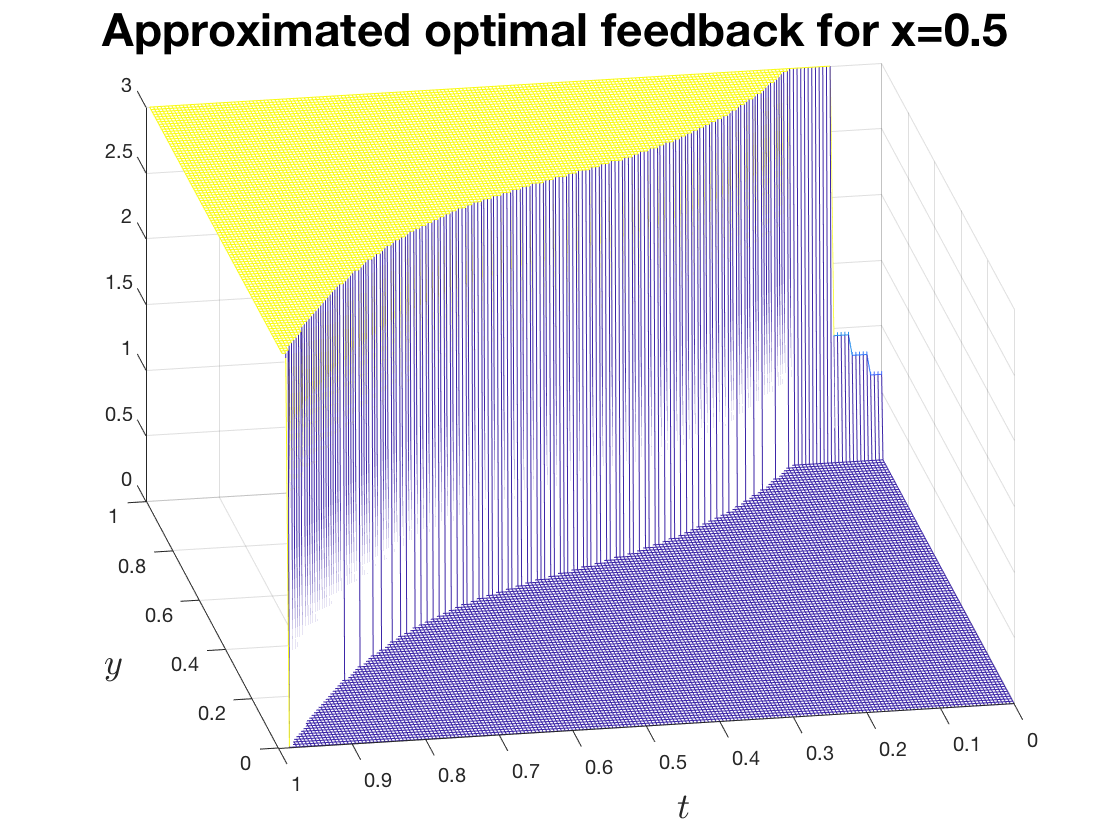}
\includegraphics[width=0.24\textwidth]{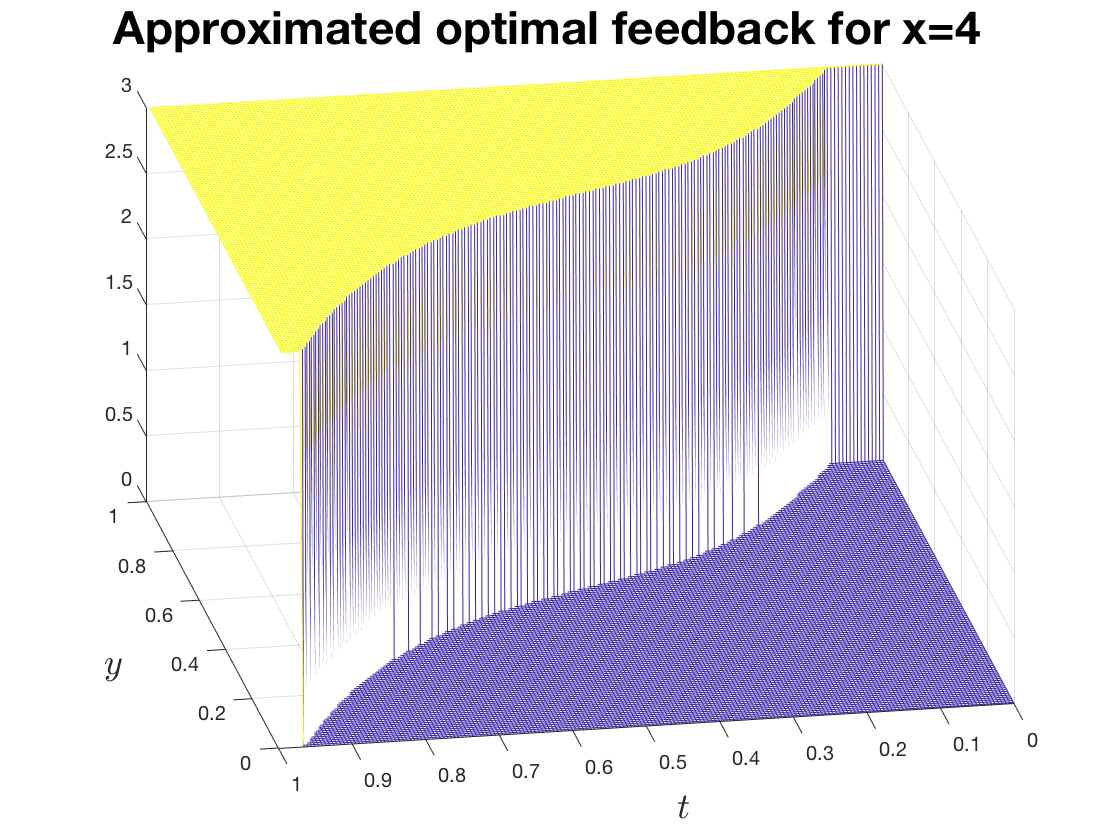}
\includegraphics[width=0.24\textwidth]{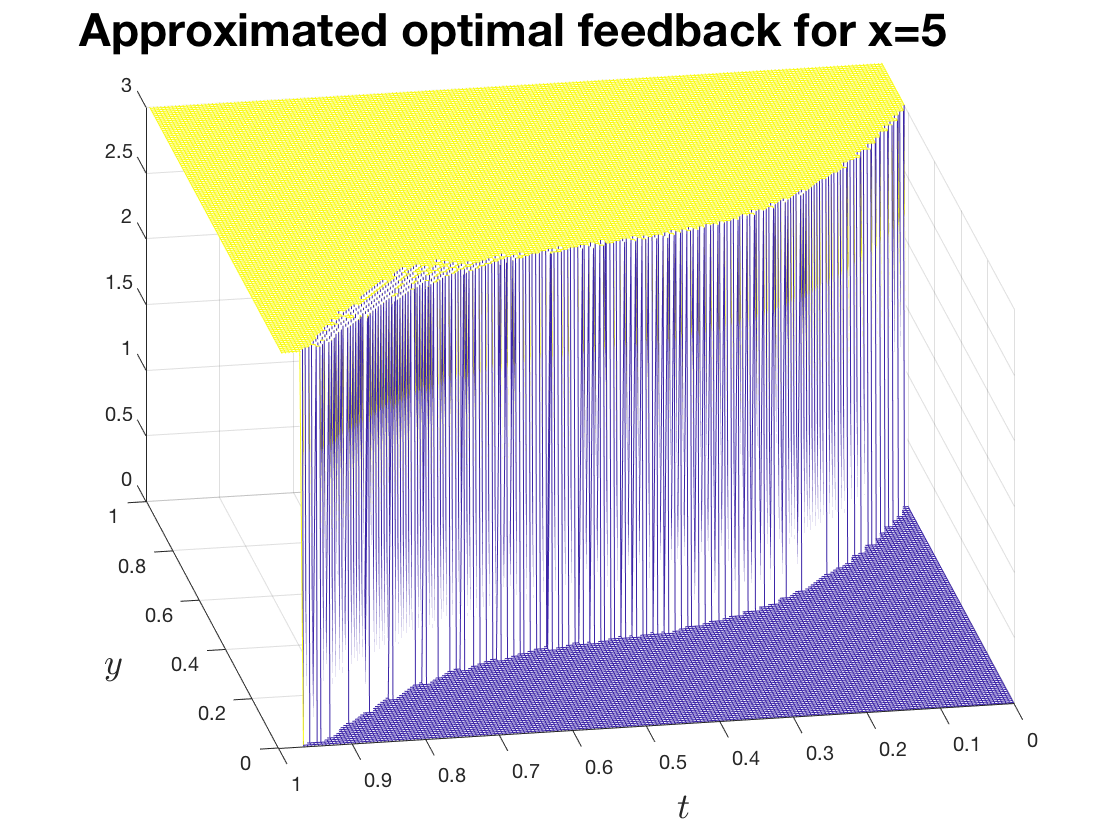}
\includegraphics[width=0.24\textwidth]{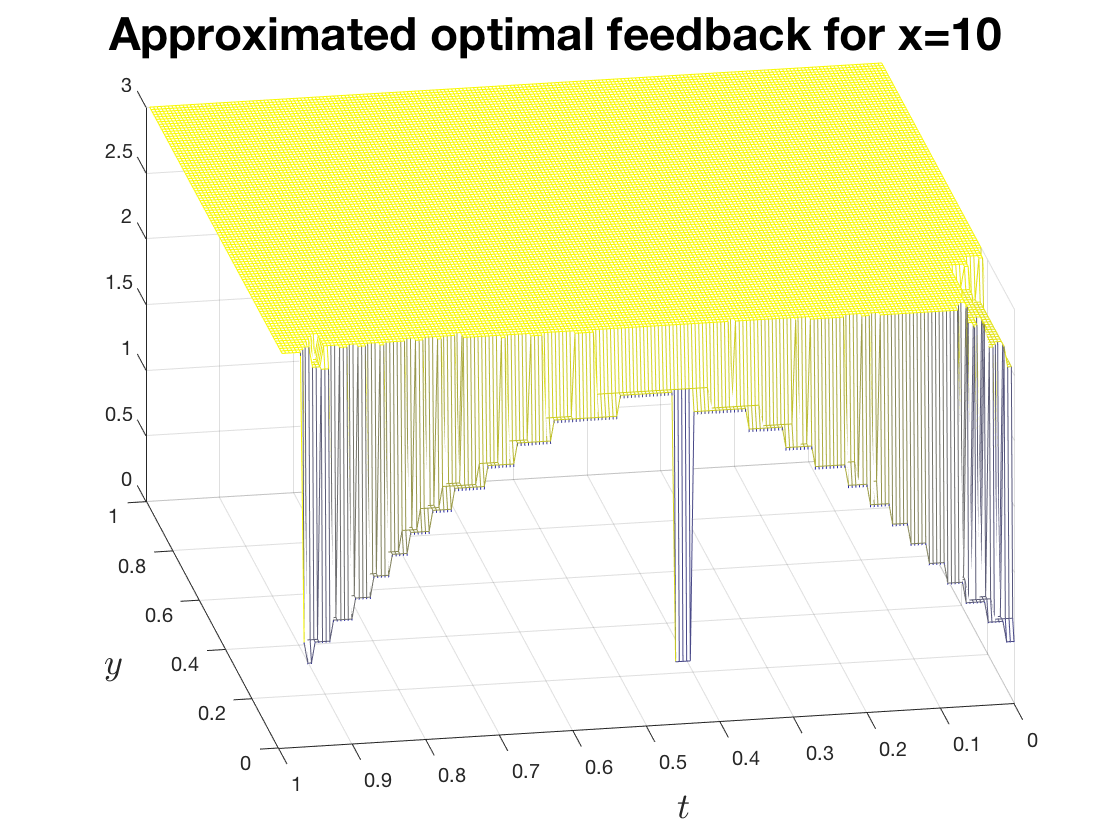}
\caption{One dam, price modeled as IGBM: the reconstructed value function and optimal feedback control with the controllability assumption (H3) with $\bar u=3$.}\label{fig:1IGBM_u3} 
\end{figure}
When $\bar u = 3$, we can see from Figure \ref{fig:1IGBM_u3} (top) that
the reconstructed value function well approximate those obtained by the direct solution of the HJB equation (see Figure \ref{fig:1dam_IGBMprice}), as well as the optimal control (Figure \ref{fig:1IGBM_u3}, three bottom panels). 
\begin{figure}[!h]
\includegraphics[width=0.6\textwidth]{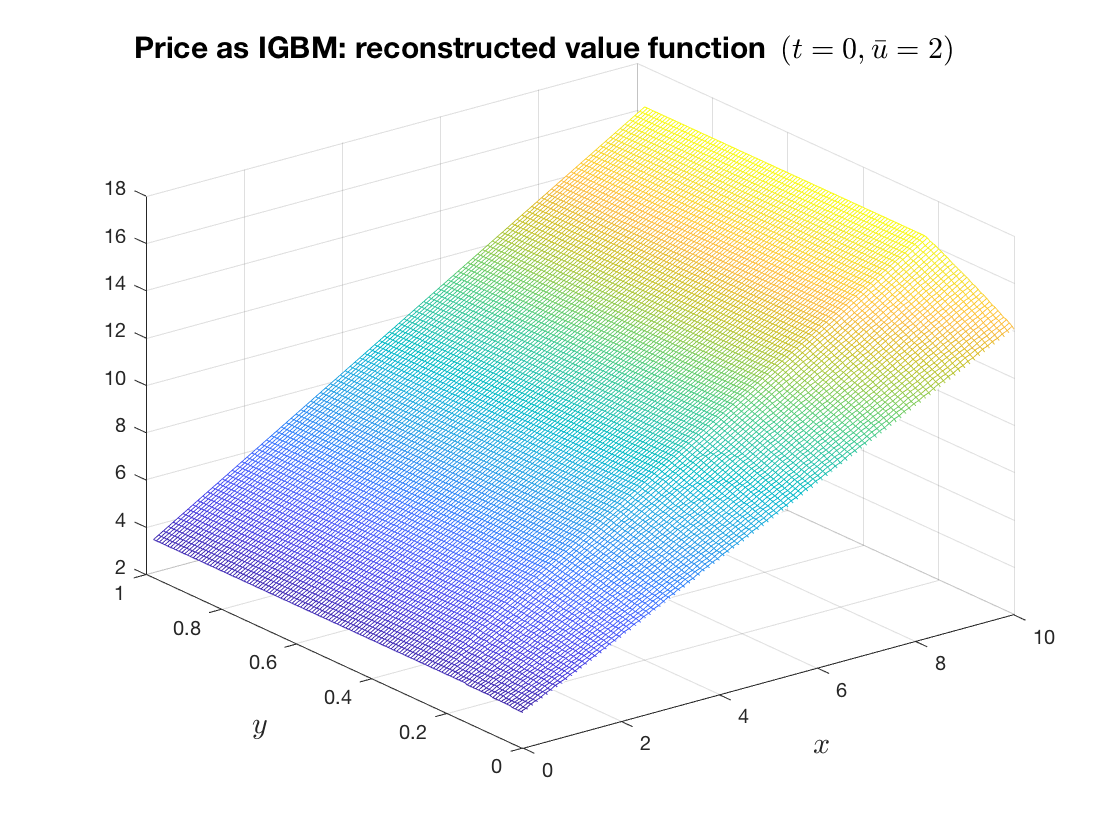}\\
\includegraphics[width=0.24\textwidth]{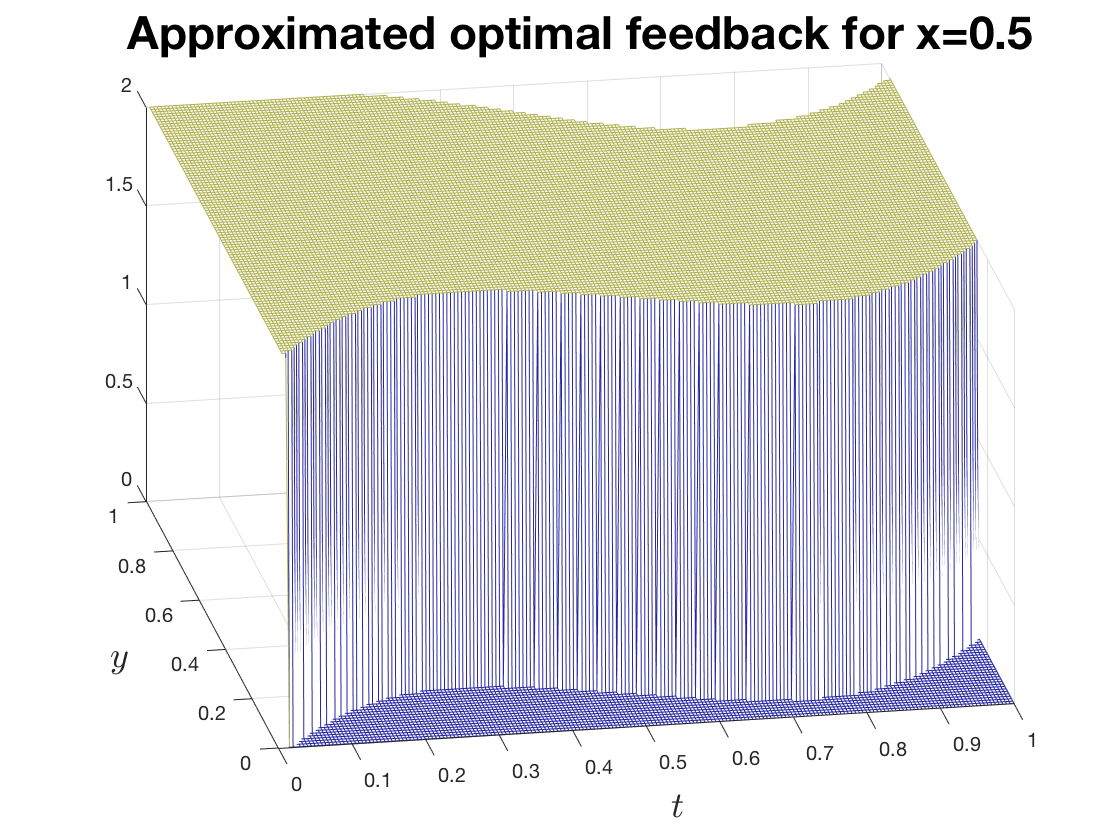}
\includegraphics[width=0.24\textwidth]{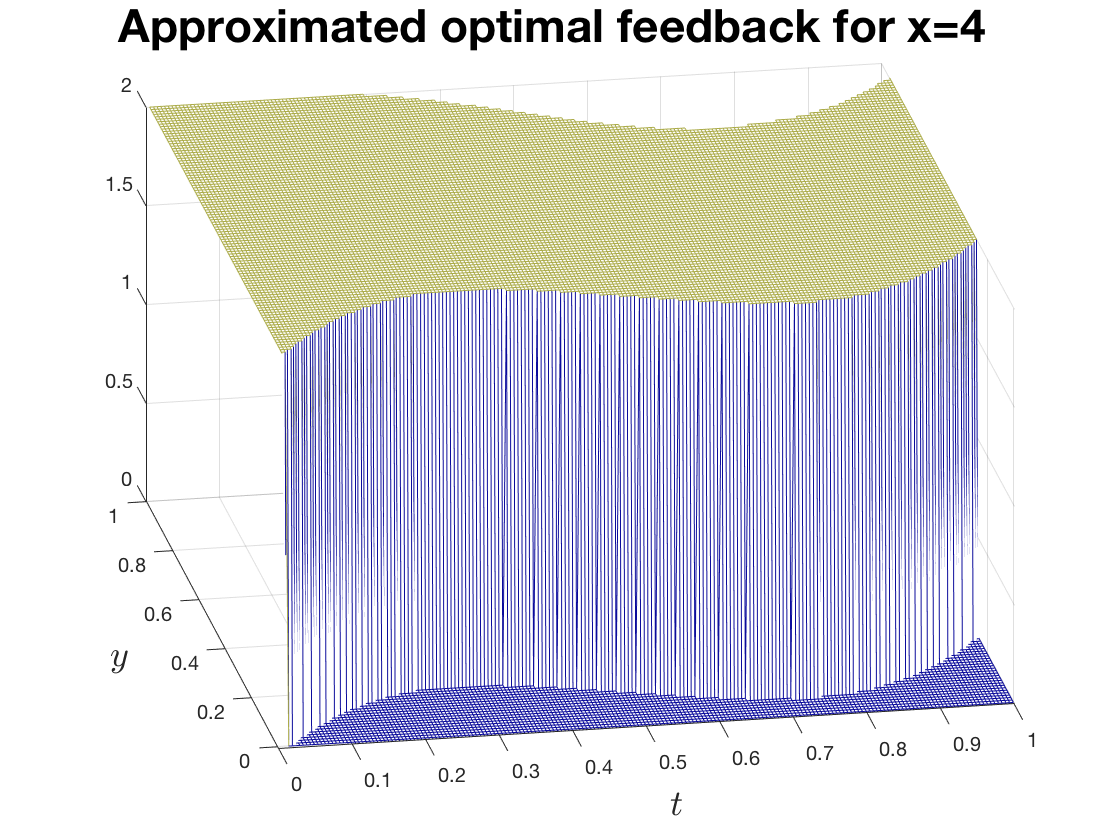}
\includegraphics[width=0.24\textwidth]{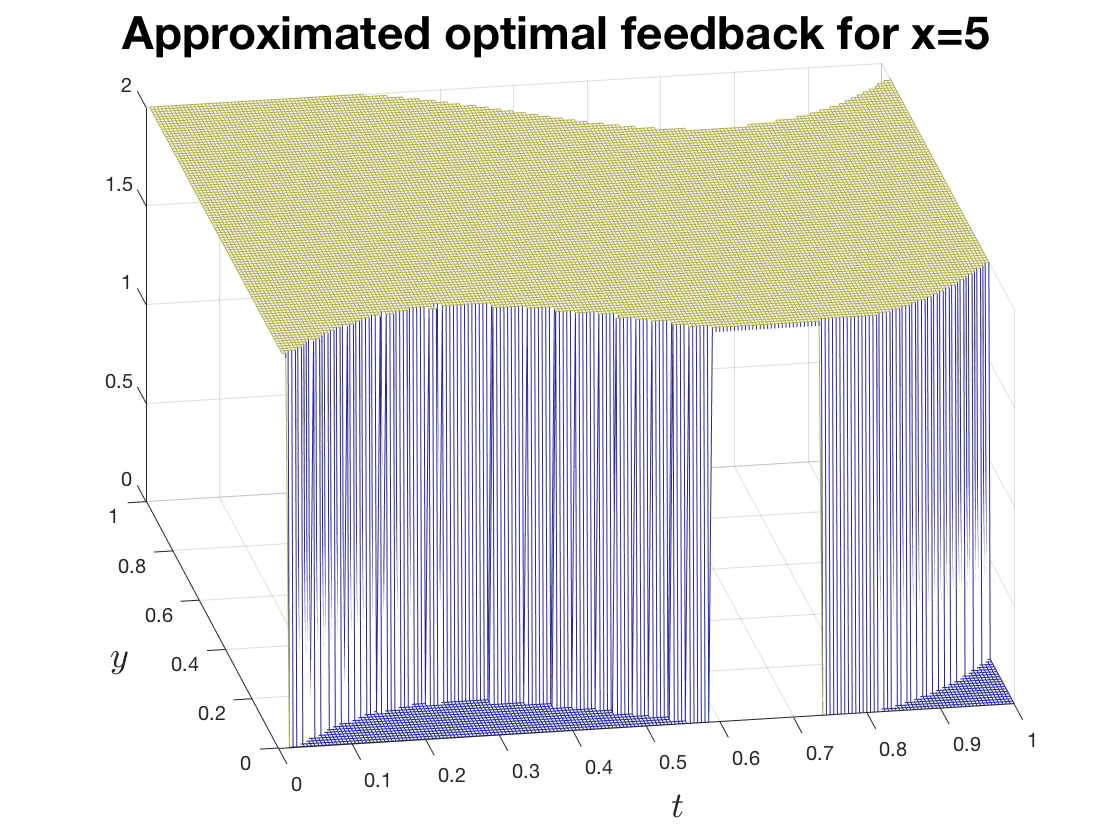}
\includegraphics[width=0.24\textwidth]{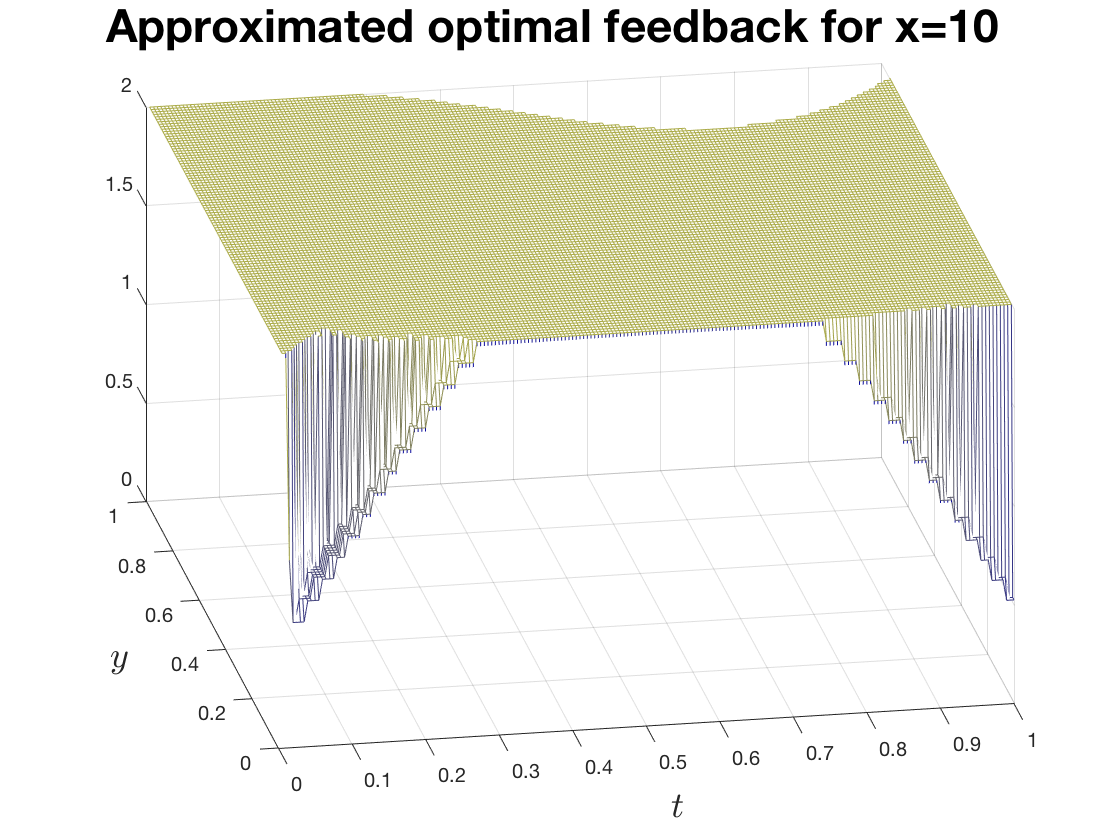}
\caption{Single dam, price modeled as a IGBM with $b(t,x) = 5-x$ and $\sigma(t,x)=0.1 x$, maximal discharge rate $\bar u=2$ (assumption (H3) is not satisfied).}\label{fig:1IGBM_u2} 
\end{figure}
When $\bar u = 2$, Figure \ref{fig:1IGBM_u2} represents the reconstructed value function (top) and the optimal control (four bottom panels). In this case, the value function exhibits a linear growth in $x$ (again here $V(t,0,y) > 0$), while in $y$ we have the same kind of discontinuity in the linear growth that we had with the GBM, i.e. the slope lowers from $y \simeq 0.2$ on. The interpretation here is similar: also here, for levels of $y$ sufficiently high, we lose flexibility in the control $u$, as one must check whether the strategy which was optimal in the previous case is still admissible. 
Indeed, as the four panels at the bottom of  Figure \ref{fig:1IGBM_u2} show for big levels of $y$, one is compelled to discharge the dam even if the price is low. As a consequence, also here the ranges of the value functions are different: the controllable case has a maximum value above 20, while in the uncontrollable case it is slightly greater than 16. The optimal control is also here of bang-bang type and increases with respect to $x$, with the same interpretation as in the controllable case.

\bigskip

\noindent
\textbf{Acknowledgments}.
The authors wish to thank Olivier Bokanowski, Paolo Falbo, Fulvio Fontini, Jacco Thijssen, Ralf Wunderlich and Hasnaa Zidani, for interesting discussions, and all the
participants to the workshop "Investments, Energy, and Green Economy 2019" in Brescia, the XLIII AMASES Conference in Perugia and the workshop  "Control Theory and applications" in L'Aquila.


\begin{thebibliography}{10}

\bibitem{ABZ13}
A.~Altarovici, O.~Bokanowski, and H.~Zidani.
\newblock {A} general {H}amilton-{J}acobi framework for non-linear
  state-constrained control problems.
\newblock {\em ESAIM: Control, Optimisation and Calculus of Variations (COCV)},
  19:337--357, 2013.

\bibitem{BB95}
G.~Barles and J.~Burdeau.
\newblock The {D}irichlet problem for semi-linear second-order degenerate
  elliptic equations and applications to stochastic exit time control problems.
\newblock {\em Communications in Partial Differential Equations},
  20(1-2):129--178, 1995.

\bibitem{BR98}
G.~Barles and E.~Rouy.
\newblock A strong comparison result for the {B}ellman equation arising in
  stochastic exit time control problems and its applications.
\newblock {\em Comm. Partial Differential Equations}, 23(11-12):1995--2033,
  1998.

\bibitem{BCV}
M.~Basei, A.~Cesaroni, and T.~Vargiolu.
\newblock Optimal exercise of swing contracts in energy markets: an integral
  constrained optimal control problem.
\newblock {\em SIAM Journal of Financial Mathematics}, 5(1):581--608, 2014.

\bibitem{BCR19}
C.~Benezet, J.-F. Chassagneux, and C.~Reisinger.
\newblock A numerical scheme for the quantile hedging problem.
\newblock preprint https://arxiv.org/abs/1902.11228.

\bibitem{BCS14}
O.~Bokanowski, Y.~Cheng, and C.~Shu.
\newblock {A} discontinuous {G}alerkin scheme for front propagation with
  obstacles.
\newblock {\em Numerische Mathematik}, 126:1--31, 2014.

\bibitem{BFZ10}
O.~Bokanowski, N.~Forcadel, and H.~Zidani.
\newblock Reachability and minimal times for state constrained nonlinear
  problems without any controllability assumption.
\newblock {\em SIAM Journal on Control and Optimization}, 48(7):4292--4316,
  2010.

\bibitem{BPZ15}
O.~Bokanowski, A.~Picarelli, and H.~Zidani.
\newblock Dynamic programming and error estimates for stochastic control
  problems with maximum cost.
\newblock {\em Applied Mathematics \& Optimization}, 71(1):125--163, 2015.

\bibitem{BPZ16}
O.~Bokanowski, A.~Picarelli, and H.~Zidani.
\newblock State-constrained stochastic optimal control problems via
  reachability approach.
\newblock {\em SIAM Journal on Control and Optimization}, 54(5):2568--2593,
  2016.

\bibitem{BLM}
P.~Brown, J.~P. Lopes, and M.~A. Matos.
\newblock Optimization of pumped storage capacity in an isolated power system
  with large renewable penetration.
\newblock {\em IEEE Transactions on Power Systems}, 23(2):523--531, 2008.

\bibitem{CCGV}
G.~Callegaro, L.~Campi, V.~Giusto, and T.~Vargiolu.
\newblock Utility indifference pricing and hedging for structured contracts in
  energy markets.
\newblock {\em Mathematical Methods of Operations Research}, 85(2):265--303,
  2017.

\bibitem{Capriotti}
L.~Capriotti, Y.~Jang, and G.~Shaimerdenova.
\newblock Approximation methods for {I}nhomogeneous {G}eometric {B}rownian
  {M}otion.
\newblock {\em International Journal of Theoretical and Applied Finance}.

\bibitem{CapDolcLions90}
I.~Capuzzo-Dolcetta and P.-L. Lions.
\newblock Active portfolio management with benchmarking: adding a value-at-risk
  constraint.
\newblock {\em Transactions of the American Mathematical Society},
  318(2):643--683, 1990.

\bibitem{CL90}
I.~Capuzzo-Dolcetta and P.-L. Lions.
\newblock Hamilton-{J}acobi equations with state constraints.
\newblock {\em Transactions of the American Mathematical Society},
  318(2):643--683, 1990.

\bibitem{ForChe07}
Z.~Chen and P.~A. Forsyth.
\newblock A semi-lagrangian approach for natural gas storage valuation and
  optimal operation.
\newblock {\em SIAM Journal on Scientific Computing}, 30(1):339--368, 2007.

\bibitem{ForChe08}
Z.~Chen and P.~A. Forsyth.
\newblock Pricing hydroelectric power plants with/without operational
  restrictions: a stochastic control approach.
\newblock pages 253--281, 2008.
\newblock in Ehrhardt M, Editor. Nonlinear Models in Mathematical Finance, Nova
  Science Publisher.

\bibitem{DebrJako}
K.~Debrabant and E.~R. Jakobsen.
\newblock Semi-{L}agrangian schemes for linear and fully non-linear diffusion
  equations.
\newblock {\em Math. Comp.}, 82(283):1433--1462, 2013.

\bibitem{FGL94}
M.~Falcone, T.~Giorgi, and P.~Loreti.
\newblock Level sets of viscosity solutions: some applications to fronts and
  rendez-vous problems.
\newblock {\em SIAM Journal on Applied Mathematics}, 54(5):1335--1354, 1994.

\bibitem{FelWeb}
B.~Felix and C.~Weber.
\newblock Valuation of multiple hyro reservoir storage systems in competitive
  electricity markets.
\newblock 2014.
\newblock EWL Working Paper No. 01/2014. Available at SSRN:
  https://ssrn.com/abstract=2424674 or http://dx.doi.org/10.2139/ssrn.2424674.

\bibitem{IK96}
H.~Ishii and S.~Koike.
\newblock A new formulation of state constraint problems for first-order
  {PDE}s.
\newblock {\em SIAM Journal on Control and Optimization}, 34(2):554--571, 1996.

\bibitem{IL02}
H.~Ishii and P.~Loreti.
\newblock A class of stochastic optimal control problems with state constraint.
\newblock {\em Indiana University Mathematics Journal}, 51(5):1167--1196, 2002.

\bibitem{K94}
M.~A. Katsoulakis.
\newblock Viscosity solutions of second-order fully nonlinear elliptic
  equations with state constraints.
\newblock {\em Indiana University Mathematics Journal}, 43(2):493--520, 1994.

\bibitem{KV06}
A.~B. Kurzhanski and P.~Varaiya.
\newblock Ellipsoidal techniques for reachability under state constraints.
\newblock {\em SIAM Journal on Control and Optimization}, 45(4):1369--1394,
  2006.

\bibitem{LWM}
N.~L\"ohndorf, D.~Wozabal, and S.~Minner.
\newblock Optimizing trading decisions for hydro storage systems using
  approximate dual dynamic programming.
\newblock {\em Operations Research}, 61(4):810--823, 2013.

\bibitem{ML11}
K.~Margellos and J.~Lygeros.
\newblock {H}amilton-{J}acobi formulation for {r}each-avoid differential games.
\newblock {\em IEEE Transactions on Automatic Control}, 56:1849--1861, 2011.

\bibitem{McDonaldSiegel}
R.~McDonald and D.~Siegel.
\newblock Investment and the valuation of firms when there is an option to shut
  down.
\newblock {\em International Economic Review}, 26(2):331--349, 1985.

\bibitem{DOE}
D.~of~Energy of~the U.S.A.
\newblock {DOE Global Energy Storage Database}.
\newblock {\em {\tt www.energystorageexchange.org}}, 2017.

\bibitem{OS88}
S.~Osher and J.~A. Sethian.
\newblock Fronts propagating with curvature-dependent speed: algorithms based
  on {H}amilton-{J}acobi formulations.
\newblock {\em Journal of computational physics}, 79(1):12--49, 1988.

\bibitem{ShaWun}
A.~A. Shardin and R.~Wunderlich.
\newblock Partially observable stochastic optimal control problems for an
  energy storage.
\newblock 89(1):280--310, 2017.

\bibitem{S86}
H.~M. Soner.
\newblock Optimal control with state-space constraint. {I}.
\newblock {\em SIAM Journal on Control and Optimization}, 24(3):552--561, 1986.

\bibitem{S862}
H.~M. Soner.
\newblock Optimal control with state-space constraint. {II}.
\newblock {\em SIAM Journal on Control and Optimization}, 24(6):1110--1122,
  1986.

\bibitem{ST02b}
H.~M. Soner and N.~Touzi.
\newblock A stochastic representation for the level set equations.
\newblock {\em Communications in Partial Differential Equations},
  27(9-10):2031--2053, 2002.

\bibitem{TDR}
M.~Thompson, M.~Davison, and H.~Rasmussen.
\newblock Optimal control of hydroelectric facility incorporating pump storage.
\newblock {\em Operations Research} 52(4):546--562, 2004.

\bibitem{VMBI}
M.~T. Vespucci, F.~Maggioni, M.~I. Bertocchi, and M.~Innorta.
\newblock A stochastic model for the daily coordination of pumped storage hydro
  plants and wind power plants.
\newblock {\em Ann Oper Res}, 193:91--105, 2012.

\bibitem{ZhaDav}
G.~Zhao and M.~Davison.
\newblock Optimal control of hydroelectric facility incorporating pump storage.
\newblock {\em Renewable Energy}, 34(4):1064--1077, 2009.

\end{thebibliography}
\end{document}